\documentclass[dvipsnames]{article}
\usepackage{graphicx} 
\usepackage{amssymb}
\usepackage{amsmath}
\usepackage{multirow}
\usepackage{longtable}
\usepackage{tikz}
\usetikzlibrary{cd} 
\usepackage{xcolor}
\usepackage{amsmath}
\usetikzlibrary{calc}
\usepackage{float} 
\definecolor{farbe1}{RGB}{184,135,66}   
\definecolor{farbe2}{RGB}{0,128,255}    
\definecolor{farbe3}{RGB}{255,0,0}      
\definecolor{farbe14}{RGB}{0,255,0}      
\definecolor{farbe5}{RGB}{176,131,11}    
\definecolor{farbe11}{RGB}{255,140,0}   
\definecolor{farbe6}{RGB}{0,102,102}    
\definecolor{farbe7}{RGB}{255,0,255}    
\definecolor{farbe8}{RGB}{0,0,128}      
\definecolor{farbe9}{RGB}{128,0,128}    
\definecolor{farbe10}{RGB}{0,0,0} 
\definecolor{farbe12}{RGB}{128,0,0}     
\definecolor{farbe13}{RGB}{0,128,0}     
\definecolor{farbe4}{RGB}{128,128,0}   
\definecolor{farbe15}{RGB}{0,0,255}     
\definecolor{farbe16}{RGB}{75,0,130}    
\definecolor{farbe17}{RGB}{255,105,180} 
\definecolor{gelbgruen}{RGB} {121, 230, 25}
\definecolor{dunkelblau}{RGB}{136, 0, 255}
\definecolor{burgunder}{RGB}{136, 0, 0}
\usepackage{amsthm}
\usepackage{enumitem}
\usepackage{hyperref}
\usepackage[
backend=biber,
style=alphabetic,
sorting=nyt]{biblatex}
\DeclareFieldFormat{postnote}{#1}
\DeclareFieldFormat{multipostnote}{#1}
\newtheorem{theorem}{Theorem}
\newtheorem{lemma}[theorem]{Lemma}
\newtheorem{corollary}[theorem]{Corollary}
\newtheorem{proposition}[theorem]{Proposition}

\theoremstyle{definition}
\newtheorem{definition}[theorem]{Definition}
\newtheorem{example}[theorem]{Example}
\theoremstyle{remark}
\newtheorem{observation}[theorem]{Observation}
\newtheorem{remark}[theorem]{Remark}

\newtheorem{notation}[theorem]{Notation}

\DeclareMathOperator{\CAT}{CAT}

\DeclareMathOperator{\Sym}{Sym}

\DeclareMathOperator{\Aut}{Aut}

\DeclareMathOperator{\id}{id}

\DeclareMathOperator{\FF}{\mathbb F}

\DeclareMathOperator{\KK}{\mathbb K}
\DeclareMathOperator{\SL}{SL}

\DeclareMathOperator{\Ad}{Ad}
\DeclareMathOperator{\Ab}{Ab}
\DeclareMathOperator{\Girth}{Girth}

\DeclareMathOperator{\im}{im}

\DeclareMathOperator{\LS}{LS}
\DeclareMathOperator{\Isom}{Isom}

\title{On Chamber-regular $\tilde C_2$-Lattices}

\author{
    Franziska Stamer\thanks{\href{mailto:franziska.stamer@uni-paderborn.de}{franziska.stamer@uni-paderborn.de}, Universität Paderborn}
    \and
    Thomas Titz Mite\thanks{\href{mailto:titz.mite.math@gmail.de}{titz.mite.math@gmail.de}, Justus-Liebig-Universität Gießen}
}
\date{November 2025}

\begin{document}

\maketitle

\begin{abstract}
    We construct the first examples of chamber-regular lattices on $\tilde C_2$-buildings. Assuming a conjecture of Kantor, our list of examples becomes a classification for type-preserving, chamber-regular $\tilde C_2$-lattices on locally finite $\tilde C_2$-buildings. The links of special vertices in the buildings we construct, are all isomorphic to the unique generalized quadrangle $Q$ of order (3,5). In particular, our constructions involve chamber-regular actions on $Q$. These actions on $Q$ are the first and if Kantor's conjecture holds the only chamber-regular actions on a finite generalized quadrangle and therefore interesting in their own right. Moreover $Q$ is not Moufang and therefore none of our examples are Bruhat-Tits buildings and all our lattices are exotic building lattices.
\end{abstract}

\section{Introduction}

    Euclidean buildings were introduced by Bruhat and Tits as associated geometries to simple algebraic groups $\textbf{G}$ defined over a valued field $\KK$ \cite{BruhatTits1, BruhatTits2}. The buildings that arise from the Bruhat--Tits construction are called Bruhat--Tits buildings and form a proper subclass of the class of Euclidean buildings. By work of Tits and Weiss, an irreducible Euclidean building, that is not Bruhat--Tits, has dimension at most two \cite{Tits_classification,Weiss}, so it is either a tree or of type $\tilde A_2, \tilde C_2$ or $\tilde G_2$. We call an irreducible Euclidean building which is neither a tree nor Bruhat--Tits an \textit{exotic building}.

    We will be interested in groups that act properly discontinously, and cocompactly on Euclidean buildings.
    The prime examples of such actions on Bruhat--Tits buildings are given by $S$-arithmetic groups: if $\Gamma$ is a uniform lattice in $\textbf{G}(\KK)$, then it acts properly discontiniuously and cocompactly on the associated Bruhat--Tits building.
    Inspired by this, we call a group that acts properly discontiniuously, cocompactly and faithfully on a Euclidean building $X$ a \textit{uniform lattice on $X$}. In this article a uniform lattice on an exotic building will just be called \textit{exotic lattice}.
    In the 20th century, several deep results for higher-rank $S$-arithmetic groups were established; for instance, they are known to enjoy Kazhdan's property (T) \cite{Kazhdan} and to be hereditarily just-infinite (after factoring out the center) \cite{MargulisNormalSubgroups}. More recently, it has been shown that exotic lattices also possess these two properties; see \cite{Zuk,Oppenheim} for (T) and \cite{BaderFurmanLecureux,LecureuxWitzel} for just-infiniteness.
    However, it is conjectured \cite{BaderCapraceLecureux, LecureuxWitzel}, that, in contrast to $S$-arithmetic groups, exotic lattices are virtually simple \cite[Conjecture 1.3]{BaderCapraceLecureux}. Moreover, the first examples of simple $\tilde C_2$-lattices have recently been constructed in a very explicit manner \cite{TitzMiteWitzel}. This further motivates the study of exotic lattices, as they constitute a promising source of new examples of infinite simple groups.

    One can construct examples of exotic lattices by considering a finite non-positively curved complex of groups \cite[Chapter III.C]{BridsonHaefliger}. This approach has the advantage that the lattice is given explicitly by a finite presentation.
    A particularly simple form of a complex of groups is a triangle of groups. In this article, we construct exotic chamber-regular $\tilde C_2$-lattices through triangles of groups.
    
    To explain this in more detail, we briefly describe vertex links in $\tilde C_2$-buildings:
    the vertex links in a $\tilde C_2$-building are complete bipartite graphs and generalized quadrangles.
    These are examples of generalized polygons, for which the notion of Moufangness is defined. The Moufang property is a strong symmetry condition and Moufang polygons are classified \cite{TitsWeiss_polygons}. In the case of complete bipartite graphs, the structure is too degenerate for the Moufang condition to impose any meaningful restriction, and every such graph is Moufang. The notion becomes interesting only for more complex examples, such as generalized quadrangles. Non-Moufang quadrangles, however, are less well understood.
    An example of an interesting non-Moufang quadrangle is the unique generalized quadrangle $Q$ of order $(3,5)$; we refer to \cite{DixmierZara} or \cite[6.2.4]{PayneThas} for its uniqueness. Our starting point is the following exceptional phenomenon:
    \begin{proposition}\label{prop:intro_regular_action_on_Q}
        The generalized quadrangle $Q$ admits 11 actions that are regular on chambers.
    \end{proposition}
    We point out that a theorem of Seitz implies that a finite generalized quadrangle admitting a chamber-regular action is necessarily non-Moufang; see \cite{Seitz} or \cite[Theorem 4.8.7]{VanMaldeghem}.
    On the other hand, non-Moufang quadrangles with rich symmetry are rare. It has been conjectured by Kantor \cite{Kantor} that the only finite, non-Moufang quadrangles with chamber-transitive automorphism group are $Q$ and the quadrangle arising from the Lunelli--Sce hyperoval $Q_{\LS}$.
    One can easily check with a computer, that $Q_{\LS}$ does not admit a chamber-regular action \cite{ChamberRegularCode}. So assuming Kantor's conjecture the examples we provide are the only chamber-regular actions on finite generalized quadrangles.

    Using the chamber-regular actions from Proposition \ref{prop:intro_regular_action_on_Q}, we construct chamber-regular actions $(\Gamma,X)$ on $\tilde C_2$-buildings by writing down a triangle of groups involving two of these actions and a chamber-regular action on a complete bipartite graph.
    Recall that a vertex in a $\tilde C_2$-building is called \textit{special} if its vertex link is a generalized quadrangle, and with this terminology every special vertex link in one of our constructed buildings is $Q$. 
    We classify all lattices arising in this way, and vice versa if a lattice acts type-preservingly and chamber-regularly on a building with the described vertex links it must arise from this construction. Note that if Kantor's conjecture holds, our lattices are the only lattices that act type-preservingly and chamber-regularly on a locally finite $\tilde C_2$-building.
    \begin{theorem}[Main Theorem]\label{thm:main}
        There are exactly 3044 type-preserving, chamber-regular lattices (up to isomorphism) on $\tilde C_2$-buildings whose links of special vertices are the generalized quadrangle of order (3,5).
    \end{theorem}
    As exotic lattices, the constructed examples enjoy Kazhdan's property (T) \cite{Oppenheim} and are hereditarily just-infinite \cite{LecureuxWitzel}. Moreover, our approach provides explicit presentations for the lattices, we give an example:

    \begin{example}\label{exp:presentation_lattice}
        Let $\Gamma$ be the group presented by the generators $a,b,c$ subject to the relations $\mathcal R$ indicated below. Then the coset complex with respect to the subgroups $\langle a \rangle, \langle b \rangle$, $\langle c \rangle$ is an exotic $\tilde C_2$-building on which $\Gamma$ acts chamber-regularly. 
        \begin{align*}
            \mathcal R := \bigl \{
            & a^4, \;\;\; b^6, \;\;\; c^6, \;\;\; [b,c], \\
            & a  b^{-1} a^{-1}  b^2  a^{-1}  b  a^{-2}  b,
            && a  b  a^{-2}  b^{-1}  a^{-1}  b  a^{-2}  b^{-1},
            && a  b  a  b^{-1}  a  b^{-2}  a^{-1}  b^2,\\
            & a  c^{-1}  a^{-1}  c^2  a^{-1}  c  a^{-2}  c,
            && a  c  a^{-2}  c^{-1}  a^{-1}  c  a^{-2}  c^{-1},
            && a  c  a  c^{-1}  a  c^{-2}  a^{-1}  c^2  \bigr \}
        \end{align*}
    \end{example}
    
    In general, chamber-transitive lattices on Euclidean buildings are very rare, underscoring the interest in their construction. Under suitable assumptions, such lattices admit a classification, which typically yields a finite list of exceptional examples.
    In \cite{KantorLieblerTits}, lattices acting type-preservingly and chamber-transitively on Bruhat--Tits buildings are classified. In \cite{KirschmerNebe}, all examples from \cite{KantorLieblerTits} are recovered, together with an additional lattice, indicating that the original proof is not complete as stated. An overview of the known chamber-regular lattices on Bruhat--Tits buildings is provided in Table~1 in \cite{KirschmerNebe}.
    Timmesfeld studies chamber-transitive actions with finite chamber stabilizers and Moufang 2-residues \cite{Timmesfeld_large_thickness, Timmesfeld_Rank3}. Of particular relevance to our setting is \cite[Theorem 1]{Timmesfeld_Rank3}, which imposes strong restrictions on the types, thicknesses, and stabilizers of chamber-transitive lattices on two-dimensional Euclidean buildings with Moufang polygons as vertex stabilizers.

    This article is structured as follows.
    In Sections \ref{sec:triangles_of_groups} and \ref{sec:buildings} we introduce triangles of groups and Euclidean buildings and discuss some elementary properties.
    In Section \ref{sec:quadrangle_Q} we define the unique generalized quadrangle $Q$ of order~(3,5) and construct two chamber-regular actions explicitly.
    The data that encodes the 11 chamber-regular action on $Q$ can be found in Appendix \ref{app:presentations}. These actions have been computed with the computer algebra system GAP \cite{GAP4}.
    Using these actions we construct chamber-regular $\tilde C_2$-lattices and classify the results up to group isomorphism in Section \ref{sec:classification}.    
    Appendix \ref{app:more_data} contains some more combinatorial data that is needed for the classification.\\

    \noindent \textbf{Acknowledgments.} The authors thank Stefan Witzel for helpful discussions. They are also grateful to the anonymous referees for their valuable comments and suggestions, which have helped improve this article.

    \noindent \textbf{Computer results}
    Some results in this article were obtained with the assistance of a computer. We provide code for verifying these results in the GitHub repository \cite{ChamberRegularCode}. The code is written in GAP~\cite{GAP4}.

\section{Triangles of groups}\label{sec:triangles_of_groups}
    A triangle of groups is a commutative diagram consisting of seven groups. It is a compact way to encode group actions on simply-connected, 2-dimensional simplicial complexes that are transitive on 2-simplices. Triangles of groups have been studied by Stallings--Gersten \cite{StallingsGersten, Stallings}. In Section \ref{sec:classification} we construct chamber-regular $\tilde C_2$-lattices through triangles of groups. 

    \subsection{Triangles of groups and developements}
    
    \begin{definition}
    Let $A,E_1, E_2, E_3, V_1, V_2, V_3$ be groups. Let $\alpha_i : A \to E_i$ and $\epsilon_{ij} : E_i \to V_j$ be monomorphisms for $1\leq i, j \leq 3, i\neq j$. We call this collection of groups and monorphisms a non-degenerate triangle of groups if the following hold:
    \begin{enumerate}
        \item The diagram in Figure \ref{fig:comm_dia_triangle} commutes,
        \item $\epsilon_{ik}(E_i) \cap \epsilon_{jk}(E_j)=\epsilon_{ik}\circ\alpha_i(A)= \epsilon_{jk}\circ \alpha_j(A) \;\text{for}\; \{i,j, k\}=\{1, 2, 3\}$,
        \item $\langle \epsilon_{ik}(E_i), \epsilon_{jk}(E_j)\rangle = V_k $ and $ \epsilon_{ik}(E_i)
        \neq \epsilon_{ik}(E_i)\cap \epsilon_{jk}(E_j)\neq \epsilon_{jk}(E_j)\\\text{for}\; 
        \{i,j, k\}=\{1, 2, 3\}$.
    \end{enumerate}
    The groups in a triangle of groups $T$ are called local groups. The group $A$ is called face group, the groups $E_1, E_2, E_3$ are called edge groups and the groups $V_1, V_2, V_3$ are called vertex groups.
    \end{definition}

    \begin{figure}
        \[
        \begin{tikzpicture}
        \def\outerRadius{3cm}
        \def\innerRadius{1.5cm}
        \coordinate (V3) at (270:\outerRadius);
        \coordinate (V1) at (30:\outerRadius);
        \coordinate (V2) at (150:\outerRadius);
        \coordinate (E3) at (90:\innerRadius);
        \coordinate (E1) at (210:\innerRadius);
        \coordinate (E2) at (330:\innerRadius);
        \coordinate (A) at (0:0); 
        \node (V3) at (270:\outerRadius){$V_3$};
        \node (V1) at (30:\outerRadius){$V_1$};
        \node (V2) at (150:\outerRadius){$V_2$};
        \node (E3) at (90:\innerRadius){$E_3$};
        \node (E1) at (210:\innerRadius){$E_1$};
        \node (E2) at (330:\innerRadius){$E_2$};
        \node (A) at (0:0){$A$}; 
        \draw[->] (E3) to node[above] {\footnotesize$\epsilon_{32}$} (V2);
        \draw[->] (E3) to node[above] {\footnotesize$\epsilon_{31}$} (V1);
        \draw[->] (E1) to node[midway, sloped, below] {\footnotesize$\epsilon_{13}$} (V3);
        \draw[->] (E1) to node[midway, sloped, below] {\footnotesize$\epsilon_{12}$} (V2);
        \draw[->] (E2) to node[midway, sloped, below] {\footnotesize$\epsilon_{21}$} (V1);
        \draw[->] (E2) to node[midway, sloped, below] {\footnotesize$\epsilon_{23}$} (V3);
        \draw[->] (A) to node[midway, right] {\footnotesize$\alpha_3$} (E3);
        \draw[->] (A) to node[midway, sloped, above] {\footnotesize$\alpha_1$} (E1);
        \draw[->] (A) to node[midway, sloped, above] {\footnotesize$\alpha_2$} (E2);
        \end{tikzpicture}
        \]
        \caption{A triangle of groups.}
        \label{fig:comm_dia_triangle}
    \end{figure}

    \begin{remark}
        A possibly degenerate triangle of groups is a commutative diagram as above not necessarily satisfying Properties 2 and 3. In this article, we only consider non-degenerate triangles of groups and just call them triangles of groups.
    \end{remark}

    \begin{remark}
        A triangle of groups is a special case of a complex of groups as studied in \cite[III.C]{BridsonHaefliger}. In fact, triangles of groups are simple complexes of groups in the sense of \cite[II.12]{BridsonHaefliger}, which are technically less involved than general complexes of groups.
    \end{remark}

    In Figure \ref{fig:comm_dia_triangle}, the edge groups and vertex groups are indexed by $\{1,2,3\}$. This enumeration is arbitrary but it will be convenient to have such an enumeration. This motivates the definition of a type-function.

    \begin{definition}
        Let $T$ be triangle of groups. Then any bijection from the set of vertex groups to $\{1,2,3\}$ is called a type function on $T$. If we have a type function, we can extend it to the set of edge groups by requiring that the edge group with type $i$ embeds in the vertex groups with types $j$ and $k$ for $\{i,j,k\} = \{1,2,3\}$.
    \end{definition}

    For the rest of this section we fix $T$ as the triangle of groups from Figure~\ref{fig:comm_dia_triangle}. We equip it with the type function $f$ indicated by the notation ${f:V_i \mapsto i}, {\; E_i\mapsto i}$.
    
    \begin{definition}
        Let $T'$ be another triangle of groups with local groups $A'$, $(E_i')_{1\leq i \leq 3}$ and $(V_i')_{1\leq i \leq 3}$ and let $f'$ be the type function indicated by the notation.
        \begin{enumerate}
            \item A family of isomorphisms $\phi_A : A \to A', \phi_{E_i} : E_i \to E_i', \phi_{V_i} : V_i \to V_i'$ is called type-preserving isomorphism between $(T, f)$ and $(T', f')$ if $\phi_{E_i}\circ \alpha_i=\alpha'_i \circ \phi_{A}$ and $\phi_{V_j} \circ \epsilon_{ij} = \epsilon_{ij}' \circ \phi_{E_i}$ for $1\leq i, j \leq 3, i\neq j$. If it is clear which type-functions we are referring to, we might omit them in the notation.
            \item
            We call $T$ and $T'$ isomorphic, if they are type-preservingly isomorphic for some type-functions on $T$ and $T'$.
        \end{enumerate}
    \end{definition}

    \begin{remark}
        Note that an isomorphism between triangles of groups is always determined by the isomorphisms between the edge groups.
    \end{remark}

    We now describe the action associated to \textit{developable} triangles of groups.
        
    \begin{definition}
        \begin{enumerate}
            \item We define the fundamental group $\pi_1(T)$ of $T$ as the colimit of $T$ (considered as commutative diagram).
            \item We call $T$ developable if the canonical morphisms from the vertex groups to the fundamental group $\pi_1(T)$ are injective.
        \end{enumerate}
    \end{definition}

    If $T$ is developable, then we call the images of the local groups in $\pi_1(T)$ also just local groups, and we denote them by $A, E_i$ and $V_i$ as well.

    A developable triangle of groups $T$ has an associated geometry on which its fundamental group acts: \textit{the development}. The development is a special case of a coset complex, which we now introduce.

    \begin{definition}
        Let $G$ be a group and let $\mathcal H =(H_0, \dots, H_d)$ be a family consisting of distinct subgroups. The coset complex $\mathcal{CC}(G, \mathcal H)$ is the abstract simplicial complex with vertex set $\cup_i G/H_i$ and maximal simplices $\{gH_0, \dots, g H_d\}$ for $g \in G$. 
    \end{definition}

    \begin{observation}
        The coset complex associated to $G$ and $\mathcal H =(H_0, \dots, H_d)$ is $d$-dimensional and pure. It is connected if and only if the subgroups in $\mathcal H$ generate $G$.
        The group $G$ acts on the coset complex $\mathcal{CC}(G, \mathcal H)$ by left-multiplication and the action is transitive on maximal simplices. Every simplex is of the form $F = \{g H_{i_0}, \dots, g H_{i_k}\}$ for some $J :=\{i_0, \dots, i_k\} \subseteq \{0, \dots, d\}$. The stabilizer of the simplex $F$ is $g(\cap_{j\in J} H_{j}) g^{-1}$. In particular if the groups in $\mathcal H$ have trivial intersection, the action of $G$ is regular on maximal simplices.
    \end{observation}    

    \begin{definition}
        Assume that $T$ is developable.
        We define the abstract development $D^{\Ab}(T)$ as the coset complex associated to $\pi_1(T)$ and the images of the vertex groups.
        We define the development $D(T)$ of $T$ as the geometric realization of $D^{\Ab}(T)$.
    \end{definition}

    The abstract development $D^{\Ab}(T)$ is two-dimensional and pure, and the type-function~$f$ of $T$ induces a type-function on both $D^{\Ab}(T)$ and $D(T)$.
    The canonical action of the fundamental group is type-preserving and transitive on 2-simplices.
    The action has seven orbits on simplices which can be indexed by the coset spaces of the seven local groups.

    The following result is a consequence of Corollary II.12.21 in \cite{BridsonHaefliger}.

    \begin{proposition}
        The development (of a developable triangle of groups) is connected and simply-connected.
    \end{proposition}

    It is important to mention that we can also associate triangles of groups to suitable actions. 

    \begin{definition}\label{def:action_from_triangle}
        Let $X$ be connected, simply connected, pure two-dimensional simplicial complex $X$ such that every vertex link is a connected, bipartite graph without leaves and multiple edges. Assume that $X$ is equipped with a type function and let $\Gamma$ be a group of type-preserving, simplical isometries, that is transitive on 2-simplices.
        Then we say that the action $(\Gamma, X)$ arises from a triangle of groups.
        In that case for any choice of a 2-simplex $t$ we get a triangle of groups $T$ by considering the stabilizers of the faces of $t$. We call $T$ the triangle of groups associated to the action $(\Gamma,X)$ with respect to $t$.
        Different choices of $t$ just yield isomorphic triangles of groups and therefore we will often omit explicit reference to $t$.
    \end{definition}

    This definition is justified by the the following proposition, which is a direct consequence of Proposition II.12.20 (1) in \cite{BridsonHaefliger}.

    \begin{proposition}\label{prop:uniqueness_dev}
        Let $(\Gamma, X)$ be an action arising from a triangle of groups, and let $T$ be a triangle of group associated to it. Then the actions $(\Gamma, X)$ and $(\pi_1(T), D(T))$ are type-preservingly isomorphic.
    \end{proposition}

    \subsection{Non-positive curvature}
    We continue with a sufficient condition for developability, namely non-positive curvature.
    In this subsection $T$ is a triangle of groups with local actions on finite graphs.
         
    \begin{definition}
        If $E\neq F$ are subgroups of a group $V$, then the coset graph of $V$ with respect to $E$ and $F$ is the coset complex $\mathcal C \mathcal C(V, (E,F))$.
        It is a bipartite simplicial graph and $V$ acts on it via left multiplication. The action is type-preserving and transitive on edges.
    \end{definition}

    \begin{definition}
        \begin{enumerate}
            \item The local action of type $i$ in $T$ is the action by the vertex group $V_i$ on its coset graph $\Gamma_i$ with respect to $\epsilon_{ji}(E_j)$ and $\epsilon_{ki}(E_k)$. 
            The angle of type $i$ is defined as $\frac{2\pi}{\Girth(\Gamma_i)}$.
            \item We call $T$ non-positively curved if the sum of its angles is at most $\pi$.
        \end{enumerate}
    \end{definition}

    Note that by non-degenerality of $T$, the graphs in the local actions are connected, have no leaves and have at least two vertices of each type, in particular these graphs have finite girth.

    \begin{proposition}\label{prop:dev_cat0}
        Assume that $T$ is non-positively curved. Then $T$ is developable and the development can be turned into a complete $\CAT(0)$-space by equipping every 2-simplex with the metric of a triangle in $\mathbb E^2$ or $\mathbb H^2$ with the same angles as $T$. If we do so, the fundamental group acts by isometries.
    \end{proposition}

    \begin{proof}
        If the local groups are finite, this is a consequence of \cite[Theorem II.12.28]{BridsonHaefliger}. In the general case, we consider $T$ as complex of groups \cite[Chapter III.C]{BridsonHaefliger}. Let $\mathcal Y$ be the underyling scwol. Its geometric realization $|\mathcal Y|$ is just a topological triangle, and we equip it with the metric as described in the proposition. Then $T$ is non-positively curved in the sense of \cite[Chapter III.C, Definition 4.16]{BridsonHaefliger} and our proposition follows from \cite[Chapter III.C, Theorem 4.17]{BridsonHaefliger} together with the short subsequent discussion.
    \end{proof}

    \begin{remark}
        Assume that $T$ is non-positively curved.
        Let $X$ be the development equipped with the metric from Proposition~\ref{prop:dev_cat0}. Then $X$ is locally compact and its automorphism group $\Aut(X)$, which is the group of simplicial isometries, can be equipped with the compact-open topology.
        If we do so, it becomes a (possibly discrete) locally compact group. Note that $\Gamma \leq \Aut(X)$ might not be closed. However, if in addition the local groups of $T$ are all finite, then $\Gamma$ is a uniform lattice in $\Aut(X)$, i.e. it is a discrete cocompact subgroup.
    \end{remark}
    
    \subsection{Global actions with given local actions}\label{subsec:given_local}
    
    In this subsection we study the number of isomorphism classes of triangles of groups with given local actions. We are mostly interested in triangles of groups with trivial face group, so we fix the following notation.

    For $1 \leq i \leq 3$ let $V_i$ be a group and let $E^i_A, E^i_B$ be subgroups that have trivial intersection and together generate the group $V_i$.
    \begin{definition}
    \begin{enumerate}
        \item
        For $\{i,j,k\} = \{1,2,3\}$ let $f_i : \{j, k\} \to \{A,B\}$ be a bijection. We call $(f_i)_i$ a matching of $((V_i,E^i_A, E^i_B))_i$ if $E^j_{f_j(i)} \cong E^k_{f_j(i)}$ for all $\{i,j,k\} = \{1,2,3\}$.
        If such a matching exists, we call the triple $((V_i,E^i_A, E^i_B))_i$ compatible with respect to the matching $(f_i)_i$.
        \item 
        Assume that the triple $((V_i,E^i_A, E^i_B))_i$ is compatible with respect to the matching $(f_i)_i$.
        Let $T$ be a triangle of groups with vertex groups $(V_i)_i$ and type function indicated by the notation.
        Then we call $T$ a triangle of groups with local action datum $((V_i,E^i_A, E^i_B))_i$ matched along $(f_i)_i$ if the following holds for all $i\neq j$: let $E_i$ be the edge group of type $i$ and denote the monomorphism from $E_i$ to $V_j$ with $\epsilon_{ij}$, then $\epsilon_{ij}(E_i) = E^j_{f_j(i)}$.
    \end{enumerate}
    \end{definition}

    For the rest of this section, we assume that the triple $((V_i,E^i_A, E^i_B))_i$ is compatible with respect to a matching $(f_i)_i$.
    
    We fix abstract groups $E_i$ isomorphic to $E_{f_j(i)}^j$ and $E_{f_k(i)}^k$ and we fix monomorphisms $\kappa_{ij} : E_i \to V_j$, such that $\kappa_{ij}(E_i) = E_{f_j(i)}^j \leq V_j$ for $1 \leq i \neq j \leq 3$. 
    For any choice of $\gamma_{ij} \in \Aut(E_i)$ with $1 \leq i \neq j \leq 3$ let $T((\gamma_{ij})_{ij})$ be the triangle of groups defined in Figure~\ref{fig:T_gamma_ij} and equip it with the type function indicated by the notation. We denote the family consisting of these triangles with $\mathcal T$. Note that any triangle of groups with local action datum $((V_i,E^i_A, E^i_B))_i$ matched along $(f_i)_i$ is canonically isomorphic to a triangle of groups in $\mathcal T$.
    \begin{figure}
        \centering
        \begin{tikzpicture}
            \def\outerRadius{3cm}
            \def\innerRadius{1.5cm}
            \coordinate (V1) at (270:\outerRadius);
            \coordinate (V2) at (30:\outerRadius);
            \coordinate (V3) at (150:\outerRadius);
            \coordinate (E1) at (90:\innerRadius);
            \coordinate (E2) at (210:\innerRadius);
            \coordinate (E3) at (330:\innerRadius);
            \coordinate (A) at (0:0); 
            \coordinate (x) at (0: 3);
            \node (V1) at (270:\outerRadius){$V_3$};
            \node (V2) at (30:\outerRadius){$V_1$};
            \node (V3) at (150:\outerRadius){$V_2$};
            \node (E1) at (90:\innerRadius){$E_3$};
            \node (E2) at (210:\innerRadius){$E_1$};
            \node (E3) at (330:\innerRadius){$E_2$};
            \node (A) at (0:0){$1$}; 
            \draw[->] (E1) to node[above] {\footnotesize $\kappa_{32}\circ \gamma_{32}$} (V3);
            \draw[->] (E1) to node[above] {\footnotesize $\kappa_{31}\circ \gamma_{31}$} (V2);
            \draw[->] (E2) to node[midway, sloped, below] {\footnotesize$\kappa_{12}\circ \gamma_{12}$} (V3);
            \draw[->] (E2) to node[midway, sloped, below] {\footnotesize$\kappa_{13}\circ \gamma_{13}$} (V1);
            \draw[->] (E3) to node[midway, sloped, below] {\footnotesize$\kappa_{23}\circ \gamma_{23}$} (V1);
            \draw[->] (E3) to node[midway, sloped, below] {\footnotesize$\kappa_{21}\circ \gamma_{21}$} (V2);
            \draw[->] (A) to (E1);
            \draw[->] (A) to (E2);
            \draw[->] (A) to (E3);
        \end{tikzpicture}
        \caption{The triangle of groups $T((\gamma_{ij})_{ij})$.}
        \label{fig:T_gamma_ij}
    \end{figure}

    \begin{definition}
        If we have two subgroups $E_j \neq E_k$ of a group $V_i$. We define the local automorphism group $\Sigma(V_i, E_j, E_k)$ of this triple as follows
        \[
            \Sigma(V_i, E_j, E_k) := \{\sigma\in \Aut(V_i)\mid \sigma(E_j) = E_j, \; \sigma(E_k) = E_k \}.
        \]
    \end{definition}

    We denote by $\Sigma^i$ the local automorphism group of $(V_i, E_A^i, E_B^i)$. 
    For $\sigma^j \in \Sigma^j$ let $\sigma^j_i \in \Aut(E_i)$ be the unique automorphism satisfying $\epsilon_{ij} \circ \sigma_i^j = \sigma^j \circ \epsilon_{ij}$. We denote this map by $\Theta^j_i:\sigma^j \mapsto \sigma^j_i$.
    Now we define $\Theta^j : \Sigma^j \to \Aut(E_i) \times \Aut(E_k), \; \sigma_j\mapsto (\sigma_i^j, \sigma_k^j)$ and note that $\Theta^j$ is injective.
    Finally, we define the group $\mathcal I:=\prod_i \Aut(E_i) \times \prod_i \Sigma^i$ and a left action of $\mathcal I$ on the family~$\mathcal T$.

    \begin{definition}
        Let $\phi=(\phi_{E_i})_i\times(\phi_{V_i})_i\in \mathcal I$ and $T := T((\gamma_{ij})_{ij}) \in \mathcal T$. Then we define $\phi.T \in \mathcal T$ as follows.
        \begin{align*}
        \phi . T:&= T\left((\underbrace{\Theta_i^j(\phi_{V_j})}_{\in \Aut(E_i)}\circ \underbrace{\gamma_{ij}}_{\in \Aut(E_i)}\circ \underbrace{\phi_{E_i}^{-1}}_{\in \Aut(E_i)})_{ij}\right)
        \end{align*}
    \end{definition}

    \begin{lemma}
        We have that $T, T'\in \mathcal{T}$ are type-preservinlgy isomorphic if and only if they lie in the same orbit under $\mathcal I$.
    \end{lemma}

    \begin{proof}
        Direct calculations show that if $\phi \in \mathcal I$, then it is a type-preserving isomorphism from $T$ to $\phi.T$. For the converse, assume that $T, T' \in \mathcal T$ are type-preservingly isomorphic via $\phi$. Since $\phi$ is type-preserving, it induces automorphism on the local groups. Moreover, the automorphisms of the vertex groups, stabilize the edge groups. Therefore, $\phi$ encodes an element in $\mathcal I$ and being an isomorphism translates precisely to $T' = \phi.T$.
    \end{proof}
    
    \begin{observation}\label{obs:double_cosets}
    We can identify the family $\mathcal T$ with $C = \prod_i\Aut(E_i)^2$ by
    identifying $T((\gamma_{ij})_{ij})$ with the tuple
    $(\gamma_{12},\allowbreak \gamma_{13},\allowbreak \gamma_{21},\allowbreak \gamma_{23},\allowbreak \gamma_{31},\allowbreak \gamma_{32})$. Then the maps $\iota_E$ and $\iota_V$ defined below are group monomorphisms from $\prod_i \Aut( E_i)$ and $\prod_i \Sigma^i$ to~$C$.
    \begin{align*}
        \iota_E :&
        &\Aut(E_1) \times \Aut(E_2) \times \Aut(E_3) &\to C,\\
        &&(\phi_{E_1}, \phi_{E_2}, \phi_{E_3}) &\mapsto (\phi_{E_1}, \phi_{E_1}, \phi_{E_2}, \phi_{E_2}, \phi_{E_3}, \phi_{E_3}) \hspace{3cm}
    \end{align*}
    \begin{align*}
        \iota_V :&
        & \Sigma^1 \times \Sigma^2 \times \Sigma^3 & \to C,\\
        &&(\phi_{V_1}, \phi_{V_2}, \phi_{V_3}) & \mapsto
        (\Theta_1^2(\phi_{V_2}), \Theta_1^3(\phi_{V_3}), \Theta_2^1(\phi_{V_1}), \Theta_2^3(\phi_{V_3}), \Theta_3^1(\phi_{V_1}),\Theta_3^2(\phi_{V_2}))
    \end{align*}
    Furthermore, $(\sigma_E, \sigma_V) \in \mathcal I$ acts on $T \in C$ via $\iota_E(\sigma_E) T \iota_V(\sigma_V^{-1})$. In particular the type-preserving isomorphism classes in $C$ correspond to double cosets in $C$ with respect to its subgroups $\iota_E \left(\prod_i \Aut( E_i)\right)$ and $\iota_V\left(\prod_i \Sigma^i\right)$.
    \end{observation}

\section{Euclidean buildings of dimension 2}\label{sec:buildings}
    Euclidean buildings are non-positively curved simplicial complexes that are covered by copies of the Euclidean space (flats). The maximal flats are called apartments.
    The gluing of the apartments respects a simplicial structure on the Euclidean space, that is induced by a discrete reflection group.
    Together with symmetric spaces, Euclidean buildings provide the main examples of higher rank $\CAT(0)$-spaces.
    In this article we will focus on Euclidean buildings of dimension two, meaning that the apartments are copies of the Euclidean plane.
    
    A Euclidean building may in general be reducible, meaning it splits as a product of two Euclidean buildings. In this article, we only consider irreducible Euclidean buildings and will omit the term \textit{irreducible} throughout.
    
    For a detailed treatment of (Euclidean) buildings, we refer to \cite{AbramenkoBrown} and \cite{Rousseau}.
    
    \subsection{Definitions}
    If $W$ is a discrete reflection group on the Euclidean space with compact fundamental domain, then the reflection axes induce a cellular structure on the Euclidean space.
    The groups generated by reflections in the Euclidean plane along the sides of a triangle with angles $(\frac \pi 3, \frac \pi 3, \frac \pi 3)$, $(\frac \pi 2, \frac\pi 4, \frac \pi 4)$ or $(\frac \pi 2, \frac \pi 3, \frac \pi 6)$ induce simplicial structures and the resulting simplicial complexes are called model apartments of type $\tilde A_2$, $\tilde C_2$ and $\tilde G_2$ respectively. 

    \begin{definition}
        Let $X$ be a two-dimensional simplicial complex and call every simplicial subcomplex that is isomorphic to one of the model apartments an apartment. Then $X$ is an irreducible Euclidean building of dimension two if the following hold.
        \begin{enumerate}
            \item For every two simplices, there is an apartment containing both.
            \item For every two apartments, there is a simplicial isomorphism between both apartments that is the identity on the intersection.
            \item Every 1-simplex is contained in at least three 2-simplices.
        \end{enumerate}
    \end{definition}

    Note that the first two conditions in the definition assure that we can define the distance of two points in a building by measuring the distance in an enveloping apartment. This metric turns a Euclidean building into a $\CAT(0)$-space. In particular Euclidean buildings are contractible. If the maximal number of triangles that share an edge in $X$ is finite then $X$ is called locally finite. In that case the building is a locally compact space.

    We will be interested in groups that act geometrically on locally compact Euclidean buildings.

    \begin{definition}
        Let $X$ be a locally finite, two-dimensional Euclidean building and let $\Gamma \leq \Aut(X)$. Then $\Gamma$ is called a uniform lattice on $X$ if the action is cocompact and the stabilizers of simplices are finite.
    \end{definition}

    The previous definition originates from the notion of uniform lattices in locally compact groups.
        
    \begin{remark} 
        Given a locally compact group $G$ and a discrete subgroup $\Gamma$, then up to scaling there is a unique left-$G$-invariant regular Borel measure on the quotient $G/\Gamma$. If, with respect to this measure, $G/\Gamma$ has finite volume, then $\Gamma$ is called a lattice in $G$. If the quotient is compact, then $\Gamma$ is a uniform lattice.

        To connect this with isometric actions, let $X$ be a locally finite Euclidean building.
        Then the automorphism group $\Aut(X)$ can be equipped with the compact-open topology turning it into a locally compact group.
        Now, a group $\Gamma \leq \Aut(X)$ is as discrete subgroup if and only if stabilizers of simplices are finite.
        Moreover, if $\Aut(X)$ acts cocompactly on $X$, then we have the following equivalence: a discrete subgroup $\Gamma \leq \Isom(X)$ acts cocompactly on $X$ if and only if $G/\Gamma$ has finite volume.
    \end{remark}

    \subsection{The local approach}

    In the seventies Bruhat and Tits developed a method to associate a Euclidean building to a simple algebraic group over a field with non-Archimedean valuation \cite{BruhatTits1}, \cite{BruhatTits2}. The resulting buildings are called Bruhat-Tits building and are the classical examples of Euclidean buildings.
    Tits and Weiss showed that if a Euclidean building is of dimension at least three, then it is Bruhat-Tits \cite{Tits_classification, Weiss}.
    However, in dimension two, a Euclidean building might not be Bruhat-Tits, and we call these buildings exotic.
    Constructing exotic buildings (preferably with an automorphism group that acts cocompactly) is already a non-trivial task and there is no algebraic recipe to do so. 
    Several constructions of exotic buildings, including the ones in this article, are based on the fact that buildings can be recognized locally \cite{Tits}. We state the two-dimensional version of this in Proposition \ref{prop:local_approach} below.

    \begin{definition}
        A bipartite graph, is called a generalized $m$-gon if its diameter is $m$ and its girth is $2m$. Moreover, we require that every vertex has valency at least three.
    \end{definition}

    Generalized $m$-gons are precisely the spherical buildings of dimension 1. 
    Following standard terminology, we call the edges in a generalized $m$-gon chambers.
    Generalized $m$-gons only exist for $m \in \{2,3,4,6,8\}$ \cite{FeitHigman}. Generalized 2-gons are complete bipartite graphs and generalized 3-gons, 4-gons, 6-gons and 8-gons are called generalized triangles, generalized quadrangles, generalized hexagons and generalized octagons, respectively.

    \begin{proposition}\label{prop:local_approach}
        Let $X$ be a two-dimensional, simply connected simplicial complex.
        Assume there is a type function on $X$ that associates to each vertex one of three types. Assume that the links of vertices of type $i$ are generalized $m_i$-gons for the same $m_i$. Furthermore assume that $\sum_i \frac 1 {m_i} = 1$. Then $X$ is a Euclidean building.
    \end{proposition}

    We apply the proposition to triangles of groups.

    \begin{corollary}
        Let $T$ be a triangle of groups such that its local actions only involve generalized polygons and such that its angle sum is $\pi$. Then the development of $T$ is a Euclidean building of dimension 2.
    \end{corollary}

    \subsection{Action rigidity for building lattices}

    In this subsection we prove action rigidity for uniform lattices on Euclidean buildings. It is a consequence of quasi-isometric rigidity for Euclidean buildings. We will use action rigidity in Section \ref{sec:classification} to classify chamber-regular lattices up to group isomorphism.

    The following is a special case of \cite[Theorem III]{KramerWeiss14}.

    \begin{theorem}\label{thm:qi_rigidity}
        Let $X_1$ and $X_2$ be locally finite Euclidean buildings of dimension at least two. Let $f : X_1 \to X_2$ be a quasi-isometry. Then there
        exists unique isometry $\bar f:X_1 \to X_2$ that is at bounded distance from $f$.
    \end{theorem}

    As consequence we obtain action rigidity for Euclidean building lattices:

    \begin{proposition}\label{prop:action_rigidity}
        Let $\Gamma$ be an abstract group and let $X_1,X_2$ be localy finite Euclidean buildings of dimension at least two.
        Let $\phi_1 : \Gamma \to \Aut(X_1)$ and $\phi_2 : \Gamma \to \Aut(X_2)$ be injective maps, such that the images are uniform lattices on $X_1$ and $X_2$ respectively. Then the actions of $\phi_1(\Gamma)$ and $\phi_2(\Gamma)$ are isomorphic, i.e. there exists an isomorphism $\theta: X_1\to X_2$ such that $\theta \circ \phi_1(g) = \phi_2(g)\circ \theta$ for all $g \in \Gamma$.
    \end{proposition}

    \begin{proof}
        By the \v{S}varc--Milnor lemma, $\Gamma$ is quasi-isometric to both $X_1$ and $X_2$, so $X_1$ and $X_2$ are isomorphic by Theorem \ref{thm:qi_rigidity}. 
        So without loss of generality, we can assume that $X_1=X_2$ and denote the building just by $X$.
        Choose a base point $o \in X$. Define $\tilde \phi_i : \Gamma \to X, g \mapsto \phi_i(g)(o)$ for $i \in \{1,2\}$.
        We slightly perturb the image of $\tilde \phi_2$, such that it becomes injective and $\tilde \phi_2(g)$ stays close to $\phi_2(g)(o)$.
        Let $\kappa_2 : X\to \Gamma$ be a quasi-inverse such that $\kappa_2 \circ \tilde \phi_2 = \id_\Gamma$.
        Then $\tilde \phi_1 \circ \kappa_2 : X \to X$ is a quasi-isometry. By Theorem \ref{thm:qi_rigidity}, there is a unique $\theta\in\Aut(X)$ at finite distance $D$:
        \[
            d(\tilde \phi_1 \circ \kappa_2(x), \theta(x)) < D \quad \text{for all } x\in X.
        \]
        We claim that for all $g \in \Gamma$, we have that $\theta \circ \phi_2(g) = \phi_1(g) \circ\theta$. In particular the lattices $\phi_1(\Gamma)$ and $\phi_2(\Gamma)$ are conjugated in $\Aut(X)$.
        We show that $\theta \circ \phi_2(g)$ and $\phi_1(g) \circ \theta$ are at finite distance and therefore equal by Theorem \ref{thm:qi_rigidity}.
        Since $\phi_2(\Gamma).o$ is cobounded, it suffices to show that for all $h \in \Gamma$ the distances
        \[
            d((\theta \circ \phi_2(g)(\phi_2(h)(o)), (\phi_1(g)\circ \theta)(\phi_2(h)(o)))
        \]
        are uniformly bounded.
        We have 
        \[
            (\theta \circ \phi_2(g))(\phi_2(h)(o)) = \theta \circ \phi_2(gh)(o),
        \]
        and this point is at distance at most $D$ from the following point
        \[
            \tilde \phi_1 \circ \kappa_2 \circ \phi_2(gh)(o)
            = \tilde \phi_1\circ \kappa_2 \circ \tilde \phi_2(gh) 
            = \tilde \phi_1 (gh)
            = \phi_1(gh)(o).
        \]
        On the other hand, the point
        \[
            (\phi_1(g) \circ \theta)(\phi_2(h)(o))
        \]
        is at distance at most $D$ from 
        \begin{align*}
            \phi_1(g) \circ \tilde \phi_1 \circ \kappa_2 \circ \phi_2(h) (o)
            &= \phi_1(g) \circ \tilde \phi_1 \circ \kappa_2 \circ \tilde \phi_2(h)
            &&= \phi_1(g) \circ \tilde \phi_1(h) \\
            &= \phi_1(g) \circ \phi_1(h)(o)
            &&\hspace{0.2mm}= \phi_1(gh)(o).
        \end{align*}
    \end{proof}

    \begin{corollary}\label{cor:group_gives_triangles}
        Let $T, T'$ be developable triangles of groups with finite local groups and whose developements are Euclidean buildings.
        Let $\Gamma, \Gamma'$ be their fundamental groups. Then $T$ and $T'$ are isomorphic if and only if $\Gamma$ and $\Gamma'$ are isomorphic.
    \end{corollary}

    \begin{proof}
        It it clear that isomorphic triangles of groups induce isomorphic actions, in particular they have isomorphic fundamental groups.
        Now assume that the fundamental groups $\Gamma, \Gamma'$ are isomorphic. By the Propsotion \ref{prop:action_rigidity} the action on the corresponding buildings most be isomorphic, so the induced triangles are isomorphic as well.   
    \end{proof}

    \subsection{Recovering the action of a building lattice}

    In this section, we explain how the action of a chamber-regular lattice on a Euclidean building can be reconstructed from the group structure. As usual, we phrase our results in terms of triangles of groups.
    The results obtained in this subsection are not needed in the remainder of this article.

    The following lemma allows to detect vertex and edge stabilizers in fundamental groups by group-theoretic means. 
    While detecting vertex stabilizers is straightforward, detecting edge stabilizers requires additional assumptions.
    \begin{lemma}\label{lem:vertex_stab}
        Let $T$ be a non-positively curved triangle of groups with finite local groups.
        Let $\Gamma$ be the fundamental group and let $X$ be the development.
        Identify simplices of $X$ with cosets in $\Gamma$.
        \begin{enumerate}
            \item The map $xL \mapsto \Ad(x)(L)$ that associates to a vertex its stabilizer is a well-defined bijection from the set of vertices to the set $\mathcal V$ consisting of the maximal finite subgroups in~$\Gamma$.
            \item Assume that at least one of the following holds:
            \begin{enumerate}[label=(\roman*)]
                \item No edge group embeds as a proper subgroup in another edge group.
                \item The development is a Euclidean building.
            \end{enumerate}
            Then the map $xL \mapsto \Ad(x)(L)$ that associates to an edge its stabilizer is a well-defined surjective map from the set of edges to the set $\mathcal E$ which is the set of maximal elements in the following family.
            \[
                \bar{\mathcal E} := \{E \leq \Gamma \mid \text{there exists different $V_1 ,V_2 \in \mathcal V$ both containing $E$} \}.
            \]
        \end{enumerate}
    \end{lemma}

    \begin{proof}
        By Theorem \ref{prop:dev_cat0} the development $X$ can be equipped with a complete $\CAT(0)$-metric. Then by Theorem II.2.8 in \cite{BridsonHaefliger} any finite group in $\Gamma$ has a non-empty convex fix point set.
        But the stabilizers of points are conjugates of the local groups.
        In particular vertex stabilizers are precisely the maximal finite subgroups.
        We have shown that the first map is well-defined and surjective. To see that it is injective assume that two vertices have the same stabilizer. Then this stabilizer also fixes the geodesic between the vertices but this geodesic contains points which are not vertices. Stabilizers of such points are non-maximal finite subgroups, so this yields a contradiction.

        Clearly assigning to an edge its stabilizer is a well-defined map from the set of edges to $\bar {\mathcal E}$.
        Therefore, we need to show that every maximal element in $\bar {\mathcal E}$ is an edge stabilizer and that every edge stabilizer is a maximal element in $\bar {\mathcal E}$.
        Let $E \in \bar {\mathcal E}$, i.e. a subgroup contained in different vertex stabilizers.
        Then $E$ fixes a geodesic between the corresponding vertices, so it is contained in an edge stabilizer.
        This shows that the maximal elements in $\bar{\mathcal E}$ are all edge stabilizers.
        It remains to show that no edge stabilizer is contained in another edge stabilizer.
        If assumption (i) is satisfied this clearly cannot happen. So we assume that the development $X$ is a Euclidean building.
        If $X$ is of type $\tilde A_2$ then the edge groups are all of the same order, so assumption (i) is fulfilled.
        Now we assume that $X$ is of type $\tilde C_2$ or $\tilde G_2$.
        We make an observation about the fixed point set of an edge stabilizer:
        since interior points of chambers have a strictly smaller stabilizer than edges, the fixed point set of any edge stabilizer must be a tree supported on vertices and edges in the building $X$.
        Since the tree is convex in $X$, the angle between any pair of edges issuing from a vertex is $\pi$.
        Now assume there exist edge stabilizers $E \subsetneq E'$.
        Then the fixed tree of $E'$ is a proper subtree of the fixed tree of $E$. In particular there exist a vertex $x$ and two edges $e,e'$ both incident with $x$ such that $E$ and $E'$ are the stabilizers of $e$ and $e'$.
        In the building $X$, the edges $e$ and $e'$ must be of different type, otherwise their stabilizers would have the same order. If $X$ is of type $\tilde C_2$ this cannot happen, because edges of different types cannot form an angle of size $\pi$. 
        If $X$ is of type $\tilde G_2$ edges of different type can only form an angle of size $\pi$ if they intersect at a vertex whose link is an $A_2$-building. In that case the stabilizer of these edges must be of the same size.
    \end{proof}

    \begin{proposition}\label{prop:recover}
        Let $T$ be a developable triangle of groups with finite local groups, trivial face group and such that the development $X$ is a Euclidean building.
        Let $V_1, V_2, V_3$ be the vertex groups and let $\Gamma$ be the fundamental group.
        Let $\mathcal V$ and $\mathcal E$ be as in Lemma \ref{lem:vertex_stab}.     
        Then any 3-set $\{V_1',V_2',V_3'\} \subseteq \mathcal V$ satisfying the properties (i) to (iv) below is a conjugate of $\{V_1, V_2, V_3\}$. 
        \begin{enumerate}[label=(\roman*)]
            \item The groups $V_1', V_2', V_3'$ are pairwise not conjugate.
            \item The groups $V_i',V_j'$ intersect in an element $E_k' \in \mathcal E$ with $\{i,j,k\} = \{1,2,3\}$.
            \item The groups $E_i',E_j', E_k'$ are pairwise distinct.
            \item The groups $V_1',V_2', V_3'$ generate $\Gamma$.
        \end{enumerate}
        In particular, given $\Gamma$ as abstract group, one recovers the action on the building by constructing a triangle of groups out of any 3-set as above.
    \end{proposition}
    
    \begin{proof}        
        Let $\{V_1', V_2', V_3'\} \subseteq \mathcal V$ be a 3-set satisfying (i) - (iii).
        There are three vertices $x_1, x_2, x_3 \in X$ stabilized by $V_1', V_2', V_3'$ respectively.
        Since the groups $V_1', V_2', V_3'$ are pairwise not conjugate, the vertices $x_1, x_2, x_3$ are necessarily of three different types in the building $X$.
        Let $\{i,j,k\} = \{1,2,3\}$ and let $\gamma_k$ be the geodesic between $x_i$ and $x_j$.
        These geodesics $\gamma_1, \gamma_2, \gamma_3$ must be edge paths stabilized by the groups $E_1', E_2', E_3'$.
        The geodesics $\gamma_i, \gamma_j$ only intersect in the vertex stabilized $V_k'$, because otherwise there would be at least two vertices fixed by $\langle E_i', E_j'\rangle$ and since $E_i' \neq E_j'$, neither $E_i$ nor $E_j$ would be maximal among the groups that are contained in two vertex stabilizers.
        In particular none of the angles in the geodesic triangle $\Delta := \Delta(x_1, x_2, x_3)$ is 0.
        On the other hand $\Delta$ is a geodesic triangle in a Euclidean building whose sides are edge paths. In particular if none of the angles is 0, their sum must be $\pi$. Now we can apply the Flat Triangle Lemma \cite[Proposition II.2.9]{BridsonHaefliger} and deduce that $\Delta$ is isometric to its Euclidean comparison triangle. By \cite[Theorem VI.7.2]{Brown} the triangle $\Delta$ is contained in an apartment.

        Define a triangle of groups $T'$ as follows: its vertex groups and edge groups are $(V_i')_i$ and $(E_i')_i$, the face group is trivial, and the monomorphims in $T'$ are just the restricted identity maps.
        The local actions in $T'$ are the same local actions as in $T$. In particular, $T'$ induces a chamber-regular action on a Euclidean building by Proposition \ref{prop:local_approach}.
        Note that, the notation $V_i$ arises in two different contexts: as local groups in $\Gamma':=\pi_1(T')$ and as subgroups of $\Gamma$. To avoid ambiguity we denote the local groups in $\Gamma'$ by $\bar V_i'$.
        Let $\Phi$ be the canonical homomorphism from $\Gamma'$ to $\Gamma$.
        Since torsion in $\Gamma$ is bounded, $\Phi$ has infinite image and by the normal subgroup theorem for Euclidean building lattices \cite{LecureuxWitzel}, $\Phi$ is injective.
        Now assume that condition (iv) holds. Then $\Phi$ is an isomorphism and in particular $\Phi(\Gamma')$ acts as a uniform lattice on $X$. By Proposition \ref{prop:action_rigidity}, the actions $(\Gamma', D(T'))$ and $(\Phi(\Gamma), X)$ are isomorphic. A $\Phi$-equivariant isomorphism from $D(T')$ to $X$ will take the base chamber to a chamber in $X$. The groups $\bar V_i$ which are the stabilizers of the vertices in the base chamber in $D(T')$ are therefore mapped to the stabilizers of vertices in a chamber in $X$. 
        These vertices are $x_1,x_2,x_3$ so $\Delta$ is a chamber.
        Since $\Gamma$ acts chamber-regularly, the unique element that maps the base chamber in $X$ to $\Delta$ conjugates $\{ V_1, V_2, V_3\}$ to $\{V_1', V_2', V_3'\}$.
    \end{proof}
    
    \begin{figure}
        \centering
        \begin{tikzpicture}
            \coordinate (A) at (-3,0);        
            \coordinate (B) at (3,0);         
            \coordinate (C) at (0,-3);
            \coordinate (D) at (-1.8,0);
            \coordinate (E) at (-2.4,-0.6);
            \coordinate (F) at (1.8,0);
            \coordinate (G) at (2.4,-0.6);
            \coordinate (H) at (-0.6,-2.4);
            \coordinate (I) at (0.6,-2.4);
            \coordinate (J) at (-1.2, 0);
            \coordinate (K) at (-2.1,-0.9);
            \coordinate (L) at (1.2,0);
            \coordinate (M) at (2.1,-0.9);
            \coordinate (N) at (-0.9,-2.1);
            \coordinate (O) at (0.9,-2.1);
            \coordinate (P) at (-1.5, -0.3);
            \coordinate (Q) at (-1.8, -0.6);
            \coordinate (R) at (1.5, -0.3);
            \coordinate (S) at (1.8, -0.6);
            \coordinate (T) at (0.3, -2.1);
            \coordinate (U) at (-0.3, -2.1);
            \coordinate (V) at (-0.6, -2.1);
            \coordinate (W) at (0.6, -2.1);
            
            \draw[thick, RoyalBlue] (A) -- (D);
            \draw[thick, Red] (A) -- (E);
            \draw[thick, RoyalBlue] (B) -- (F);
            \draw[thick, ForestGreen] (B) -- (G);
            \draw[thick, Red] (C) -- (H);
            \draw[thick, ForestGreen] (C) -- (I);
            \draw[thick, ForestGreen] (D) -- (E);
            \draw[thick, Red] (F) -- (G);
            \draw[thick, RoyalBlue] (H) -- (I);
            \draw[thick, Red] (E) -- (K); 
            \draw[thick, RoyalBlue] (D) -- (J); 
            \draw[thick, RoyalBlue] (F) -- (L); 
            \draw[thick, ForestGreen] (G) -- (M); 
            \draw[thick, Red] (H) -- (N); 
            \draw[thick, ForestGreen] (I) -- (O); 
            \draw[thick, ForestGreen] (D) -- (P); 
            \draw[thick, RoyalBlue] (D) -- (Q); 
            \draw[thick, Red] (F) -- (R); 
            \draw[thick, RoyalBlue] (F) -- (S); 
            \draw[thick, Red] (H) -- (U); 
            \draw[thick, ForestGreen] (I) -- (T);
            \draw[thick, RoyalBlue] (H) -- (V);
            \draw[thick, RoyalBlue] (I) -- (W);
            \draw[thick, dashed, <->, shorten >=5mm, shorten <=5mm] (K) -- (N);
            \draw[thick, dashed, <->, shorten >=7mm, shorten <=7mm] (J) -- (L);
            \draw[thick, dashed, <->, shorten >=5mm, shorten <=5mm] (O) -- (M);
            \node[left] at (A) {$x_1$};
            \node[right] at (B) {$x_2$};
            \node[below] at (C) {$x_3$};            
            \draw[fill = Mulberry, thick] (A) circle (0.7mm);
            \draw[fill = BlueGreen, thick] (B) circle (0.7mm);
            \draw[fill = Dandelion, thick] (C) circle (0.7mm);
            \draw[fill = BlueGreen, thick] (D) circle (0.7mm);
            \draw[fill = Dandelion, thick] (E) circle (0.7mm);
            \draw[fill = Mulberry, thick] (F) circle (0.7mm);
            \draw[fill = Dandelion, thick] (G) circle (0.7mm);
            \draw[fill = Mulberry, thick] (H) circle (0.7mm);
            \draw[fill = BlueGreen, thick] (I) circle (0.7mm);
        \end{tikzpicture}
        \caption{The configuration of $x_1, x_2, x_3$ in the proof of Proposition \ref{prop:recover} in type $\tilde C_2$.}
        \label{fig:enter-label}
    \end{figure}
    
    \begin{remark}
        The assumptions in Lemma \ref{lem:vertex_stab} are necessary to deduce the second statement on edge stabilizers. One may consider the alternating group $V:= A_5$ and its subgroups $E=\langle(1 \; 2\;3)\rangle$ and $F:=\langle(1\;4\;5),(1\;4)(2\;3)\rangle$. Note that $E\cap F = 1$, $\langle E, F\rangle = A_5$ and that the element $g = (2 \; 4)(3 \; 5)$ conjugates $E$ to a subgroup of $F$. In particular, if $T$ is a triangle of groups with one local action being the one defined by $(V,E,F)$, then the stabilizer of the edge $gE \in D(T)$ in $\pi_1(T)$ is not maximal among the subgroups that embed in two vertex stabilizers since $\Ad(g)(E) \leq F$.
    \end{remark}

\section{The generalized quadrangle of order $(3,5)$}\label{sec:quadrangle_Q}
    In this section we define the generalized quadrangle $Q$ of order (3,5) and describe two chamber-regular action on $Q$.
    
    Up to this point, we have viewed generalized quadrangles as bipartite graphs of diameter four and girth eight, as this perspective is more natural when they arise as vertex links. However, it is also standard to regard them as point-line geometries, and the two viewpoints are essentially equivalent:

    Given a generalized quadrangle presented as a point-line geometry, its incidence graph is a bipartite graph of diameter four and girth eight. Conversely, starting from a generalized quadrangle given as a bipartite graph, one obtains a point-line geometry by declaring one vertex type to consist of points and the other of lines, with incidence defined by adjacency. Depending on which vertex type is chosen as the set of points, one recovers either the original geometry or its dual.

    In this section, and only in this section, we consider generalized quadrangles as point-line geometries, as this will be more convenient. Given a generalized quadrangle, we define its order $(s,t)$ to be the pair of integers such that each point is incident with exactly $t+1$ lines and each line is incident with exactly $s+1$ points. If $(s,t)$ is the order of a generalized quadrangle, then the valencies of the two vertex types in its incidence graph are $t+1$ and $s+1$, respectively.
    
    The generalized quadrangle $Q$ is a member of the family of slanted symplectic quadrangles, which are Payne-derivations of symplectic quadrangles; we refer to \cite{GrundhoferJoswigStroppel} for more details on this family of quadrangles. 
    We recall a well-known construction of the slanted symplectic quadrangle $W^\diamond(\FF)$ over the field~$\FF$:
    let $M(x,y,z) \in \SL_4(\FF)$ be the matrix defined below and let $S$ be the group consisting of elements of this form. 
    Let $X_0 = \{M(0,0,\mu) \mid \mu \in \FF\}$.
    Identify the projective line $\textbf{P}^1$ with 1-dimensional subspaces of $\FF^2$ and for each $\textbf{p} \in \textbf{P}^1$, choose a non-zero vector $(p_a,p_b)^T \in \textbf{p} \subseteq \FF^2$ to define the subgroup $X_{\textbf{p}} = \{M(\mu p_a, \mu p_b, 0) \mid \mu \in \FF\}$.
    Now the slanted symplectic quadrangle $W^\diamond(\FF)$ is the incidence geometry with points $S$ and lines being the left cosets of the subgroups $X_0$ and $X_{\textbf{p}}$.
    Note that $S$ acts point-regularly on $W^\diamond(\FF)$ by left multiplication. If $\FF$ is finite of order $q$ the quadrangle $W^\diamond(\FF)$ is of order $(q-1, q+1)$.
    \[
    M(x,y,z) :=
    \left(
        \begin{matrix}
            1 & -x & y & z \\
            0 & 1  & 0 & y \\
            0 & 0  & 1 & x \\
            0 & 0  & 0 & 1
        \end{matrix}
    \right).
    \]
    If $\FF$ is of order $q = 2^k$, then $S$ is elementary abelian of order $2^{3k}$.
    In particular this is the case for $q=4$ which yields the unique quadrangle $Q$ of order (3,5). 
    We fix $q=4$. Let $\alpha \in \FF_4$ be the element satisfying $\alpha^2+\alpha+1=1$. The assignment $\Phi$ indicated below extends to an isomorphism from $S$ (defined over $\FF_4$) to the vector space $V:= \FF_2^6$. We denote the standard basis with $(e_i)_i$.
    \begin{align*}
        \Phi:\hspace{-8mm} &&
        M(1,0,0)
        &\mapsto e_1, 
        &
        M(\alpha, 0,0)
        & \mapsto e_2,
        &
        M(0,1,0)
        & \mapsto e_3,\\
        && 
        M(0,\alpha,0)
        &\mapsto e_4,
        &
        M(0,0,1)
        &\mapsto e_5,
        &
        M(0,0,\alpha)
        &\mapsto e_6.
    \end{align*}
    The images of the subgroups $X_0$ and $X_{\textbf{p}}$ are the following subspaces.
    \begin{align*}
        U_1 &:= \langle e_5, e_6 \rangle_{\FF_2} && \hspace{-1.6cm}= \Phi \left(X_0\right) ,
        \\
        U_2 &:= \langle e_1, e_2 \rangle_{\FF_2} && \hspace{-1.6cm}= \Phi\left(X_{\langle(1,0)^T\rangle_{\FF_2}}\right) ,
        \\
        U_3 &:= \langle e_3, e_4 \rangle_{\FF_2} &&\hspace{-1.6cm}= \Phi\left(X_{\langle(0,1)^T\rangle_{\FF_2}}\right) ,
        \\
        U_4 &:= \langle e_1 + e_3 + e_5, e_2 + e_4 + e_5 + e_6\rangle_{\FF_2} &&\hspace{-1.6cm}= \Phi\left(X_{\langle(1,1)^T\rangle_{\FF_2}}\right) ,
        \\
        U_5 &:= \langle e_1 + e_4 + e_6, e_2 + e_3 + e_4 + e_5 \rangle_{\FF_2} &&\hspace{-1.6cm}= \Phi\left(X_{\langle(1,\alpha)^T\rangle_{\FF_2}}\right) ,
        \\
        U_6 &:= \langle e_2 + e_3 + e_6, e_1+e_2+e_4+e_5 \rangle_{\FF_2} &&\hspace{-1.6cm}= \Phi\left(X_{\langle(\alpha,1)^T\rangle_{\FF_2}}\right) .
        \end{align*}
        
    Now consider the matrices $A,B,C$ from Table \ref{tbl:matrices}. We make the following observation, which can be verified by hand or with code provided in \cite{ChamberRegularCode}.
    
    \begin{observation}
        \begin{enumerate}
            \item The matrix $A$ is of order six and $\langle A \rangle$ acts regularly (by left multiplication) on the subspaces $U_i$. In particular the affine group $V \rtimes \langle A \rangle$ acts regularly on the chambers of $Q$.
            \item The matrices $B$ and $C$ are of order two and three respectively and $\langle B,C \rangle \cong \Sym(3)$ acts regularly (by left multiplication) on the subspaces $U_i$. In particular the affine group $V \rtimes \langle B,C \rangle$ acts regularly on the chambers of~$Q$.
        \end{enumerate}
    \end{observation}

    \begin{remark}
        The groups $V \rtimes \langle A \rangle$ and $V \rtimes \langle B,C \rangle$ are isomorphic to the groups $L_7$ and $L_{11}$ from Table \ref{tbl:presentations_quadrangle} in Appendix \ref{app:presentations}.
    \end{remark}
    
    \begin{table}\label{tbl:matrices}
    \begin{align*}
        A &:= \begin{pmatrix}
          0 & 0 & 1 & 1 & 1 & 1 \\
          0 & 0 & 0 & 1 & 1 & 0 \\
          1 & 1 & 1 & 1 & 1 & 0 \\
          0 & 1 & 0 & 1 & 0 & 1 \\
          0 & 0 & 1 & 0 & 1 & 0 \\
          0 & 0 & 0 & 1 & 1 & 1
        \end{pmatrix},
        &
        B &:= \begin{pmatrix}
          1 & 0 & 1 & 1 & 0 & 1 \\
          1 & 1 & 0 & 1 & 1 & 1 \\
          1 & 1 & 1 & 1 & 1 & 0 \\
          0 & 1 & 0 & 1 & 0 & 1 \\
          1 & 1 & 1 & 0 & 0 & 1 \\
          1 & 0 & 0 & 1 & 1 & 0
        \end{pmatrix},\\
        C &:= \begin{pmatrix}
          1 & 1 & 0 & 1 & 0 & 0 \\
          1 & 0 & 1 & 1 & 0 & 0 \\
          1 & 1 & 1 & 1 & 1 & 0 \\
          1 & 0 & 1 & 0 & 1 & 1 \\
          0 & 1 & 1 & 1 & 0 & 0 \\
          1 & 0 & 0 & 1 & 0 & 0
        \end{pmatrix},
        &
        D &:= \begin{pmatrix}
            0 & 1 & 0 & 0 & 0 & 0 \\
            1 & 1 & 0 & 0 & 0 & 0 \\
            0 & 0 & 0 & 1 & 0 & 0 \\
            0 & 0 & 1 & 1 & 0 & 0 \\
            0 & 0 & 0 & 0 & 1 & 1 \\
            0 & 0 & 0 & 0 & 1 & 0
        \end{pmatrix}.
    \end{align*}
        \caption{Matrices over $\FF_2$.}
        \label{tbl:matrices}
    \end{table}

    The following proposition was checked with a computer \cite{ChamberRegularCode}.

    \begin{proposition}
        \begin{enumerate}
            \item 
            The affine group $V \rtimes \langle A,C\rangle$ is the full automorphism group of $Q$. In particular the automorphism group of $Q$ contains a normal subgroup acting regularly on the points.
            \item
            Up to conjugacy there are two chamber-regular subgroups of $\Aut(Q)$ that contain the normal subgroup $V$. These are exactly the groups from the previuos observation.
            \item
            Since $\langle A,C\rangle$ permutes the subspaces $U_1, \dots, U_6$ we get a homomorphism~$\Psi$ from $\langle A, C\rangle$ to $\Sym(6)$. The kernel is cyclic of order three and generated by the matrix~$D$ indicated in Table \ref{tbl:matrices}. The images of $A$ and $C$ under $\Psi$ are $(1,5,2,3,4,6)$ and $(1,3,5)(2,4,6)$ respectively. In particular $\Psi$ is surjective and we have the following (non-split) exact sequence.
            \[
                1 \rightarrow C_3 \cong \langle D \rangle \rightarrow \langle A,C\rangle \rightarrow \Sym(6) \rightarrow 1.
            \]
        \end{enumerate}
    \end{proposition}

    \begin{remark}
        We point out that the quadrangle $Q$ admits many more regular actions than the ones we discuss in the previous discussion. In fact (up to conjugacy in the automorphism group) there are 58 groups acting regular on the points, six groups acting regular on the lines, and 11 groups acting regular on the chambers of $Q$. Several of the point- and line-regular actions arise as subactions of the chamber-regular ones.
        In Appendix \ref{app:presentations} we provide explicit presentations for the 11 chamber-regular groups.
    \end{remark}
    \begin{figure}[H]
    \centering
    \begin{align*}
    \begin{tikzpicture}[scale=0.5]
        \def\Radiusa{0.5cm}
        \def\Radiusb{1cm}
        \def\Radiusc{2.3cm}
        \def\Radiusd{3.5cm}
        \def\Radiuse{4.5cm}
        \def\Radiusf{5.5cm}
        \def\Radiush{6.5cm}
        \def\Radiusg{7.5cm}
            \coordinate (1a) at (270:\Radiusa);
            \coordinate (14a) at (270:\Radiusb);
            \coordinate (2a) at (250:\Radiusc);
            \coordinate (7a) at (290:\Radiusc);
            \coordinate (3a) at (250:\Radiusd);
            \coordinate (8a) at (290:\Radiusd);
            \coordinate (12a) at ($(3a)!0.5!(8a)$);
            \coordinate (4a) at (250:\Radiuse);
            \coordinate (9a) at (290:\Radiuse);
            \coordinate (13a) at ($(4a)!0.5!(9a)$);
            \coordinate (5a) at (250:\Radiusf);
            \coordinate (10a) at (290:\Radiusf);
            \coordinate (15a) at ($(5a)!0.5!(10a)$);
            \coordinate (5aa) at (250:\Radiush);
            \coordinate (6a) at (250:\Radiusg);
            \coordinate (11a) at (290:\Radiusg);
            \coordinate (16a) at ($(6a)!0.5!(11a)$);
            \coordinate (15aa) at ($(15a)!0.5!(16a)$);
            \coordinate (1b) at (210:\Radiusa);
            \coordinate (14b) at (210:\Radiusb);
            \coordinate (2b) at (190:\Radiusc);
            \coordinate (7b) at (230:\Radiusc);
            \coordinate (3b) at (190:\Radiusd);
            \coordinate (8b) at (230:\Radiusd);
            \coordinate (12b) at ($(3b)!0.5!(8b)$);
            \coordinate (4b) at (190:\Radiuse);
            \coordinate (9b) at (230:\Radiuse);
            \coordinate (13b) at ($(4b)!0.5!(9b)$);
            \coordinate (5b) at (190:\Radiusf);
            \coordinate (10b) at (230:\Radiusf);
            \coordinate (15b) at ($(5b)!0.5!(10b)$);
            \coordinate (5bb) at (190:\Radiush);
            \coordinate (6b) at (190:\Radiusg);
            \coordinate (11b) at (230:\Radiusg);
            \coordinate (16b) at ($(6b)!0.5!(11b)$);
            \coordinate (15bb) at ($(15b)!0.5!(16b)$);
            \coordinate (1c) at (150:\Radiusa);
            \coordinate (14c) at (150:\Radiusb);
            \coordinate (2c) at (130:\Radiusc);
            \coordinate (7c) at (170:\Radiusc);
            \coordinate (3c) at (130:\Radiusd);
            \coordinate (8c) at (170:\Radiusd);
            \coordinate (12c) at ($(3c)!0.5!(8c)$);
            \coordinate (4c) at (130:\Radiuse);
            \coordinate (9c) at (170:\Radiuse);
            \coordinate (13c) at ($(4c)!0.5!(9c)$);
            \coordinate (5c) at (130:\Radiusf);
            \coordinate (10c) at (170:\Radiusf);
            \coordinate (15c) at ($(5c)!0.5!(10c)$);
            \coordinate (5cc) at (130:\Radiush);
            \coordinate (6c) at (130:\Radiusg);
            \coordinate (11c) at (170:\Radiusg);
            \coordinate (16c) at ($(6c)!0.5!(11c)$);
            \coordinate (15cc) at ($(15c)!0.5!(16c)$);
            \coordinate (1d) at (90:\Radiusa);
            \coordinate (14d) at (90:\Radiusb);
            \coordinate (2d) at (70:\Radiusc);
            \coordinate (7d) at (110:\Radiusc);
            \coordinate (3d) at (70:\Radiusd);
            \coordinate (8d) at (110:\Radiusd);
            \coordinate (12d) at ($(3d)!0.5!(8d)$);
            \coordinate (4d) at (70:\Radiuse);
            \coordinate (9d) at (110:\Radiuse);
            \coordinate (13d) at ($(4d)!0.5!(9d)$);
            \coordinate (5d) at (70:\Radiusf);
            \coordinate (10d) at (110:\Radiusf);
            \coordinate (15d) at ($(5d)!0.5!(10d)$);
            \coordinate (5dd) at (70:\Radiush);
            \coordinate (6d) at (70:\Radiusg);
            \coordinate (11d) at (110:\Radiusg);
            \coordinate (16d) at ($(6d)!0.5!(11d)$);
            \coordinate (15dd) at ($(15d)!0.5!(16d)$);
            \coordinate (1e) at (30:\Radiusa);
            \coordinate (14e) at (30:\Radiusb);
            \coordinate (2e) at (10:\Radiusc);
            \coordinate (7e) at (50:\Radiusc);
            \coordinate (3e) at (10:\Radiusd);
            \coordinate (8e) at (50:\Radiusd);
            \coordinate (12e) at ($(3e)!0.5!(8e)$);
            \coordinate (4e) at (10:\Radiuse);
            \coordinate (9e) at (50:\Radiuse);
            \coordinate (13e) at ($(4e)!0.5!(9e)$);
            \coordinate (5e) at (10:\Radiusf);
            \coordinate (10e) at (50:\Radiusf);
            \coordinate (15e) at ($(5e)!0.5!(10e)$);
            \coordinate (5ee) at (10:\Radiush);
            \coordinate (6e) at (10:\Radiusg);
            \coordinate (11e) at (50:\Radiusg);
            \coordinate (16e) at ($(6e)!0.5!(11e)$);
            \coordinate (15ee) at ($(15e)!0.5!(16e)$);
            \coordinate (1f) at (330:\Radiusa);
            \coordinate (14f) at (330:\Radiusb);
            \coordinate (2f) at (310:\Radiusc);
            \coordinate (7f) at (350:\Radiusc);
            \coordinate (3f) at (310:\Radiusd);
            \coordinate (8f) at (350:\Radiusd);
            \coordinate (12f) at ($(3f)!0.5!(8f)$);
            \coordinate (4f) at (310:\Radiuse);
            \coordinate (9f) at (350:\Radiuse);
            \coordinate (13f) at ($(4f)!0.5!(9f)$);
            \coordinate (5f) at (310:\Radiusf);
            \coordinate (10f) at (350:\Radiusf);
            \coordinate (15f) at ($(5f)!0.5!(10f)$);
            \coordinate (5ff) at (310:\Radiush);
            \coordinate (6f) at (310:\Radiusg);
            \coordinate (11f) at (350:\Radiusg);
            \coordinate (16f) at ($(11f)!0.5!(6f)$);
            \coordinate (15ff) at ($(15f)!0.5!(16f)$);
               \draw[line width=0.2mm, radius=\Radiusa,  color =black] circle;
               \draw[line width=0.2mm, radius=\Radiusb,  color=black] circle;  
                 \draw[line width = 0.2mm, color = black] (1a) -- (2a);
                 \draw[line width = 0.2mm, color = black] (2a) -- (3a);
                 \draw[line width = 0.2mm, color = black] (3a) -- (4a);
                 \draw[line width = 0.2mm, color = black] (4a) -- (5a);
                 \draw[line width = 0.2mm, color = black] (5a) -- (6a);
                 \draw[line width = 0.2mm, color = black] (1a) -- (7a);
                 \draw[line width = 0.2mm, color = black] (7a) -- (8a);
                 \draw[line width = 0.2mm, color = black] (8a) -- (9a);
                 \draw[line width = 0.2mm, color = black] (9a) -- (10a);
                 \draw[line width = 0.2mm, color = black] (10a) -- (11a);
                 \draw[line width = 0.2mm, color = black] (1a) -- (14a);
                 \draw[line width = 0.2mm, color = black] (14a) -- (12a);
                 \draw[line width = 0.2mm, color = black] (12a) -- (13a);
                 \draw[line width = 0.2mm, color = black] (13a) -- (15a);
                 \draw[line width = 0.2mm, color = black] (15a) -- (16a);
                \draw[line width = 0.2mm, color = black] (1b) -- (2b);
                \draw[line width = 0.2mm, color = black] (2b) -- (3b);
                \draw[line width = 0.2mm, color = black] (3b) -- (4b);
                \draw[line width = 0.2mm, color = black] (4b) -- (5b);
                \draw[line width = 0.2mm, color = black] (5b) -- (6b);
                \draw[line width = 0.2mm, color = black] (1b) -- (7b);
                \draw[line width = 0.2mm, color = black] (7b) -- (8b);
                \draw[line width = 0.2mm, color = black] (8b) -- (9b);
                \draw[line width = 0.2mm, color = black] (9b) -- (10b);
                \draw[line width = 0.2mm, color = black] (10b) -- (11b);
                \draw[line width = 0.2mm, color = black] (1b) -- (14b);
                \draw[line width = 0.2mm, color = black] (14b) -- (12b);
                \draw[line width = 0.2mm, color = black] (12b) -- (13b);
                \draw[line width = 0.2mm, color = black] (13b) -- (15b);
                \draw[line width = 0.2mm, color = black] (15b) -- (16b);
                \draw[line width = 0.2mm, color = black] (1c) -- (2c);
                \draw[line width = 0.2mm, color = black] (2c) -- (3c);
                \draw[line width = 0.2mm, color = black] (3c) -- (4c);
                \draw[line width = 0.2mm, color = black] (4c) -- (5c);
                \draw[line width = 0.2mm, color = black] (5c) -- (6c);
                \draw[line width = 0.2mm, color = black] (1c) -- (7c);
                \draw[line width = 0.2mm, color = black] (7c) -- (8c);
                \draw[line width = 0.2mm, color = black] (8c) -- (9c);
                \draw[line width = 0.2mm, color = black] (9c) -- (10c);
                \draw[line width = 0.2mm, color = black] (10c) -- (11c);
                \draw[line width = 0.2mm, color = black] (1c) -- (14c);
                \draw[line width = 0.2mm, color = black] (14c) -- (12c);
                \draw[line width = 0.2mm, color = black] (12c) -- (13c);
                \draw[line width = 0.2mm, color = black] (13c) -- (15c);
                \draw[line width = 0.2mm, color = black] (15c) -- (16c);
                \draw[line width = 0.2mm, color = black] (1d) -- (2d);
                \draw[line width = 0.2mm, color = black] (2d) -- (3d);
                \draw[line width = 0.2mm, color = black] (3d) -- (4d);
                \draw[line width = 0.2mm, color = black] (4d) -- (5d);
                \draw[line width = 0.2mm, color = black] (5d) -- (6d);
                \draw[line width = 0.2mm, color = black] (1d) -- (7d);
                \draw[line width = 0.2mm, color = black] (7d) -- (8d);
                \draw[line width = 0.2mm, color = black] (8d) -- (9d);
                \draw[line width = 0.2mm, color = black] (9d) -- (10d);
                \draw[line width = 0.2mm, color = black] (10d) -- (11d);
                \draw[line width = 0.2mm, color = black] (1d) -- (14d);
                \draw[line width = 0.2mm, color = black] (14d) -- (12d);
                \draw[line width = 0.2mm, color = black] (12d) -- (13d);
                \draw[line width = 0.2mm, color = black] (13d) -- (15d);
                \draw[line width = 0.2mm, color = black] (15d) -- (16d);
                \draw[line width = 0.2mm, color = black] (1e) -- (2e);
                \draw[line width = 0.2mm, color = black] (2e) -- (3e);
                \draw[line width = 0.2mm, color = black] (3e) -- (4e);
                \draw[line width = 0.2mm, color = black] (4e) -- (5e);
                \draw[line width = 0.2mm, color = black] (5e) -- (6e);
                \draw[line width = 0.2mm, color = black] (1e) -- (7e);
                \draw[line width = 0.2mm, color = black] (7e) -- (8e);
                \draw[line width = 0.2mm, color = black] (8e) -- (9e);
                \draw[line width = 0.2mm, color = black] (9e) -- (10e);
                \draw[line width = 0.2mm, color = black] (10e) -- (11e);
                \draw[line width = 0.2mm, color = black] (1e) -- (14e);
                \draw[line width = 0.2mm, color = black] (14e) -- (12e);
                \draw[line width = 0.2mm, color = black] (12e) -- (13e);
                \draw[line width = 0.2mm, color = black] (13e) -- (15e);
                \draw[line width = 0.2mm, color = black] (15e) -- (16e);  
                \draw[line width = 0.2mm, color = black] (1f) -- (2f);
                \draw[line width = 0.2mm, color = black] (2f) -- (3f);
                \draw[line width = 0.2mm, color = black] (3f) -- (4f);
                \draw[line width = 0.2mm, color = black] (4f) -- (5f);
                \draw[line width = 0.2mm, color = black] (5f) -- (6f);
                \draw[line width = 0.2mm, color = black] (1f) -- (7f);
                \draw[line width = 0.2mm, color = black] (7f) -- (8f);
                \draw[line width = 0.2mm, color = black] (8f) -- (9f);
                \draw[line width = 0.2mm, color = black] (9f) -- (10f);
                \draw[line width = 0.2mm, color = black] (10f) -- (11f);
                \draw[line width = 0.2mm, color = black] (1f) -- (14f);
                \draw[line width = 0.2mm, color = black] (14f) -- (12f);
                \draw[line width = 0.2mm, color = black] (12f) -- (13f);
                \draw[line width = 0.2mm, color = black] (13f) -- (15f);
                \draw[line width = 0.2mm, color = black] (15f) -- (16f);    
                \draw[line width= 0.2mm, bend left=65, color = farbe2]( 2a) to (5b);
                \draw[line width = 0.2mm, bend left=90, color = farbe2] (5b) to (4c);
                \draw[line width = 0.2mm, bend right=40, color = farbe2] (4c) to (16d);
                \draw[line width = 0.2mm, bend left=90, color = farbe2] (16d) to (13e);
                \draw[line width = 0.2mm, bend right=10, color = farbe2] (13e) to (3f); 
                \draw[line width= 0.2mm, bend left=65, color = farbe15]( 2b) to (5c);
                \draw[line width = 0.2mm, bend left=90, color = farbe15] (5c) to (4d);
                \draw[line width = 0.2mm, bend right=40, color = farbe15] (4d) to (16e);
                \draw[line width = 0.2mm, bend left=90, color = farbe15] (16e) to (13f);
                \draw[line width = 0.2mm, bend right=10, color = farbe15] (13f) to (3a);
                \draw[line width= 0.2mm, bend left=65, color = farbe2]( 2c) to (5d);
                \draw[line width = 0.2mm, bend left=90, color = farbe2] (5d) to (4e);
                \draw[line width = 0.2mm, bend right=40, color = farbe2] (4e) to (16f);
                \draw[line width = 0.2mm, bend left=90, color = farbe2] (16f) to (13a);
                \draw[line width = 0.2mm, bend right=10, color = farbe2] (13a) to (3b);
                \draw[line width= 0.2mm, bend left=65, color = farbe15]( 2d) to (5e);
                \draw[line width = 0.2mm, bend left=90, color = farbe15] (5e) to (4f);
                \draw[line width = 0.2mm, bend right=40, color = farbe15] (4f) to (16a);
                \draw[line width = 0.2mm, bend left=90, color = farbe15] (16a) to (13b);
                \draw[line width = 0.2mm, bend right=10, color = farbe15] (13b) to (3c);
                \draw[line width= 0.2mm, bend left=65, color = farbe2]( 2e) to (5f);
                \draw[line width = 0.2mm, bend left=90, color = farbe2] (5f) to (4a);
                \draw[line width = 0.2mm, bend right=40, color = farbe2] (4a) to (16b);
                \draw[line width = 0.2mm, bend left=90, color = farbe2] (16b) to (13c);
                \draw[line width = 0.2mm, bend right=10, color = farbe2] (13c) to (3d);
                \draw[line width= 0.2mm, bend left=65, color = farbe15]( 2f) to (5a);
                \draw[line width = 0.2mm, bend left=90, color = farbe15] (5a) to (4b);
                \draw[line width = 0.2mm, bend right=40, color = farbe15] (4b) to (16c);
                \draw[line width = 0.2mm, bend left=90, color = farbe15] (16c) to (13d);
                \draw[line width = 0.2mm, bend right=10, color = farbe15] (13d) to (3e);
                \draw[line width= 0.2mm, bend right=5, color = farbe5]( 14a) to (8b);
                \draw[line width = 0.2mm, bend right=25, color = farbe5] (8b) to (4c);
                \draw[line width = 0.2mm, bend right=40, color = farbe5] (4c) to (6d);
                \draw[line width = 0.2mm, bend left=10, color = farbe5] (6d) to (3e);
                \draw[line width = 0.2mm, bend right=10, color = farbe5] (3e) to (9f);
                \draw[line width= 0.2mm, bend right=5, color = farbe11]( 14b) to (8c);
                \draw[line width = 0.2mm, bend right=25, color = farbe11] (8c) to (4d);
                \draw[line width = 0.2mm, bend right=40, color = farbe11] (4d) to (6e);
                \draw[line width = 0.2mm, bend left=10, color = farbe11] (6e) to (3f);
                \draw[line width = 0.2mm, bend right=10, color = farbe11] (3f) to (9a);
                \draw[line width= 0.2mm, bend right=5, color = farbe5]( 14c) to (8d);
                \draw[line width = 0.2mm, bend right=25, color = farbe5] (8d) to (4e);
                \draw[line width = 0.2mm, bend right=40, color = farbe5] (4e) to (6f);
                \draw[line width = 0.2mm, bend left=10, color = farbe5] (6f) to (3a);
                \draw[line width = 0.2mm, bend right=10, color = farbe5] (3a) to (9b);
                \draw[line width= 0.2mm, bend right=5, color = farbe11]( 14d) to (8e);
                \draw[line width = 0.2mm, bend right=25, color = farbe11] (8e) to (4f);
                \draw[line width = 0.2mm, bend right=40, color = farbe11] (4f) to (6a);
                \draw[line width = 0.2mm, bend left=10, color = farbe11] (6a) to (3b);
                \draw[line width = 0.2mm, bend right=10, color = farbe11] (3b) to (9c);
                \draw[line width= 0.2mm, bend right=5, color = farbe5]( 14e) to (8f);
                \draw[line width = 0.2mm, bend right=25, color = farbe5] (8f) to (4a);
                \draw[line width = 0.2mm, bend right=40, color = farbe5] (4a) to (6b);
                \draw[line width = 0.2mm, bend left=10, color = farbe5] (6b) to (3c);
                \draw[line width = 0.2mm, bend right=10, color = farbe5] (3c) to (9d);
                \draw[line width= 0.2mm, bend right=5, color = farbe11]( 14f) to (8a);
                \draw[line width = 0.2mm, bend right=25, color = farbe11] (8a) to (4b);
                \draw[line width = 0.2mm, bend right=40, color = farbe11] (4b) to (6c);
                \draw[line width = 0.2mm, bend left=10, color = farbe11] (6c) to (3d);
                \draw[line width = 0.2mm, bend right=10, color = farbe11] (3d) to (9e);
                %
                \draw[line width= 0.2mm, bend left=45, color = dunkelblau]( 11a) to (13b);
                \draw[line width = 0.2mm, bend left=70, color = dunkelblau] (13b) to (16c);
                \draw[line width = 0.2mm, bend left=45, color = dunkelblau] (16c) to (11d);
                \draw[line width = 0.2mm, bend left=45, color = dunkelblau] (11d) to (13e);
                \draw[line width = 0.2mm, bend left=70, color = dunkelblau ] (13e) to (16f);
                \draw[line width = 0.2mm, bend left=45, color = dunkelblau] (16f) to (11a);
                \draw[line width= 0.2mm, bend left=45, color = dunkelblau]( 11b) to (13c);
                \draw[line width = 0.2mm, bend left=70, color = dunkelblau] (13c) to (16d);
                \draw[line width = 0.2mm, bend left=45, color = dunkelblau] (16d) to (11e);
                \draw[line width = 0.2mm, bend left=45, color = dunkelblau] (11e) to (13f);
                \draw[line width = 0.2mm, bend left=70, color = dunkelblau ] (13f) to (16a);
                \draw[line width = 0.2mm, bend left=45, color = dunkelblau] (16a) to (11b);
                \draw[line width= 0.2mm, bend left=45, color = dunkelblau]( 11c) to (13d);
                \draw[line width = 0.2mm, bend left=70, color = dunkelblau] (13d) to (16e);
                \draw[line width = 0.2mm, bend left=45, color = dunkelblau] (16e) to (11f);
                \draw[line width = 0.2mm, bend left=45, color = dunkelblau] (11f) to (13a);
                \draw[line width = 0.2mm, bend left=70, color = dunkelblau ] (13a) to (16b);
                \draw[line width = 0.2mm, bend left=45, color = dunkelblau] (16b) to (11c);
                \draw[line width= 0.2mm, bend right=5, color = farbe9]( 2a) to (15b);
                \draw[line width = 0.2mm, bend right=25, color = farbe9] (15b) to (12c);
                \draw[line width = 0.2mm, bend left=0, color = farbe9] (12c) to (5d);
                \draw[line width = 0.2mm, bend left=10, color = farbe9] (5d) to (7e);
                \draw[line width = 0.2mm, bend right=10, color = farbe9] (7e) to (10f);
                \draw[line width= 0.2mm, bend right=5, color = farbe7]( 2b) to (15c);
                \draw[line width = 0.2mm, bend right=25, color = farbe7] (15c) to (12d);
                \draw[line width = 0.2mm, bend left=0, color = farbe7] (12d) to (5e);
                \draw[line width = 0.2mm, bend left=10, color = farbe7] (5e) to (7f);
                \draw[line width = 0.2mm, bend right=10, color = farbe7] (7f) to (10a);
                \draw[line width= 0.2mm, bend right=5, color = farbe9]( 2c) to (15d);
                \draw[line width = 0.2mm, bend right=25, color = farbe9] (15d) to (12e);
                \draw[line width = 0.2mm, bend left=0, color = farbe9] (12e) to (5f);
                \draw[line width = 0.2mm, bend left=10, color = farbe9] (5f) to (7a);
                \draw[line width = 0.2mm, bend right=10, color = farbe9] (7a) to (10b);
                \draw[line width= 0.2mm, bend right=5, color = farbe7]( 2d) to (15e);
                \draw[line width = 0.2mm, bend right=25, color = farbe7] (15e) to (12f);
                \draw[line width = 0.2mm, bend left=0, color = farbe7] (12f) to (5a);
                \draw[line width = 0.2mm, bend left=10, color = farbe7] (5a) to (7b);
                \draw[line width = 0.2mm, bend right=10, color = farbe7] (7b) to (10c);
                \draw[line width= 0.2mm, bend right=5, color = farbe9]( 2e) to (15f);
                \draw[line width = 0.2mm, bend right=25, color = farbe9] (15f) to (12a);
                \draw[line width = 0.2mm, bend left=0, color = farbe9] (12a) to (5b);
                \draw[line width = 0.2mm, bend left=10, color = farbe9] (5b) to (7c);
                \draw[line width = 0.2mm, bend right=10, color = farbe9] (7c) to (10d);
                \draw[line width= 0.2mm, bend right=5, color = farbe7]( 2f) to (15a);
                \draw[line width = 0.2mm, bend right=25, color = farbe7] (15a) to (12b);
                \draw[line width = 0.2mm, bend left=0, color = farbe7] (12b) to (5c);
                \draw[line width = 0.2mm, bend left=10, color = farbe7] (5c) to (7d);
                \draw[line width = 0.2mm, bend right=10, color = farbe7] (7d) to (10e);
                \draw[line width= 0.2mm, bend left=45, color = BlueGreen]( 6a) to (11b);
                \draw[line width = 0.2mm, bend left=70, color = BlueGreen] (11b) to (6c);
                \draw[line width = 0.2mm, bend left=45, color = BlueGreen] (6c) to (11d);
                \draw[line width = 0.2mm, bend left=70, color = BlueGreen] (11d) to (6e);
                \draw[line width = 0.2mm, bend left=45, color = BlueGreen ] (6e) to (11f);
                \draw[line width= 0.2mm, bend left=70, color = BlueGreen](11f) to (6a);
                \draw[line width= 0.2mm, bend left=45, color = dunkelblau] (6b) to (11c);
                \draw[line width = 0.2mm, bend left=70, color = dunkelblau] (11c) to (6d);
                \draw[line width = 0.2mm, bend left=45, color = dunkelblau] (6d) to (11e);
                \draw[line width = 0.2mm, bend left=70, color = dunkelblau] (11e) to (6f);
                \draw[line width = 0.2mm, bend left=45, color = dunkelblau] (6f) to (11a);
                \draw[line width= 0.2mm, bend left=70, color = dunkelblau] (11a) to (6b);
                \draw[line width= 0.2mm, bend left=5, color = burgunder](14a) to (2b);
                \draw[line width = 0.2mm, bend left=25, color = burgunder] (2b) to (10c);
                \draw[line width = 0.2mm, bend left=25, color = burgunder] (10c) to (11d);
                \draw[line width = 0.2mm, bend left=90, color = burgunder] (11d) to (7e);
                \draw[line width = 0.2mm, bend right=10, color = burgunder] (7e) to (5f);
                \draw[line width= 0.2mm, bend left=5, color = burgunder] (14b) to (2c);
                \draw[line width = 0.2mm, bend left=25, color = burgunder] (2c) to (10d);
                \draw[line width = 0.2mm, bend left=25, color = burgunder] (10d) to (11e);
                \draw[line width = 0.2mm, bend left=90, color = burgunder] (11e) to (7f);
                \draw[line width = 0.2mm, bend right=10, color = burgunder] (7f) to (5a);
                \draw[line width= 0.2mm, bend left=5, color = burgunder] (14c) to (2d);
                \draw[line width = 0.2mm, bend left=25, color = burgunder] (2d) to (10e);
                \draw[line width = 0.2mm, bend left=25, color = burgunder] (10e) to (11f);
                \draw[line width = 0.2mm, bend left=90, color = burgunder] (11f) to (7a);
                \draw[line width = 0.2mm, bend right=10, color = burgunder] (7a) to (5b);
                \draw[line width= 0.2mm, bend left=5, color = burgunder] (14d) to (2e);
                \draw[line width = 0.2mm, bend left=25, color = burgunder] (2e) to (10f);
                \draw[line width = 0.2mm, bend left=25, color = burgunder] (10f) to (11a);
                \draw[line width = 0.2mm, bend left=90, color = burgunder] (11a) to (7b);
                \draw[line width = 0.2mm, bend right=10, color = burgunder] (7b) to (5c);
                \draw[line width= 0.2mm, bend left=5, color = burgunder] (14e) to (2f);
                \draw[line width = 0.2mm, bend left=25, color = burgunder] (2f) to (10a);
                \draw[line width = 0.2mm, bend left=25, color = burgunder] (10a) to (11b);
                \draw[line width = 0.2mm, bend left=90, color = burgunder] (11b) to (7c);
                \draw[line width = 0.2mm, bend right=10, color = burgunder] (7c) to (5d);
                \draw[line width= 0.2mm, bend left=5, color = burgunder] (14f) to (2a);
                \draw[line width = 0.2mm, bend left=25, color = burgunder] (2a) to (10b);
                \draw[line width = 0.2mm, bend left=25, color = burgunder] (10b) to (11c);
                \draw[line width = 0.2mm, bend left=90, color = burgunder] (11c) to (7d);
                \draw[line width = 0.2mm, bend right=10, color = burgunder] (7d) to (5e);
                \draw[line width= 0.2mm, bend left=30, color = farbe14] (15a) to (5aa);
                \draw[line width= 0.2mm, bend left=30, color = farbe14] (5aa) to (9b);
                \draw[line width = 0.2mm, bend left=50, color = farbe14] (9b) to (9c);
                \draw[line width = 0.2mm, bend left=60, color = farbe14] (9c) to (3d);
                \draw[line width = 0.2mm, bend left=30, color = farbe14] (3d) to (13e);
                \draw[line width = 0.2mm, bend left=10, color = farbe14] (13e) to (7f);
                \draw[line width= 0.2mm, bend left=30, color = farbe13] (15b) to (5bb);
                \draw[line width= 0.2mm, bend left=30, color = farbe13] (5bb) to (9c);
                \draw[line width = 0.2mm, bend left=50, color = farbe13] (9c) to (9d);
                \draw[line width = 0.2mm, bend left=60, color = farbe13] (9d) to (3e);
                \draw[line width = 0.2mm, bend left=30, color = farbe13] (3e) to (13f);
                \draw[line width = 0.2mm, bend left=10, color = farbe13] (13f) to (7a);
                \draw[line width= 0.2mm, bend left=30, color = farbe14] (15c) to (5cc);
                \draw[line width= 0.2mm, bend left=30, color = farbe14] (5cc) to (9d);
                \draw[line width = 0.2mm, bend left=50, color = farbe14] (9d) to (9e);
                \draw[line width = 0.2mm, bend left=60, color = farbe14] (9e) to (3f);
                \draw[line width = 0.2mm, bend left=30, color = farbe14] (3f) to (13a);
                \draw[line width = 0.2mm, bend left=10, color = farbe14] (13a) to (7b);
                \draw[line width= 0.2mm, bend left=30, color = farbe13] (15d) to (5dd);
                \draw[line width= 0.2mm, bend left=30, color = farbe13] (5dd) to (9e);
                \draw[line width = 0.2mm, bend left=50, color = farbe13] (9e) to (9f);
                \draw[line width = 0.2mm, bend left=60, color = farbe13] (9f) to (3a);
                \draw[line width = 0.2mm, bend left=30, color = farbe13] (3a) to (13b);
                \draw[line width = 0.2mm, bend left=10, color = farbe13] (13b) to (7c);
                \draw[line width= 0.2mm, bend left=30, color = farbe14] (15e) to (5ee);
                \draw[line width= 0.2mm, bend left=30, color = farbe14] (5ee) to (9f);
                \draw[line width = 0.2mm, bend left=50, color = farbe14] (9f) to (9a);
                \draw[line width = 0.2mm, bend left=60, color = farbe14] (9a) to (3b);
                \draw[line width = 0.2mm, bend left=30, color = farbe14] (3b) to (13c);
                \draw[line width = 0.2mm, bend left=10, color = farbe14] (13c) to (7d);
                \draw[line width= 0.2mm, bend left=30, color = farbe13] (15f) to (5ff);
                \draw[line width= 0.2mm, bend left=30, color = farbe13] (5ff) to (9a);
                \draw[line width = 0.2mm, bend left=50, color = farbe13] (9a) to (9b);
                \draw[line width = 0.2mm, bend left=60, color = farbe13] (9b) to (3c);
                \draw[line width = 0.2mm, bend left=30, color = farbe13] (3c) to (13d);
                \draw[line width = 0.2mm, bend left=10, color = farbe13] (13d) to (7e);
                \draw[line width= 0.2mm, bend left=20, color = farbe6] (4a) to (8b);
                \draw[line width = 0.2mm, bend left=80, color = farbe6] (8b) to (8c);
                \draw[line width = 0.2mm, bend right=45, color = farbe6] (8c) to (12d);
                \draw[line width = 0.2mm, bend left=70, color = farbe6] (12d) to (10e);
                \draw[line width = 0.2mm, bend left=40, color = farbe6] (10e) to (16f);
                \draw[line width= 0.2mm, bend left=20, color = farbe10] (4b) to (8c);
                \draw[line width = 0.2mm, bend left=80, color = farbe10] (8c) to (8d);
                \draw[line width = 0.2mm, bend right=45, color = farbe10] (8d) to (12e);
                \draw[line width = 0.2mm, bend left=70, color = farbe10] (12e) to (10f);
                \draw[line width = 0.2mm, bend left=40, color = farbe10] (10f) to (16a);
                \draw[line width= 0.2mm, bend left=20, color = farbe6] (4c) to (8d);
                \draw[line width = 0.2mm, bend left=80, color = farbe6] (8d) to (8e);
                 \draw[line width = 0.2mm, bend right=45, color = farbe6] (8e) to (12f);
                \draw[line width = 0.2mm, bend left=70, color = farbe6] (12f) to (10a);
                \draw[line width = 0.2mm, bend left=40, color = farbe6] (10a) to (16b);
                \draw[line width= 0.2mm, bend left=20, color = farbe10] (4d) to (8e);
                \draw[line width = 0.2mm, bend left=80, color = farbe10] (8e) to (8f);
                \draw[line width = 0.2mm, bend right=45, color = farbe10] (8f) to (12a);
                \draw[line width = 0.2mm, bend left=70, color = farbe10] (12a) to (10b);
                \draw[line width = 0.2mm, bend left=40, color = farbe10] (10b) to (16c);
                \draw[line width= 0.2mm, bend left=20, color = farbe6] (4e) to (8f);
                \draw[line width = 0.2mm, bend left=80, color = farbe6] (8f) to (8a);
                \draw[line width = 0.2mm, bend right=45, color = farbe6] (8a) to (12b);
                \draw[line width = 0.2mm, bend left=70, color = farbe6] (12b) to (10c);
                \draw[line width = 0.2mm, bend left=40, color = farbe6] (10c) to (16d);
                \draw[line width= 0.2mm, bend left=20, color = farbe10] (4f) to (8a);
                \draw[line width = 0.2mm, bend left=80, color = farbe10] (8a) to (8b);
                \draw[line width = 0.2mm, bend right=45, color = farbe10] (8b) to (12c);
                \draw[line width = 0.2mm, bend left=70, color = farbe10] (12c) to (10d);
                \draw[line width = 0.2mm, bend left=40, color = farbe10] (10d) to (16e);
                \draw[line width= 0.2mm, bend left=20, color = farbe12] (6a) to (12b);
                \draw[line width = 0.2mm, bend left=20, color = farbe12] (12b) to (15c);
                \draw[line width = 0.2mm, bend left=80, color = farbe12] (15c) to (6d);
                \draw[line width = 0.2mm, bend left=20, color = farbe12] (6d) to (12e);
                \draw[line width = 0.2mm, bend left=20, color = farbe12] (12e) to (15f);
                \draw[line width = 0.2mm, bend left=80, color = farbe12] (15f) to (6a);  
                \draw[line width= 0.2mm, bend left=20, color = farbe12] (6b) to (12c);
                \draw[line width = 0.2mm, bend left=20, color = farbe12] (12c) to (15d);
                \draw[line width = 0.2mm, bend left=80, color = farbe12] (15d) to (6e);
                \draw[line width = 0.2mm, bend left=20, color = farbe12] (6e) to (12f);
                \draw[line width = 0.2mm, bend left=20, color = farbe12] (12f) to (15a);
                \draw[line width = 0.2mm, bend left=80, color = farbe12] (15a) to (6b);
                \draw[line width= 0.2mm, bend left=20, color = farbe12] (6c) to (12d);
                \draw[line width = 0.2mm, bend left=20, color = farbe12] (12d) to (15e);
                \draw[line width = 0.2mm, bend left=80, color = farbe12] (15e) to (6f);
                \draw[line width = 0.2mm, bend left=20, color = farbe12] (6f) to (12a);
                \draw[line width = 0.2mm, bend left=20, color = farbe12] (12a) to (15b);
                \draw[line width = 0.2mm, bend left=80, color = farbe12] (15b) to (6c);
              \foreach \point in {1a, 1b, 1c, 1d, 1e, 1f, 2a, 2b, 2c, 2d, 2e, 2f, 3a, 3b, 3c, 3d, 3e, 3f, 4a, 4b, 4c, 4d, 4e, 4f, 5a, 5b, 5c, 5d, 5e, 5f, 6a, 6b, 6c, 6d, 6e, 6f, 7a, 7b, 7c, 7d, 7e, 7f, 8a, 8b, 8c, 8d, 8e, 8f, 9a, 9b, 9c, 9d, 9e, 9f, 10a, 10b, 10c, 10d, 10e, 10f, 11a, 11b, 11c, 11d, 11e, 11f, 12a, 12b, 12c, 12d, 12e, 12f, 13a, 13b, 13c, 13d, 13e, 13f, 14a, 14b, 14c, 14d, 14e, 14f, 15a, 15b, 15c, 15d, 15e, 15f, 16a, 16b, 16c, 16d, 16e, 16f} {
                \draw[fill=black, thick] (\point) circle (0.7mm);
              } 
        \end{tikzpicture}
            \end{align*}
            \caption{The generalized quadrangle of order (5,3), the dual of $Q$.}\label{ganzes duales Viereck}
        \end{figure}

\section{Chamber-regular actions on $\tilde C_2$-buildings}\label{sec:classification}

    In this section we classify lattices acting type-preservingly and chamber-regularly on $\tilde C_2$-buildings whose links of special vertices are the unique generalized quadrangle $Q$ of order (3,5). 
    Such an action $(\Gamma, X)$ arises from a triangle of groups and let $T$ be a triangle of groups arising from such an action.
    Since the action $(\Gamma, X)$ is regular on chambers, $T$ has trivial face group and in particular the local actions are edge-regular.
    In fact, the local actions in $T$ are isomorphic to the actions of vertex stabilizers in $\Gamma$ on the corresponding vertex links.
    Since $X$ is a $\tilde C_2$-building, two of the local actions are chamber-regular actions on~$Q$, and the third local action must be a chamber-regular action on a complete bipartite graph. Since the valencies in $Q$ are four and six, the complete bipartite graph is either $K_{4,4}$ or $K_{6,6}$.
    
    Now assume that $((V_i, \Lambda_i))_i$ is a triple consisting of edge-regular actions, and moreover, assume that two of the three local actions are on $Q$ and the third one is on $K_{4,4}$ or $K_{6,6}$.
    For each action $(V_i, \Lambda_i)$ choose subgroups $E^i_A, E^i_B$ such that the action of $V_i$ on the associated coset graph is isomorphic to the initial action.
    Assume that $((V_i, E^i_A, E^i_B))_i$ is compatible with respect to a matching $(f_i)_i$.
    Let $T$ be a triangle of groups with local action datum $((V_i, E^i_A, E^i_B))_i$ matched along $(f_i)_i$. Then $T$ is non-positively curved and by Proposition \ref{prop:local_approach} its fundamental group is a type-preserving, chamber-regular lattice on a $\tilde C_2$-building.

    Therefore, we can classify the triangles of groups that give rise to actions fulfilling our assumptions by proceedings as follows.
    We list all possible local actions and work out which ones form compatible triples with respect to some matching. 
    Then for each compatible triple $((V_i, E^i_A, E^i_B))_i$ with respect to the matching $(f_i)_i$, we count the number of isomorphism classes of triangles of groups with local action datum $((V_i, E^i_A, E^i_B))_i$ along the matching $(f_i)_i$.
    
    The Lemmas \ref{lem:reg_actions_on_Q}, \ref{lem:reg_actions_on_K44}, \ref{lem:reg_actions_on_K66} below capture all chamber-regular actions on the quadrangle $Q$ and the complete bipartite graphs $K_{4,4}$ and $K_{6,6}$.
    The Tables \ref{tbl:presentations_quadrangle}, \ref{tbl:K44}, \ref{tbl:K66} these lemmas are referring to contain group presentations with generators labeled by $a_i$ and $b_i$. The groups $F_A := \langle a_i \rangle$ and $F_B := \langle b_i \rangle$ encode the action as follows: they intersect trivially and generate the group defined by the presentation. The associated coset graph is either $Q$, $K_{4,4}$ or $K_{6,6}$ depending on the table. In each case the groups $F_A$ and $F_B$ are of order four or six. If one of them is cyclic, say $F_A$, the presentation only contains the generator $a_1$ and otherwise we have generators $a_1, a_2$. Similarly, we may have only the generator $b_1$ or two generators $b_1, b_2$ of the second type.
    The tables have been calculated with a computer \cite{ChamberRegularCode}.

    \begin{example}
        The following presentations encode chamber-regular actions on $Q$ and $K_{4,4}$ respectively. These are the local actions in the triangle of groups from Example \ref{exp:presentation_lattice}.
        \begin{align*}
            L_1 :=
            \big \langle
            a_1,b_1 \; \big | \; &
            a_1^4,
            \quad b_1^6,
            \quad a_1 b_1^{-1} a_1^{-1} b_1^2 a_1^{-1} b_1 a_1^{-2} b_1,
            \quad a_1 b_1 a_1^{-2} b_1^{-1} a_1^{-1} b_1 a_1^{-2} b_1^{-1},\\
            & a_1 b_1 a_1 b_1^{-1} a_1 b_1^{-2} a_1^{-1} b_1^2,
            \big \rangle,\\[2mm]
            L_{12}:=
            \big \langle
            a_1,b_1 \; \big | \; &
            a_1^4, \quad
            b_1^4, \quad
            a_1 b_1 a_1^{-1} b_1^{-1}
            \big \rangle.
        \end{align*}
    \end{example}
    
    \begin{lemma}\label{lem:reg_actions_on_Q}
        Let $Q$ be the unique generalized quadrangle of order (3,5). Let $G$ be a group of automorphisms of $Q$ that is regular on the edges. Then the action $(G, Q)$ is isomorphic to the action of one of the groups $L_1,\dots, L_{11}$, indicated in Table \ref{tbl:presentations_quadrangle} in Appendix \ref{app:presentations}, on its coset graph with respect to the subgroups $\langle a_i \rangle$ of order four and $\langle b_i \rangle$ of order six. Furthermore, these actions are pairwise not isomorphic.
    \end{lemma}

    \begin{lemma}\label{lem:reg_actions_on_K44}
        Let $K_{4,4}$ be the complete bipartite graph with partition size $(4,4)$. Let $G$ be a group of automorphisms of $K_{4,4}$ that is regular on the edges. Then the action $(G, K_{4,4})$ is isomorphic to the action of one of the groups $L_{12},\dots, L_{21}$, indicated in Table \ref{tbl:K44} in Appendix \ref{app:presentations}, on its coset graph with respect to the subgroups $\langle a_i \rangle$ and $\langle b_i \rangle$, which are both of order four. Furthermore, these actions are pairwise not isomorphic.
    \end{lemma}

    \begin{lemma}\label{lem:reg_actions_on_K66}
        Let $K_{6,6}$ be the complete bipartite graph with partition size $(6,6)$. Let $G$ be a group of automorphisms of $K_{6,6}$ that is regular on the edges. Then the action $(G, K_{6,6})$ is isomorphic to the action of one of the groups $L_{22},\dots, L_{35}$, indicated in Table \ref{tbl:K66} in Appendix \ref{app:presentations}, on its coset graph with respect to the subgroups $\langle a_i \rangle$ and $\langle b_i \rangle$, which are both of order six. Furthermore these actions are pairwise not isomorphic.
    \end{lemma}
    
    \begin{remark}
        While the groups $L_1, \dots, L_{11}$ acting on $Q$ are pairwise non-isomorphic, several of the groups $L_{12}, \dots, L_{35}$ acting on $K_{4,4}$ and $K_{6,6}$ are indeed isomorphic. In fact, we have that $L_{13} \cong L_{18} \cong L_{19}$, $L_{15} \cong L_{16} \cong L_{21}$, $L_{22} \cong L_{26} \cong L_{27} \cong L_{28} \cong L_{30} \cong L_{31}$, $L_{23} \cong L_{24} \cong L_{32} \cong L_{33} \cong L_{34}$ and $L_{29} \cong L_{35}$. For the groups $L_{12}, \dots, L_{35}$ we provide a short description of the isomorphism type in Table \ref{tbl:descriptionL12L35} in Appendix \ref{app:more_data}.
    \end{remark}

    For $1\leq r \leq 35$, we call the subgroups $\langle a_i \rangle$ and $\langle b_i \rangle$ edge groups of $L_r$ and denote them by $F^r_A$ and $F^r_B$ respectively.
    
    As in Subsection \ref{subsec:given_local} we fix abstract groups isomorphic to the edge groups that occur. We just denote them by their isomorphism type and will refer to them as model edge groups. 
    \begin{align*}
        C_4 & := \langle e_1 \mid e_1^4 \rangle,
        &
        C_2\times C_2 & := \langle e_1, e_2 \mid e_1^2, e_2^2, (e_1e_2)^2 \rangle,
        \\
        C_6 & := \langle e_1 \mid e_1^6 \rangle,
        &
        \Sym(3) & := \langle e_1, e_2 \mid e_1^2, e_2^3, (e_1e_2)^2 \rangle.
    \end{align*}

    The presentations of the $L_r$ are chosen in such a way that, for the unique model group that is isomorphic to $F_A^r$, respectively $F_B^r$, the assignments $e_i \mapsto a_i$, respectively $e_i \mapsto b_i$, define monomorphisms from the model group to $L_r$. We denote these monomorphisms by $\kappa_A^r$ and $\kappa_B^r$ respectively and refer to them as standard embeddings.
    We make a quick observation that will be needed later.

    \begin{observation}\label{obs:swap_edge_grps}
        The group $L_r$ admits a group automorphism that swaps its edge groups if and only if $r \in \{12,13,16,17,22,25,26,29\}$. In fact, if $r$ is in this list the assignment $a_r \leftrightarrow b_r$ on the generators extends to an involutory automorphism of $L_i$.
    \end{observation}

    We will list the possible triangles of groups using the following notation.

    \begin{notation}
        \begin{enumerate}
            \item 
            Let $T$ be the triangle of groups indicated in Figure \ref{fig:RightAngledTriangleOfGroups} equipped with the type function indicated by notation. Then we denote this triangle with $\textbf{T} (V_1,\allowbreak V_2,\allowbreak V_3,\allowbreak E_1,\allowbreak E_2,\allowbreak E_3,\allowbreak (\epsilon_{ij})_{ij})$. 
            \item
            Let $(r,s,t)$ be a triple with $1\leq r,s \leq 11$, $12 \leq t \leq 21$ and such that
            $E_t \cong F^r_B \cong F^s_B$,
            $E_s \cong F^r_A \cong F^t_B$ and
            $E_r \cong F^s_A \cong F^t_A$, where $E_r,E_s, E_t$ are model edge groups.
            We denote the family of triangles $\textbf{T}(L_r, L_s, L_t, E_r, E_s, E_t, (\epsilon_{ij})_{ij})$ satisfying the conditions below with~$\mathcal T_{(r,s,t)}^{(1)}$, see Figure \ref{fig:T1}.
            \begin{align*}
                \im(\epsilon_{12}) & = F^s_A, &
                \im(\epsilon_{13}) & = F^t_A, &
                \im(\epsilon_{21}) & = F^r_A, \\
                \im(\epsilon_{23}) & = F^t_B, &
                \im(\epsilon_{31}) & = F^r_B, &
                \im(\epsilon_{32}) & = F^s_B.
            \end{align*}
            \item
            Let $(r,s,t)$ be a triple with $1\leq r,s \leq 11$, $22 \leq t \leq 35$ and such that
            $E_t \cong F^r_A \cong F^s_A$,
            $E_s \cong F^r_B \cong F^t_B$ and
            $E_r \cong F^s_B \cong F^t_A$, where $E_r,E_s, E_t$ are model edge groups.
            We denote the family of triangles $\textbf{T}(L_r, L_s, L_t, \dots, (\epsilon_{ij})_{ij})$ satisfying the conditions below with~$\mathcal T_{(r,s,t)}^{(2)}$, see Figure \ref{fig:T2}.
            \begin{align*}
                \im(\epsilon_{12}) & = F^s_B, &
                \im(\epsilon_{13}) & = F^t_A, &
                \im(\epsilon_{21}) & = F^r_B, \\
                \im(\epsilon_{23}) & = F^t_B, &
                \im(\epsilon_{31}) & = F^r_A, &
                \im(\epsilon_{32}) & = F^s_A.
            \end{align*}
        \end{enumerate}
    \end{notation}
    
    \begin{figure}
        \centering
        \begin{tikzpicture}
        \def\outerRadius{4cm}
        \coordinate (V2) at (270:\outerRadius);
        \coordinate (V3) at (0:\outerRadius);
        \coordinate (V1) at (180:\outerRadius);
        \coordinate (E2) at ($(V2)!0.5!(V1)$);;
        \coordinate (E3) at ($(V2)!0.5!(V3)$);
        \coordinate (E1) at ($(V3)!0.5!(V1)$);
        \coordinate (A) at ($(E3)!0.5!(V3)$); 
        \node (V3) at (270:\outerRadius){$V_3$};
        \node (V1) at (0:\outerRadius){$V_1$};
        \node (V2) at (180:\outerRadius){$V_2$};
        \node (E3) at ($(V2)!0.5!(V1)$){$E_3$};
        \node (E1) at ($(V2)!0.5!(V3)$){$E_1$};
        \node (E2) at ($(V3)!0.5!(V1)$){$E_2$};
        \node (A) at ($(E3)!0.5!(V3)$){$1$}; 
        \draw[->] (E3) to node[above] {\footnotesize $\epsilon_{32}$} (V2);
        \draw[->] (E3) to node[above] {\footnotesize $\epsilon_{31}$} (V1);
        \draw[->] (E1) to node[midway, sloped, below] {\footnotesize$\epsilon_{13}$} (V3);
        \draw[->] (E1) to node[midway, sloped, below] {\footnotesize$\epsilon_{12}$} (V2);
        \draw[->] (E2) to node[midway, sloped, below] {\footnotesize$\epsilon_{23}$} (V3);
        \draw[->] (E2) to node[midway, sloped, below] {\footnotesize$\epsilon_{21}$} (V1);
        \draw[->] (A) to (E3);
        \draw[->] (A) to (E1);
        \draw[->] (A) to (E2);
        \end{tikzpicture}
        \caption{The triangle of groups $\textbf{T} (V_1,\allowbreak V_2,\allowbreak V_3,\allowbreak E_1,\allowbreak E_2,\allowbreak E_3,\allowbreak (\epsilon_{ij})_{ij})$.}
        \label{fig:RightAngledTriangleOfGroups}
    \end{figure}
     \begin{figure}
        \centering
        \begin{tikzpicture}
        \def\outerRadius{4cm}
        \def\mRadius{3.75cm}
        \coordinate (V2) at (270:\outerRadius);
        \coordinate (V3) at (0:\outerRadius);
        \coordinate (V1) at (180:\outerRadius);
        \coordinate (V12) at ($(V2)!0.85!(V1)$);
        \node [rotate=-45](V12) at ($(V2)!0.85!(V1)$) {$\ge F^s_A$};
        \coordinate (V21) at ($(V2)!0.8!(V3)$);
        \node [rotate=45](V21) at ($(V2)!0.85!(V3)$) {$F^r_A\le $};
        \coordinate (V23) at ($(V3)!0.8!(V2)$);
        \node [rotate=45](V23) at ($(V3)!0.85!(V2)$) {$\ge F^t_B $};
        \coordinate (V13) at ($(V1)!0.8!(V2)$);
        \node [rotate=-45](V13) at ($(V1)!0.85!(V2)$) {$F^t_A\le  $};
        \coordinate (E2) at ($(V2)!0.5!(V1)$);;
        \coordinate (E3) at ($(V2)!0.5!(V3)$);
        \coordinate (E1) at ($(V3)!0.5!(V1)$);
        \coordinate (A) at ($(E3)!0.5!(V3)$); 
        \coordinate (V31) at ($(V1)!0.85!(V3)$);
        \node (V31) at ($(V1)!0.9!(V3)$) {$F^r_B\le $};
        \coordinate (V32) at ($(V3)!0.85!(V1)$);
        \node (V32) at ($(V3)!0.9!(V1)$) {$\ge F^s_B $};
        \node (V3) at (270:\outerRadius){$L_t$};
        \node (V1) at (0:\outerRadius){$L_r$};
        \node (V2) at (180:\outerRadius){$L_s$};
        \node (E3) at ($(V2)!0.5!(V1)$){$E_t$};
        \node[rotate=-45](E1) at ($(V2)!0.5!(V3)$){$ E_r$};
        \node (E2)[rotate=45] at ($(V3)!0.5!(V1)$){$ E_s$};
        \node (A) at ($(E3)!0.5!(V3)$){$1$}; 
        \draw[->] (E3) to node[above] {\footnotesize $\epsilon_{32}$} (V32);
        \draw[->] (E3) to node[above] {\footnotesize $\epsilon_{31}$} (V31);
        \draw[->] (E1) to node[midway, sloped, below] {\footnotesize$\epsilon_{13}$} (V13);
        \draw[->] (E1) to node[midway, sloped, below] {\footnotesize$\epsilon_{12}$} (V12);
        \draw[->] (E2) to node[midway, sloped, below] {\footnotesize$\epsilon_{23}$} (V23);
        \draw[->] (E2) to node[midway, sloped, below] {\footnotesize$\epsilon_{21}$} (V21);
        \draw[->] (A) to (E3);
        \draw[->] (A) to (E1);
        \draw[->] (A) to (E2);
        \end{tikzpicture}
        \caption{A triangle of groups in~$\mathcal T_{(r,s,t)}^{(1)}$.}
        \label{fig:T1}
    \end{figure}
    \begin{figure}
        \centering
        \begin{tikzpicture}
        \def\outerRadius{4cm}
        \def\mRadius{3.75cm}
        \coordinate (V2) at (270:\outerRadius);
        \coordinate (V3) at (0:\outerRadius);
        \coordinate (V1) at (180:\outerRadius);
        \coordinate (V12) at ($(V2)!0.8!(V1)$);
        \node [rotate=-45](V12) at ($(V2)!0.85!(V1)$) {$\ge F^s_B$};
        \coordinate (V21) at ($(V2)!0.8!(V3)$);
        \node [rotate=45](V21) at ($(V2)!0.85!(V3)$) {$F^r_B\le $};
        \coordinate (V23) at ($(V3)!0.8!(V2)$);
        \node [rotate=45](V23) at ($(V3)!0.85!(V2)$) {$\ge F^t_B $};
        \coordinate (V13) at ($(V1)!0.8!(V2)$);
        \node [rotate=-45](V13) at ($(V1)!0.85!(V2)$) {$F^t_A\le  $};
        \coordinate (E2) at ($(V2)!0.5!(V1)$);;
        \coordinate (E3) at ($(V2)!0.5!(V3)$);
        \coordinate (E1) at ($(V3)!0.5!(V1)$);
        \coordinate (A) at ($(E3)!0.5!(V3)$); 
        \coordinate (V31) at ($(V1)!0.85!(V3)$);
        \node (V31) at ($(V1)!0.9!(V3)$) {$F^r_A\le $};
        \coordinate (V32) at ($(V3)!0.85!(V1)$);
        \node (V32) at ($(V3)!0.9!(V1)$) {$\ge F^s_A $};
        \node (V3) at (270:\outerRadius){$L_t$};
        \node (V1) at (0:\outerRadius){$L_r$};
        \node (V2) at (180:\outerRadius){$L_s$};
        \node (E3) at ($(V2)!0.5!(V1)$){$E_t$};
        \node[rotate=-45](E1) at ($(V2)!0.5!(V3)$){$ E_r$};
        \node (E2)[rotate=45] at ($(V3)!0.5!(V1)$){$ E_s$};
        \node (A) at ($(E3)!0.5!(V3)$){$1$}; 
        \draw[->] (E3) to node[above] {\footnotesize $\epsilon_{32}$} (V32);
        \draw[->] (E3) to node[above] {\footnotesize $\epsilon_{31}$} (V31);
        \draw[->] (E1) to node[midway, sloped, below] {\footnotesize$\epsilon_{13}$} (V13);
        \draw[->] (E1) to node[midway, sloped, below] {\footnotesize$\epsilon_{12}$} (V12);
        \draw[->] (E2) to node[midway, sloped, below] {\footnotesize$\epsilon_{23}$} (V23);
        \draw[->] (E2) to node[midway, sloped, below] {\footnotesize$\epsilon_{21}$} (V21);
        \draw[->] (A) to (E3);
        \draw[->] (A) to (E1);
        \draw[->] (A) to (E2);
        \end{tikzpicture}
        \caption{A triangle of groups in~$\mathcal T_{(r,s,t)}^{(2)}$.}
        \label{fig:T2}
    \end{figure}
    \begin{remark}\label{rem:how_to_type_triple}
        Consider a triple $(r,s,t)$ and type groups as follows
        \begin{align*}
            V_1 &:= L_r, & E_A^1 &:= F_A^r, & E_B^1 & := F_B^r, \\
            V_2 &:= L_s, & E_A^2 &:= F_A^s, & E_B^2 & := F_B^s, \\
            V_3 &:= L_t, & E_A^3 &:= F_A^t, & E_B^3 & := F_B^t. &
        \end{align*}
        Then the family $\mathcal T^{(1)}_{(r,s,t)}$ is a (possibly redundant) representative system for triangles of groups with local action datum $((V_i, E_A^i, E_B^i))_i$ matched along 
        \begin{align*}
            f_1 &: 2 \mapsto A,\; 3 \mapsto B,&
            f_2 &: 1 \mapsto A,\; 3 \mapsto B,&
            f_3 &: 1 \mapsto A,\; 2 \mapsto B.&
        \end{align*}
        Similarly the family $\mathcal T^{(2)}_{(r,s,t)}$ is a (possibly redundant) representative system for triangles of groups with local action datum $((V_i, E_A^i, E_B^i)_i)$ matched along 
        \begin{align*}
            f_1 &: 2 \mapsto B,\; 3 \mapsto A,&
            f_2 &: 1 \mapsto B,\; 3 \mapsto A,&
            f_3 &: 1 \mapsto A,\; 2 \mapsto B.&
        \end{align*}
    \end{remark}

    By the next observation, it is sufficient to restrict to triangles of groups appearing in the families $\mathcal T_{(r,s,t)}^{(\nu)}$ for $\nu \in \{1,2\}$ and a triple $(r,s,t)$.
    
    \begin{observation}\label{obs:every_triangle_in_model_family}
        Let $(r,s,t)$ be a triple with $1\leq r,s \leq 11$ and $t \geq 12$ and type the associated groups as in Remark \ref{rem:how_to_type_triple}.
        Assume that the triple $((V_i, E_A^i, E_B^i))_i$ is compatible with respect to a matching $(f_i)_i$. Then by just considering the orders of the edge groups we have that either $f_1 : 2 \mapsto A,\; 3 \mapsto B$ and $f_2 : 1 \mapsto A,\; 3 \mapsto B$, or $f_1 : 2 \mapsto B,\; 3 \mapsto A$ and $f_2 : 1 \mapsto B,\; 3 \mapsto A$. Moreover, by interchanging the roles of $r$ and $s$, one can ensure that in both cases $f_3 : 1 \mapsto A,\; 2 \mapsto B$.
        In particular any triangle of groups with local action datum $((V_i, E_A^i, E_B^i))_i$ is isomorphic to a triangle of groups in $\mathcal T^{(1)}_{(r',s',t)}$ or $\mathcal T^{(2)}_{(r',s',t)}$, where $\{r',s'\} = \{r,s\}$.
    \end{observation}

    \begin{observation}\label{obs:model_families_isomorphisms}
        Let $\nu \in \{1,2\}$.
        Let $(r,s,t)$, $(r', s', t')$ be a triples with $1\leq r,r',s, s' \leq 11$ and $t,t' \geq 12$ and assume that the families $\mathcal T_{(r,s,t)}^{(\nu)}$ and $\mathcal T_{(r',s',t')}^{(\nu)}$ are defined. 
        Let $T \in \mathcal T_{(r,s,t)}^{(\nu)}$, let $T' \in \mathcal T_{(r',s',t')}^{(\nu)}$ and assume that $T$ and $T'$ are isomorphic via an isomorphism $\phi$. If $\phi$ is type-preserving, then we have that $(r,s,t) = (r',s',t')$. If $\phi$ is not type-preserving then $(r,s,t) = (s',r', t')$ and moreover $L_t$ must admit an automorphism that swaps its edge groups and therefore $t$ lies in the list in Observation \ref{obs:swap_edge_grps}.
        Conversely if $(r,s,t) = (s',r',t')$ and $t$ lies in the list in Observation \ref{obs:swap_edge_grps}, then for any $T \in \mathcal T_{(r,s,t)}^{(\nu)}$ there is an isomorphic $T' \in \mathcal T_{(r,s,t)}^{(\nu)}$.
    \end{observation}

    \begin{lemma}\label{lem:case446}
        Let $T$ be a triangle of groups contained in a family $T_{(r,s,t)}^{(1)}$ for some triple $(r,s,t)$. Then an isomorphic copy of $T$ lies in exactly one of the following 163 families.
        \begin{enumerate}
            \item $\mathcal T_{(1,1,t)}^{(1)}$ for $t \in \{12,13,14\}$ (three families),
            \item $\mathcal T_{(r,s,t)}^{(1)}$ for $r,s \in \{6,9\}, r\leq s$, $t \in \{12,13,14\}$ (nine families),
            \item $\mathcal T_{(9,6,14)}^{(1)}$ (one family),
            \item $\mathcal T_{(r,s,t)}^{(1)}$ for $r,s \in \{3,5,7,8\}, r\leq s$, $t \in \{15,16,17\}$ (30 families),
            \item $\mathcal T_{(r,s,15)}^{(1)}$ for $r,s \in \{3,5,7,8\}, s < r$ (six families).
            \item $\mathcal T_{(r,s,t)}^{(1)}$ for $r,s \in \{2,4,10,11\}, r\leq s$, $t \in \{15,16,17\}$ (30 families),
            \item $\mathcal T_{(r,s,15)}^{(1)}$ for $r,s \in \{2,4,10,11\}, s<r$ (six families).
            \item $\mathcal T_{(1,s,t)}^{(1)}$ for $s \in \{3,5,7,8\}$, $t \in \{18,19,20,21\}$ (16 families),
            \item $\mathcal T_{(r,s,t)}^{(1)}$ for $r\in \{2,4,10,11\}$, $s \in \{6,9\}$, $t \in \{18,19,20,21\}$ (32 families).
        \end{enumerate}
    \end{lemma}

    \begin{lemma}\label{lem:case664}
        Let $T$ be a triangle of groups contained in a family $T_{(r,s,t)}^{(2)}$ for some triple $(r,s,t)$. Then an isomorphic copy of $T$ lies in exactly one of the following 232 families.
        \begin{enumerate}
            \item $\mathcal T_{(1,1,t)}^{(2)}$ for $t \in \{22,23,24,25\}$ (four families),
            \item $\mathcal T_{(r,s,t)}^{(2)}$ for $r,s \in \{3,5,7,8\}, r \leq s$, $t \in \{22,23,24,25\}$ (40 families),
            \item $\mathcal T_{(r,s,t)}^{(2)}$ for $r,s \in \{3,5,7,8\}, s < r$, $t \in \{23,24\}$ (12 families),
            \item $\mathcal T_{(r,s,t)}^{(2)}$ for $r,s \in \{6,9\}, r \leq s$, $t \in \{26,27,28,29\}$ (12 families),
            \item $\mathcal T_{(9,6,t)}^{(2)}$ for $t \in \{27,28\}$ (two families),
            \item $\mathcal T_{(r,s,t)}^{(2)}$ for $r,s \in \{2,4,10,11\}, r \leq s$, $t \in \{26,27,28,29\}$ (40 families),
            \item $\mathcal T_{(r,s,t)}^{(2)}$ for $r,s \in \{2,4,10,11\}, s < r$, $t \in \{27,28\}$ (12 families),
            \item $\mathcal T_{(1,s,t)}^{(1)}$ for $s \in \{6,9\}$, $t \in \{30,31,32,33,34,35\}$ (12 families).
            \item $\mathcal T_{(r,s,t)}^{(1)}$ for $r\in \{3,5,7,8\}$, $s \in \{2,4,10,11\}$, $t \in \{30,31,32,33,34,35\}$ (96 families).
        \end{enumerate}
    \end{lemma}

    \begin{proof}[Proof of Lemmas \ref{lem:case446}, \ref{lem:case664}]
        To prove the two lemmas one just lists all triples $(r,s,t)$ for which the family $\mathcal T_{(r,s,t)}^{(1)}$ or $\mathcal T_{(r,s,t)}^{(2)}$ is defined, and afterwords one removes redundant triples using Observation \ref{obs:model_families_isomorphisms}. These calculations can be easily performed by hand.
    \end{proof}

    We will now determine the number of type-preserving isomorphism classes of triangles of groups in each family from the Lemmas \ref{lem:case446} and \ref{lem:case664}. To do so we will use Observation \ref{obs:double_cosets} and introduce some more necessary notation.
    If we have the group $L_r$ and the local automorphism group $\Sigma(L_r, F_A^r, F_A^r)$, then by identifying the edge groups of $L_r$ with the isomorphic model edge groups $E_A, E_B$ via the standard embeddings we obtain a group $\Sigma^r \leq \Aut(E_A)\times \Aut(E_B)$ equipped with projections $\pi_A^r : \Sigma^r \to \Aut(E_A)$ and $\pi_B^r: \Sigma^r \to \Aut(E_B)$.
    We call $\Sigma^r$ also local automorphism group of $L_r$.
    Often we have that $\Sigma^r = \im(\pi^r_A) \times \im(\pi^r_B)$ and in that case we call $\Sigma^r$ decomposable and denote $\im(\pi^r_A)$ and $\im(\pi^r_B)$ with $\Sigma^r_A$ and $\Sigma^r_B$ respectively.
    If $\Sigma^r$ is decomposable, we also call the corresponding action decomposable.
    We computed the groups $\Sigma^r$ and listed them in the Tables \ref{tbl:sigmasQ}, \ref{tbl:sigmasK44}, \ref{tbl:sigmasK66} in Appendix \ref{app:more_data}. 
    One may verify the results with code provided in \cite{ChamberRegularCode}.
    We point out the following.
    
    \begin{observation}
        For $1 \leq i \leq 35$ the local automorphism group $\Sigma^i$ is indecomposable if and only if $i \in \{24, 28,31,33\}$.
    \end{observation}

    In particular, all edge-regular actions on $K_{4,4}$ and $Q$ are decomposable. Therefore, all families from Lemma \ref{lem:case446} and most families from Lemma \ref{lem:case664} only contain triangles of groups with only decomposable local actions. In that case there is a convenient minimal representative system for the double cosets corresponding to the isomorphism classes in the family.

    \begin{lemma}\label{lem:dec_double_cosets}
        \begin{enumerate}
            \item
                Let $\mathcal T = \mathcal T_{(r,s,t)}^{(1)}$ be a family from Lemma \ref{lem:case446} and let $R_r, R_s, R_t$ be minimal representative systems for the following double cosets spaces
                \begin{align*}
                    \Sigma_A^t \backslash \Aut(E_r) / \Sigma_A^s, &&
                    \Sigma_A^r \backslash \Aut(E_s) / \Sigma_B^t, &&
                    \Sigma_B^s \backslash \Aut(E_t) / \Sigma_B^r.
                \end{align*}
                Then the following family is a minimal representative system for isomorphism classes of triangles of groups in $\mathcal T$.
                \begin{align*}
                    \{
                        \textbf{T}(L_r, L_s, L_t, E_r, E_s, E_t, (\epsilon_{ij})_{ij})
                        \mid\;
                        & \epsilon_{13} = \kappa_A^t, \;
                        \epsilon_{12} \in \kappa_A^s \circ R_r,\\
                        &\epsilon_{21} = \kappa_A^r, \;
                        \epsilon_{23} \in \kappa_B^t \circ R_s,\\
                        &\epsilon_{32} = \kappa_B^s ,\;
                        \epsilon_{31} \in \kappa_B^r \circ R_t
                    \}.
                \end{align*}
            \item
                Let $\mathcal T = \mathcal T_{(r,s,t)}^{(2)}$ be a family from Lemma \ref{lem:case664} with only decomposable local actions and let $R_r, R_s, R_t$ be minimal representative systems for the following double cosets spaces
                \begin{align*}
                    \Sigma_A^t \backslash \Aut(E_r) / \Sigma_B^s, &&
                    \Sigma_B^r \backslash \Aut(E_s) / \Sigma_B^t, &&
                    \Sigma_A^s \backslash \Aut(E_t) / \Sigma_A^r.
                \end{align*}
                Then the following family is a minimal representative system for isomorphism classes of triangles of groups in $\mathcal T$.
                \begin{align*}
                    \{
                        \textbf{T}(L_r, L_s, L_t, E_r, E_s, E_t, (\epsilon_{ij})_{ij})
                        \mid\;
                        & \epsilon_{13} = \kappa_A^t, \;
                        \epsilon_{12} \in \kappa_B^s \circ R_r,\\
                        &\epsilon_{21} = \kappa_B^r, \;
                        \epsilon_{23} \in \kappa_B^t \circ R_s,\\
                        &\epsilon_{32} = \kappa_A^s ,\;
                        \epsilon_{31} \in \kappa_A^r \circ R_t
                    \}.
                \end{align*}
        \end{enumerate}
    \end{lemma}

    \begin{proof}
        Let $\mathcal T_{(r,s,t)}^{(1)}$ be one of the considered families. Then $\mathcal T_{(r,s,t)}^{(1)}$ can be identified with the product
        $C = \Aut(E_r)^2 \times \Aut(E_s)^2 \times \Aut(E_t)^2$
        by identifying the triangles of groups with monomorphisms $(\epsilon_{ij})_{ij}$ with the tuple $(\gamma_{12},\allowbreak \gamma_{13},\allowbreak \gamma_{21},\allowbreak \gamma_{23},\allowbreak \gamma_{31},\allowbreak \gamma_{32})$, where $\gamma_{ij}$ is the unique automorphism of the corresponding model edge group with $\kappa_{ij} \circ \gamma_{ij} = \epsilon_{ij}$.
        Let $D_r \cong \Aut(E_r), D_s\cong \Aut(E_s), D_t \cong \Aut(E_t)$ be the diagonal subgroups of the first two, the middle two and the last two components respectively. Now let $R$ be a minimal representative system for the following double coset space.
        \[
            D_r \times D_s \times D_t
            \backslash C /
            \Sigma_A^s\times \Sigma_A^t \times \Sigma_A^r \times \Sigma_B^t \times \Sigma_B^r \times \Sigma_B^s.
        \]
        Now by Observation \ref{obs:double_cosets} such system $R$ corresponds to a minimal representative system of triangles of groups in $\mathcal T_{(r,s,t)}^{(1)}$ up to type-preserving isomorphism. Clearly $R_r \times 1 \times 1 \times R_s \times R_t \times 1$ is such a system, which proves the first statement. The same proof with adjusted indices proves the second statement.
    \end{proof}
    
    Note that in our model edge groups, every double coset space has size 1, 2, 3 or 6, which makes the calculations straightforward to carry out by hand.
    Applying Lemma \ref{lem:dec_double_cosets} to the families in \ref{lem:case446} yields:
    \begin{lemma}\label{lem:dec_446}
        \begin{enumerate}
            \item Each of the three families $\mathcal T_{(1,1,t)}^{(1)}$ for $t \in \{12,13,14\}$ contains two type-preserving isomorphism classes (six classes).
            \item Each of the nine families $\mathcal T_{(r,s,t)}^{(1)}$ for $r,s \in \{6,9\}$, $r\leq s$, $t \in \{12,13,14\}$ and also the family $\mathcal T_{(9,6,14)}^{(1)}$ contain six type-preserving isomorphism classes (60 classes).
            \item Each of the twenty families $\mathcal T_{(r,s,t)}^{(1)}$ for $r,s \in \{3,5,7,8\}$, $r\leq s$, $t \in \{15,16\}$ and also each of the six families $\mathcal T_{(r,s,15)}^{(1)}$ for $s < r \in \{3,5,7,8\}$ contains nine type-preserving isomorphism classes (234 classes).
            \item Each of the 10 families $\mathcal T_{(r,s,17)}^{(1)}$ for $r,s \in \{3,5,7,8\}$, $r\leq s$ contains one type-preserving isomorphism classes (ten classes).
            \item Each of the twenty families $\mathcal T_{(r,s,t)}^{(1)}$ for $r\leq s \in \{2,4,10,11\}$, $t \in \{15,16\}$ and also each of the six families $\mathcal T_{(r,s,15)}^{(1)}$ for $r,s \in \{2,4,10,11\}$, $s < r$ contains 18 type-preserving isomorphism classes (468 classes).
            \item Each of the 10 families $\mathcal T_{(r,s,17)}^{(1)}$ for $r,s \in \{2,4,10,11\}, r\leq s$ contains two type-preserving isomorphism classes (20 classes).
            \item Each of the 12 families $\mathcal T_{(1,s,t)}^{(1)}$ for $s \in \{3,5,7,8\}$, $t \in \{18,19,21\}$ contains three type-preserving isomorphism classes (36 classes).
            \item Each of the four families $\mathcal T_{(1,s,20)}^{(1)}$ for $s \in \{3,5,7,8\}$ contains one type-preserving isomorphism class (four classes).
            \item Each of the 24 families $\mathcal T_{(r,s,t)}^{(1)}$ for $r\in \{2,4,10,11\}$, $s\in \{6,9\}$, $t \in \{18,19,21\}$ contains nine type-preserving isomorphism classes (216 classes).
            \item Each of the eight families $\mathcal T_{(r,s,20)}^{(1)}$ for $r\in \{2,4,10,11\}$, $s\in \{6,9\}$ contains three type-preserving isomorphism classes (24 classes).   
        \end{enumerate}
    \end{lemma}

    Applying Lemma \ref{lem:dec_double_cosets} to the cases with three decomposable local actions in Lemma \ref{lem:case664} yields:
    \begin{lemma}\label{lem:dec_664}
        \begin{enumerate}
            \item Each of the three families $\mathcal T_{(1,1,t)}^{(2)}$ for $t\in \{22,23,25\}$ contains two type-preserving isomorphism classes (six classes).
            \item Each of the 30 families $\mathcal T_{(r,s,t)}^{(2)}$ for $r,s \in \{3,5,7,8\}$, $r\leq s$, $t\in \{22,23,25\}$ and also each of the six families $\mathcal T_{(r,s,23)}^{(2)}$ for $s,r \in\{3,5,7,8\}$, $s<r$ contains six type-preserving isomorphism classes (216 classes).
            \item Each of the three families $\mathcal T_{(r,s,26)}^{(2)}$ for $r,s \in \{6,9\}$, $r\leq s$ contains two type-preserving isomorphism classes (six classes).
            \item Each of the four families $\mathcal T_{(r,s,27)}^{(2)}$ for $r,s \in \{6,9\}$ contains six type-preserving isomorphism classes (24 classes).
            \item Each of the three families $\mathcal T_{(r,s,29)}^{(2)}$ for $r,s \in \{6,9\}$, $r\leq s$ contains 18 type-preserving isomorphism classes (54 classes).
            \item Each of the 10 families $\mathcal T_{(r,s,26)}^{(2)}$ for $r,s \in \{2,4,10,11\}$, $r\leq s$ contains six type-preserving isomorphism classes (60 classes).
            \item Each of the 16 families $\mathcal T_{(r,s,27)}^{(2)}$ for $r, s \in \{2,4,10,11\}$ contains 12 type-preserving isomorphism classes (192 classes).
            \item Each of the 10 families $\mathcal T_{(r,s,29)}^{(2)}$ for $r,s \in \{2,4,10,11\}$, $r\leq s$ contains 24 type-preserving isomorphism classes (240 classes).
            \item Each of the four families $\mathcal T_{(1,s,t)}^{(2)}$ for $s \in \{6,9\}$, $t\in \{30,34\}$ contains six type-preserving isomorphism classes (24 classes).
            \item Each of the four families $\mathcal T_{(1,s,t)}^{(2)}$ for $s \in \{6,9\}$, $t\in \{32,35\}$ contains two type-preserving isomorphism classes (eight classes).
            \item Each of the 32 families $\mathcal T_{(r,s,t)}^{(2)}$ for $r \in \{3,5,7,8\}$, $s\in \{2,4,10,11\}$, $t\in \{30,34\}$ contains 12 type-preserving isomorphism classes (384 classes).
            \item Each of the 32 families $\mathcal T_{(r,s,t)}^{(2)}$ for $r \in \{3,5,7,8\}$, $s\in \{2,4,10,11\}$, $t\in \{32,35\}$ contains six type-preserving isomorphism classes (192 classes).
        \end{enumerate}
    \end{lemma}

    The cases from Lemma \ref{lem:case664} with one indecomposable local action remain.
    
    \begin{lemma}\label{lem:non_dec}
        \begin{enumerate}
            \item The family $\mathcal T_{(1,1,24)}^{(2)}$ contains four type-preserving isomorphism classes.
            \item Each of the 16 families $\mathcal T_{(r,s,24)}^{(2)}$ for $r,s \in \{3,5,7,8\}$ contains six type-preserving isomorphism classes (96 classes).
            \item Each of the four families $\mathcal T_{(r,s,28)}^{(2)}$ for $r, s \in \{6,9\}$ contains 12 type-preserving isomorphism classes (48 classes).
            \item Each of the 16 families $\mathcal T_{(r,s,28)}^{(2)}$ for $r, s \in \{2,4,10,11\}$ contains 12 type-preserving isomorphism classes (192 classes).
            \item Each of the two families $\mathcal T_{(1,s,31)}^{(2)}$ for $s \in \{6,9\}$ contains 12 type-preserving isomorphism classes (24 classes).
            \item Each of the two families $\mathcal T_{(1,s,33)}^{(2)}$ for $s \in \{6,9\}$ contains four type-preserving isomorphism classes (eight classes).
            \item Each of the 16 families $\mathcal T_{(r,s,31)}^{(2)}$ for $r\in\{3,5,7,8\}$, $s \in \{2,4,10,11\}$ contains 12 type-preserving isomorphism classes (192 classes).
            \item Each of the 16 families $\mathcal T_{(r,s,33)}^{(2)}$ for $r\in\{3,5,7,8\}$, $s \in \{2,4,10,11\}$ contains six type-preserving isomorphism classes (96 classes).
        \end{enumerate}
    \end{lemma}

    \begin{proof}
        Let $\mathcal T_{(r,s,t)}^{(2)}$ be one of the families in the lemma. Then $\mathcal T_{(r,s,t)}^{(2)}$ can be identified with the product
        $\Aut(E_r)^2 \times \Aut(E_s)^2 \times \Aut(E_t)^2$ as in the proof of Lemma \ref{lem:dec_double_cosets}. Note that $\Sigma_A^r = 1$ for $1\leq r \leq 11$. In particular if $M'$ is a minimal representative system for the double cosets in $(*)$ indicated below, then the system $M$ in $(**)$ indicated below corresponds to a minimal representative set of type-preserving isomorphism classes in $\mathcal T_{(r,s,t)}^{(2)}$. In $(*)$ the groups $D_r \cong \Aut(E_r)$ and $D_s \cong \Aut(E_s)$ are the diagonal subgroups in the first two components and the latter two components respectively.
        \[
            (D_r \times D_s)
            \backslash
            (E_r \times E_r \times E_s \times E_s)
            /
            (\Sigma_B^a \times \Sigma^t \times \Sigma_B^r), \tag{$*$}
        \]
        \[
            M = M' \times \Aut(E_t) \times 1. \tag{$**$}
        \]
        The cases 1, 2, 3, 5, 6, 7, 8 in the Lemma are now straightforward. In case 4 we have $E_r=E_s=\Sym(3)$ and $\Sigma_B^s = \Sigma_B^r = \langle \Ad(e_1)\rangle$. The group $\Sigma^t$ is the the diagonal subgroup of $\Aut(\Sym(3))^2$. In particular there are at least two double cosets in $M$ and in fact there are exactly two. Let $t = \Ad(e_1e_2)$. We claim that $M := \{(1,1,1,1), (t,1,1,1)\}$ is a representative system. 
        Let $U$ be the union of the double cosets of the elements in $M$.
        Let $x,y \in \Aut(\Sym(3))$ be arbitrary. Then there is a $z \in \langle \Ad(e_1)\rangle$ such that the element $(yx, z, 1, 1)$ lies in $U$. In particular $(yx,1,1,z) \in U$ but then $(yx,1,1,1) \in U$. Hence $(x,1,y,1)\in U$ and this implies $U=\Aut(\Sym(3))^4$.
    \end{proof}

    Combining the Lemmas \ref{lem:dec_446}, \ref{lem:dec_664} and \ref{lem:non_dec} yields:

    \begin{proposition}
        The triangles that appear in the families $T_{(r,s,t)}^{(1)}$ and $T_{(r,s,t)}^{(2)}$ for some triple $(r,s,t)$ fall into exactly 3144 type-preserving isomorphism classes.
    \end{proposition}

    Now we compute the number of triangles of groups up to (possibly non-type-preserving) isomorphism. From the previous discussion one can easily extract a representative system up to type-preserving isomorphism. This system is separated into families and members of different families are not isomorphic. Note that in most of the families two members can only be isomorphic via a type-preserving isomorphism. This holds exactly for the families $\mathcal T_{(r,s,t)}^{(\nu)}$ with $r\neq s$ or with $t$ not in the list in Observation \ref{obs:swap_edge_grps}. Therefore, for these families the number of type-preserving isomorphism classes is the number of isomorphism classes. In what follows we study the remaining families.

    \begin{lemma}\label{lem:non_type_preserving_numbers}
        \begin{enumerate}
            \item Each of the two families $\mathcal T_{(1,1,t)}^{(1)}$ for $t\in \{12,13\}$ contains two isomorphism classes (four classes).
            \item Each of the four families $\mathcal T_{(r,r,t)}^{(1)}$ for $r \in \{6,9\}$, $t\in \{12,13\}$ contains five isomorphism classes (20 classes).
            \item Each of the four families $\mathcal T_{(r,r,16)}^{(1)}$ for $r \in \{3,5,7,8\}$ contains six isomorphism classes (24 classes).
            \item Each of the four families $\mathcal T_{(r,r,16)}^{(1)}$ for $r \in \{2,4,10,11\}$ contains 12 isomorphism classes (48 classes).
            \item Each of the four families $\mathcal T_{(r,r,17)}^{(1)}$ for $r \in \{2,4,10,11\}$ contains two isomorphism classes (eight classes).
            \item Each of the two families $\mathcal T_{(1,1,t)}^{(2)}$ for $t\in \{22,25\}$ contains two isomorphism classes (four classes).
            \item Each of the eight families $\mathcal T_{(r,r,t)}^{(2)}$ for $r \in \{3,5,7,8\}$, $t\in \{22,25\}$ contains five isomorphism classes (40 classes).
            \item Each of the two families $\mathcal T_{(r,r,26)}^{(2)}$ for $r \in \{6,9\}$ contains two isomorphism classes (four classes).
            \item Each of the two families $\mathcal T_{(r,r,29)}^{(2)}$ for $r \in \{6,9\}$ contains 12 isomorphism classes (24 classes).
            \item Each of the four families $\mathcal T_{(r,r,26)}^{(2)}$ for $r \in \{2,4,10,11\}$ contains five isomorphism classes (20 classes).
            \item Each of the four families $\mathcal T_{(r,r,29)}^{(2)}$ for $r \in \{2,4,10,11\}$ contains 15 isomorphism classes (60 classes).
        \end{enumerate}
    \end{lemma}

    \begin{proof}
        As in the proofs of Lemmas \ref{lem:dec_double_cosets} and \ref{lem:non_dec} we identify the triangles in a family with
        $C = \Aut(E_r)^2 \times \Aut(E_s)^2 \times \Aut(E_t)^2$. 
        Now let $\mathcal T$ be family considered in the lemma and let $T \in \mathcal T$.
        We obtain another (isomorphic) triangle $T' \in \mathcal T$ by replacing the type-function $f$ with $(1\;2)\circ f$ and we call $T'$ the mirror of~$T$. 
        If $T$ corresponds to the tuple $(\gamma_{12},\allowbreak \gamma_{13},\allowbreak \gamma_{21},\allowbreak \gamma_{23},\allowbreak \gamma_{31},\allowbreak \gamma_{32})$ then by Observation \ref{obs:swap_edge_grps} its mirror is type-preservingly isomorphic to the triangle in $\mathcal T$ corresponding to $(\gamma_{21},\allowbreak \gamma_{23},\allowbreak \gamma_{12},\allowbreak \gamma_{13},\allowbreak \gamma_{32},\allowbreak \gamma_{31})$.
        In the cases 1, 2, 6, 7, 8 and 10 in the lemma a minimal representative system up to type-preserving automorphism corresponds to the following set in $C$: 
        $\{(\id,\id,\id,\id,\gamma_{31},\id) \mid \gamma_{31} \in \Aut(E_3)\}$.
        We deduce that the triangle corresponding to $\{(\id,\id,\id,\id,\gamma_{31},\id)$ is (non-type-preservingly) isomorphic to $\{(\id,\id,\id,\id,\gamma_{31}^{-1},\id)$ and therefore the number of isomorphism classes in $\mathcal T$ equals the cardinality of $\{\{\gamma, \gamma^{-1}\} \mid \gamma \in \Aut(E_3)\}$.
        Now let $\mathcal T = \mathcal T^{(1)}_{(r,r,16)}$ for $r \in \{3,5,7,8\}$. Let $R = \{\id, (a_1, a_2), (a_1, a_1a_2)\} \subseteq \Aut(C_2\times C_2)$. Then a minimal representative system for type-preserving isomorphism classes in $\mathcal T$ corresponds to the set
        $\{(\id,\gamma_{13},\id,\gamma_{23},\id,\id) \mid \gamma_{13}, \gamma_{23} \in R\}$.
        Now enumerate $R = \{r_1, r_2, r_3\}$. With the same argumentation we deduce that a minimal representative system for isomorphism classes corresponds to the set $\{(\id,\gamma_{13},\id,\gamma_{23},\id,\id) \mid \gamma_{13} = r_i, \gamma_{23} = r_j, i \leq j\}$.
        Now let $\mathcal T = \mathcal T^{(1)}_{(r,r,16)}$ for $r \in \{2,4,10,11\}$. Let $S = \{\id, \Ad(b_2)\} \subseteq \Aut(\Sym(3)$ and let $R = \{r_1, r_2, r_3\}$ be as before.
        Then a minimal representative system for type-preserving isomorphism classes in $\mathcal T$ corresponds to the set
        $\{(\id,\gamma_{13},\id,\gamma_{23},\gamma_{31},\id) \mid \gamma_{13}, \gamma_{23} \in R, \gamma_{31} \in S\}$. Since $S$ consists of involutions a representative system for isomorphism classes corresponds to the set $\{(\id,\gamma_{13},\id,\gamma_{23}, \gamma_{31},\id) \mid \gamma_{13} = r_i, \gamma_{23} = r_j, i \leq j, \gamma_{31} \in S\}$.
        Now let $\mathcal T = \mathcal T^{(1)}_{(r,r,17)}$ for $r \in \{3,5,7,8\}$. Let $S$ be as before. Then a minimal representative system for type-preserving isomorphism classes corresponds to the set
        $\{(\id,\id,\id,\id,\gamma_{31},\id) \mid \gamma_{31} \in S\}$. Since $S$ consists of involutions this is also a minimal representative system for isomorphism classes.
        Now let $\mathcal T = \mathcal T^{(2)}_{(r,r,29)}$ for $r \in \{6,9\}$. This time we set $R = \{\id, \Ad(a_2), \Ad(a_1a_2)\} \subseteq \Aut(\Sym(3))$ and again we enumerate it $R=\{r_1, r_2, r_3\}$. We set $S = \Aut(C_4)$ Then a minimal representative system for the type-preserving isomorphism classes in $\mathcal T$ corresponds to the set $\{(\id,\gamma_{13},\id,\gamma_{23}, \gamma_{31},\id) \mid \gamma_{13} = r_i, \gamma_{23} = r_j, i \leq j, \gamma_{31} \in S\}$. By similar arguments as before a minimal representative system for isomorphism classes corresponds to $\{(\id,\gamma_{13},\id,\gamma_{23}, \gamma_{31},\id) \mid \gamma_{13} = r_i, \gamma_{23} = r_j, i \leq j, \gamma_{31} \in S\}$.
        For the last case let $\mathcal T = \mathcal T^{(2)}_{(r,r,29)}$ for $r \in \{2,4,10,11\}$. This time let $R = \{ {r_1, r_2} \}=\{\id, \Aut(b_2)\}$ and $S = \Aut(C_2\times C_2)$. Then a minimal representative system for the type-preserving isomorphism classes in $\mathcal T$ corresponds to the set $\{(\id,\gamma_{13},\id,\gamma_{23}, \gamma_{31},\id) \mid \gamma_{13} = r_i, \gamma_{23} = r_j, i \leq j, \gamma_{31} \in S\}$. Note that this time $S$ contains elements that are not involutions. Let $S' \subseteq S$ be a minimal set such that $S' \cup (S')^{-1} = S$. Then a minimal representative system for isomorphism classes corresponds to $\{(\id,\gamma_{13},\id,\gamma_{23}, \gamma_{31},\id) \mid \gamma_{13} = r_i, \gamma_{23} = r_j, i \leq j, \gamma_{31} \in S'\}$. 
    \end{proof}

    So in total there are 3044 isomorphism types of triangles of groups, which induce type-preserving, chamber-regular actions on $\tilde C_2$-buildings with $Q$ as links of special vertices.

    \begin{theorem}
        Up to isomorphism there exist 3044 type-preserving, chamber-regular lattices on a $\tilde C_2$-building with $Q$ as links of special vertices. Each of these groups admits only one such action. Assuming Kantor's conjecture mentioned in the introduction these are the only lattices that act type-preservingly and chamber-regularly on a locally finite $\tilde C_2$-building.
    \end{theorem}

    \begin{proof}
        As explained in the introduction of this section, every such action arises from a triangle of groups, which involves the local actions from the Lemmas \ref{lem:reg_actions_on_Q}, \ref{lem:reg_actions_on_K44}, \ref{lem:reg_actions_on_K66}.
        On the other hand, if a triangle of groups involves exactly two local actions from Lemma \ref{lem:reg_actions_on_Q} and one local action from Lemma \ref{lem:reg_actions_on_K44} or Lemma \ref{lem:reg_actions_on_K66}, then it induces a chamber-regular action on $\tilde C_2$-building by Proposition \ref{prop:local_approach}.
        By Observation \ref{obs:every_triangle_in_model_family} any such triangle is isomorphic to a triangle in a family $T_{(r,s,t)}^{(1)}$ or $T_{(r,s,t)}^{(2)}$ for some triple $(r,s,t)$.
        We counted the number of isomorphism classes of triangles of groups that can appear in such a family: there are 3044 isomorphism classes of triangles of groups.
        Since isomorphic actions yield isomorphic triangles of groups there are at most 3044 actions as in the theorem. By Proposition \ref{prop:uniqueness_dev} there are precisely 3044 actions as in the theorem.
        By Proposition \ref{cor:group_gives_triangles} the acting groups are pairwise not isomorphic.
        
        Assume that Kantor's conjecture holds, i.e. that the only finite non-Moufang quadrangles with chamber-transitive automorphism groups are $Q$ and the quadrangle $Q_{\LS}$ arising from the Lunelli--Sce hyperoval. One checks, for example using a computer, that $Q_{\LS}$ does not admit a chamber-regular action \cite{ChamberRegularCode}. On the other hand, by a theorem of Seitz \cite[Theorem 4.8.7]{VanMaldeghem}, no finite Moufang quadrangle admits a chamber-regular action. Therefore, $Q$ is the only finite generalized quadrangle admitting a chamber-regular action. Since a type-preserving chamber-regular action on a locally finite $\tilde C_2$-building induces chamber-regular actions on the vertex links, the links of special vertices in $\tilde C_2$-buildings admitting such an action must be $Q$.
    \end{proof}

    \appendix
    \section{Presentations of vertex groups}\label{app:presentations}

    \begin{longtable}{ |c|c| }
    \caption{Presentations of chamber-regular groups on $Q$.}
    \label{tbl:presentations_quadrangle}\\
    \hline
    \textbf{Grp} & \textbf{Presentation} \\
    \hline
    $L_1$
    &
    \begin{minipage}{6cm}
        {
        \vspace{-2.5mm}
        \begin{align*}
            \big \langle
            a_1,b_1 \; \big | \;
            & a_1^4, \; b_1^6,
            && a_1 b_1^{-1} a_1^{-1} b_1^2 a_1^{-1} b_1 a_1^{-2} b_1, \\
            & a_1 b_1 a_1^{-2} b_1^{-1} a_1^{-1} b_1 a_1^{-2} b_1^{-1},
            && a_1 b_1 a_1 b_1^{-1} a_1 b_1^{-2} a_1^{-1} b_1^2 
            \big \rangle
        \end{align*}
        \vspace{-5mm}
        }
    \end{minipage}\\
    \hline
    $L_2$
    &
    \begin{minipage}{6cm}
        {
        \vspace{-2.5mm}
        \begin{align*}
            \big \langle
            a_1, a_2, b_1, b_2 \; \big | \;
            & a_1^2, \; a_2^2, \; (a_1a_2)^2,
            && b_1^2, \; b_2^3, \; (b_1b_2)^2, \\
            & a_2 b_2 a_1 b_2^{-1} a_2 b_2^{-1} a_1 b_2, 
            && (a_1 b_2^{-1} a_1 b_2)^2, \\
            & a_2 b_2 a_2 b_1 a_1 b_2^{-1} a_1 b_1, 
            && a_1 a_2 b_2 a_1 b_2 b_1 a_2 b_2 a_1 b_1,\\
            & a_1 b_2 b_1 (a_1b_1b_2^{-1})^3
            \big \rangle
        \end{align*}
        \vspace{-5mm}
        }
    \end{minipage}\\
    \hline
    $L_3$
    &
    \begin{minipage}{6cm}
        {
        \vspace{-2.5mm}
        \begin{align*}
            \big \langle
            a_1, a_2, b_1 \; \big | \;
            & a_1^2, \; a_2^2, \; (a_1 a_2)^2, \; b_1^6,
            && a_1 b_1 a_2 b_1 a_2 b_1^{-1} a_2 b_1^{-1},\\
            & a_2 b_1^{-1} a_2 b_1 a_2 b_1 a_1 b_1^{-1},
            && (a_1 b_1)^4,\\
            & a_1 b_1 a_2 b_1^{-2} a_1 b_1^{-2} a_2 b_1
            \big \rangle
        \end{align*}
        \vspace{-5mm}
        }
    \end{minipage}\\
    \hline
    $L_4$
    &
    \begin{minipage}{6cm}
        {
        \vspace{-2.5mm}
        \begin{align*}
            \big \langle
            a_1, a_2, b_1, b_2
            \; \big | \;
            &
            a_1^2,
            a_2^2,
            (a_1 a_2)^2,
            && b_1^2,
            b_2^3,
            (b_1 b_2)^2,\\
            & (a_2 b_1 a_1 b_1)^2,
            && (a_1 b_1)^4,\\
            & a_2 b_2 a_2 b_1 a_1 b_2^{-1} a_1 b_1,
            && (a_1 b_2^{-1} a_1 b_2)^2, \\
            & a_2 b_2 a_1 b_2^{-1} a_2 a_1 b_2 a_1 b_2^{-1}, 
            && a_1 b_2^{-1} a_2 b_1 b_2 a_2 b_2 a_1 b_1 b_2, \\
            & a_1 b_2 a_2 b_1 b_2^{-1} a_2 b_2^{-1} a_1 b_2 b_1
            \big \rangle
        \end{align*}
        \vspace{-5mm}
        }
    \end{minipage}\\
    \hline
    $L_5$ 
    &
    \begin{minipage}{6cm}
        {
        \vspace{-2.5mm}
        \begin{align*}
            \big \langle
            a_1, a_2, b_1
            \; \big | \;
            &
            a_1^2, \;
            a_2^2, \; 
            (a_1 a_2)^2, \;
            b_1^6,
            && (a_2 b_1)^4, \\
            & (a_2 b_1 a_1 b_1)^2,
            && (a_1 b_1)^4,\\
            & a_2 b_1^{-1} a_1 b_1^2 a_2 b_1 a_1 b_1^{-2},
            && a_2 b_1^2 a_1 b_1^{-2} a_2 b_1 a_2 b_1^{-1}
            \big \rangle
        \end{align*}
        \vspace{-5mm}
        }
    \end{minipage}\\
    \hline
    $L_6$    
    &
    \begin{minipage}{6cm}
        {
        \vspace{-2.5mm}
        \begin{align*}
            \big \langle
            a_1, b_1, b_2
            \; \big | \;
            &
            a_1^4, \;
            b_1^2, \;
            b_2^3, \; 
            (b_1 b_2^{-1})^2,
            && (a_1 b_2^{-1} a_1 b_1)^2,\\
            & a_1 b_1 a_1^{-1} b_1 a_1 b_1 a_1^{-1} b_1 b_2^{-1},
            && a_1^{-1} b_2^{-1} a_1 b_2^{-1} a_1^{-1} b_1 a_1^{-1} b_1 b_2
            \big \rangle
        \end{align*}
        \vspace{-5mm}
        }
    \end{minipage}\\
    \hline
    $L_7$    
    &
    \begin{minipage}{6cm}
        {
        \vspace{-2.5mm}
        \begin{align*}
            \big \langle
            a_1, a_2, b_1
            \; \big | \;
            &
            a_1^2,
            a_2^2,
            (a_1 a_2)^2,
            b_1^6, 
            && (a_2 b_1 a_2 b_1^{-1})^2,\\
            & (a_1 b_1 a_2 b_1^{-1})^2,
            && (a_2 b_1 a_1 b_1^{-1})^2, \\
            & (a_1 b_1 a_1 b_1^{-1})^2,
            && a_1 b_1 a_2 b_1 a_1 b_1^2 a_2 b_1^2,\\
            & (a_1 b_1 a_1 b_1^2)^2,
            && a_2 b_1 a_2 b_1 a_1 a_2 b_1 a_1 b_1^3
            \big \rangle
        \end{align*}
        \vspace{-5mm}
        }
    \end{minipage}\\
    \hline
    $L_8$    
    &
    \begin{minipage}{6cm}
        {
        \vspace{-2.5mm}
        \begin{align*}
            \big \langle
            a_1, a_2, b_1
            \; \big | \;
            &
            a_1^2, \;
            a_2^2, \; 
            (a_1 a_2)^2, \;
            b_1^6,
            && (b_1 a_1 b_1^{-1} a_2)^2, \\
            & a_2 b_1 a_1 a_2 b_1 a_2 b_1^{-1} a_2 b_1^{-1},
            && a_1 b_1^2 a_2 b_1 a_2 b_1^2 a_1 b_1 
            \big \rangle
        \end{align*}
        \vspace{-5mm}
        }
    \end{minipage}\\
    \hline
    $L_9$    
    &
    \begin{minipage}{6cm}
        {
        \vspace{-2.5mm}
        \begin{align*}
            \big \langle
            a_1, b_1, b_2
            \; \big | \;
            &a_1^4, \;
            b_1^2, \;
            b_2^3, \;
            (b_1 b_2)^2, 
            && (a_1 b_2)^4,\\
            & a_1 b_2^{-1} a_1 b_1 a_1^{-1} b_2 a_1^{-1} b_1 b_2, 
            && a_1 b_1 a_1 b_1 b_2 a_1 b_2^{-1} a_1^{-1} b_2^{-1},\\
            & a_1 b_1 b_2 a_1 b_1 a_1 b_2 a_1^{-1} b_2,
            && a_1 b_2 a_1 b_1 a_1 b_1 a_1^2 b_2 b_1
            \big \rangle
        \end{align*}
        \vspace{-5mm}
        }
    \end{minipage}\\
    \hline
    $L_{10}$    
    &
    \begin{minipage}{6cm}
        {
        \vspace{-2.5mm}
        \begin{align*}
            \big \langle
            a_1, a_2, b_1, b_2
            \; \big | \;
            & a_1^2, \;
            a_2^2, \;
            (a_1 a_2)^2,
            && b_1^2, \;
            b_2^3, \;
            (b_1 b_2)^2, \\
            & (a_1 b_2^{-1} a_1 b_2)^2,
            && a_2 b_2^{-1} a_1 b_1 a_1 b_2 a_2 b_2 b_1, \\
            & a_2 a_1 b_2^{-1} a_1 b_1 a_2 b_2 a_1 b_1 b_2^{-1},
            && a_2 b_2^{-1} a_1 b_2 b_1 a_2 b_1 b_2^{-1} a_1 b_2
            \big \rangle
        \end{align*}
        \vspace{-5mm}
        }
    \end{minipage}\\
    \hline
    $L_{11}$    
    &
    \begin{minipage}{6cm}
        {
        \vspace{-2.5mm}
        \begin{align*}
            \big \langle
            a_1, a_2, b_1, b_2
            \; \big | \;
            & a_2^2, \;
            a_1^2, \;
            (a_1 a_2)^2,
            && b_1^2, \;
            b_2^3, \;
            (b_1 b_2)^2, \\
            &(a_1 b_1)^4,
            &&(a_2 b_2^{-1} a_1 b_2)^2,\\
            &(a_1 b_2^{-1} a_1 b_1)^2,
            &&(a_1 b_2^{-1} a_1 b_2)^2,\\
            &a_1 b_2 a_2 b_2 a_2 b_2 b_1 a_2 b_1,
            &&a_1 b_2 a_2 b_1 a_2 b_2^{-1} a_1 b_1 b_2
            \big \rangle
        \end{align*}
        \vspace{-5mm}
        }
    \end{minipage}\\
    \hline
    \end{longtable}

    \begin{longtable}{ |c|c| }
    \caption{Presentations of chamber-regular groups on $K_{4,4}$.}
    \label{tbl:K44}\\
    \hline
    \textbf{Grp} & \textbf{Presentation} \\
    \hline
    $L_{12}$
    &
    \begin{minipage}{6cm}
        {
        \vspace{-2.5mm}
        \begin{align*}
            \big \langle
            a_1,b_1 \; \big | \;
            a_1^4, \;
            b_1^4, \;
            a_1 b_1 a_1^{-1} b_1^{-1}
            \big \rangle
        \end{align*}
        \vspace{-5mm}
        }
    \end{minipage}\\
    \hline
    $L_{13}$
    &
    \begin{minipage}{6cm}
        {
        \vspace{-2.5mm}
        \begin{align*}
            \big \langle
            a_1,b_1 \; \big | \; &
            a_1^4, b_1^4, \;
            (a_1 b_1^{-1})^2, \;
            (a_1b_1)^2
            \big \rangle
        \end{align*}
        \vspace{-5mm}
        }
    \end{minipage}\\
    \hline
    $L_{14}$
    &
    \begin{minipage}{6cm}
        {
        \vspace{-2.5mm}
        \begin{align*}
            \big \langle
            a_1,b_1 \; \big | \;
            a_1^4, b_1^4, \;
            a_1 b_1 a_1 b_1^{-1}
            \big \rangle
        \end{align*}
        \vspace{-5mm}
        }
    \end{minipage}\\
    \hline
    $L_{15}$
    &
    \begin{minipage}{6cm}
        {
        \vspace{-2.5mm}
        \begin{align*}
            \big \langle
            a_1, a_2, b_1, b_2 \; \big | \;&
            a_1^2, \; a_2^2, \; (a_1 a_2)^2, && b_1^2, \; b_2^2, \; (b_1 b_2)^2, && (a_1 b_1)^2,\\
            & (a_1 b_2)^2, && (a_2 b_1)^2, && a_2 a_1 b_2 a_2 b_2
            \big \rangle
        \end{align*}
        \vspace{-5mm}
        }
    \end{minipage}\\
    \hline
    $L_{16}$
    &
    \begin{minipage}{6cm}
        {
        \vspace{-2.5mm}
        \begin{align*}
            \big \langle
            a_1, a_2, b_1, b_2 \; \big | \;
            & a_1^2, \; a_2^2, \; (a_1 a_2)^2, && b_1^2, \; b_2^2, \; (b_1 b_2)^2, && (a_1 b_1)^2,\\
            & (a_1 b_2)^2, && (a_2 b_1)^2, && a_2 a_1 b_2 a_2 b_2 b_1 
            \big \rangle
        \end{align*}
        \vspace{-5mm}
        }
    \end{minipage}\\
    \hline
    $L_{17}$
    &
    \begin{minipage}{6cm}
        {
        \vspace{-2.5mm}
        \begin{align*}
            \big \langle
            a_1, a_2, b_1, b_2 \; \big | &
            a_1^2, \; a_2^2, \; (a_1 a_2)^2, &&  b_1^2, \; b_2^2, \; (b_1 b_2)^2, && (a_1 b_1)^2, \\
            & (a_1 b_2)^2, && (a_2 b_1)^2, && (a_2 b_2)^2
            \big \rangle
        \end{align*}
        \vspace{-5mm}
        }
    \end{minipage}\\
    \hline
    $L_{18}$
    &
    \begin{minipage}{6cm}
        {
        \vspace{-2.5mm}
        \begin{align*}
            \big \langle
            a_1, b_1, b_2 \; \big | \;
            a_1^4, \; b_1^2, \; b_2^2, \; (b_1 b_2)^2, \; a_1 b_1 a_1^{-1} b_1, \; a_1 b_2 a_1^{-1} b_2 b_1
            \big \rangle
        \end{align*}
        \vspace{-5mm}
        }
    \end{minipage}\\
    \hline
    $L_{19}$
    &
    \begin{minipage}{6cm}
        {
        \vspace{-2.5mm}
        \begin{align*}
            \big \langle
            a_1, b_1, b_2 \; \big | \;
            a_1^4, \; b_1^2, \; b_2^2, \; (b_1 b_2)^2, \;
            a_1 b_1 a_1^{-1} b, \; (a_1b_2)^2 b_1,\; a_1^2 b_2 a_1^{-2} b_2
            \big \rangle
        \end{align*}
        \vspace{-5mm}
        }
    \end{minipage}\\
    \hline
    $L_{20}$
    &
    \begin{minipage}{6cm}
        {
        \vspace{-2.5mm}
        \begin{align*}
            \big \langle
            a_1, b_1, b_2 \; \big | \;
            a_1^4, \; b_1^2, \; b_2^2, \; (b_1 b_2)^2, \;
            a_1 b_2 a_1^{-1} b_2, \; a_1 b_1 a_1^{-1} b_1
            \big \rangle
        \end{align*}
        \vspace{-5mm}
        }
    \end{minipage}\\
    \hline
    $L_{21}$
    &
    \begin{minipage}{6cm}
        {
        \vspace{-2.5mm}
        \begin{align*}
            \big \langle
            a_1, b_1, b_2 \; \big | \;&
            a_1^4, \; b_1^2, \; b_2^2, \; (b_1 b_2)^2, \;
            a_1 b_1 a_1^{-1} b_1, \; (a_1 b_2)^2
            \big \rangle
        \end{align*}
        \vspace{-5mm}
        }
    \end{minipage}\\
    \hline
    \end{longtable}

    \begin{longtable}{ |c|c| }
    \caption{Presentations of chamber-regular groups on $K_{6,6}$.}
    \label{tbl:K66}\\
    \hline
    \textbf{Grp} & \textbf{Presentation} \\
    \hline

    $L_{22}$
    &
    \begin{minipage}{6cm}
        {
        \vspace{-2.5mm}
        \begin{align*}
            \big \langle
            a_1, b_1 \; \big | \;
            a_1^6, \;
            b_1^6, \;
            (a_1 b_1)^2,\;
            (a_1 b_1^{-1})^2\;
            \big \rangle
        \end{align*}
        \vspace{-5mm}
        }
    \end{minipage}\\
    \hline

    $L_{23}$
    &
    \begin{minipage}{6cm}
        {
        \vspace{-2.5mm}
        \begin{align*}
            \big \langle
            a_1, b_1 \; \big | \;
            a_1^6, \;
            b_1^6, \;
            a_1 b_1 a_1^{-1} b_1
            \big \rangle
        \end{align*}
        \vspace{-5mm}
        }
    \end{minipage}\\
    \hline

    $L_{24}$
    &
    \begin{minipage}{6cm}
        {
        \vspace{-2.5mm}
        \begin{align*}
            \big \langle
            a_1, b_1 \; \big | \;
            a_1^6, \;
            b_1^6, \;
            (a_1 b_1^{-1})^2, \;
            a_1^2 b_1^{-1} a_1^{-2} b_1 
            \big \rangle
        \end{align*}
        \vspace{-5mm}
        }
    \end{minipage}\\
    \hline

    $L_{25}$
    &
    \begin{minipage}{6cm}
        {
        \vspace{-2.5mm}
        \begin{align*}
            \big \langle
            a_1, b_1 \; \big | \;
            a_1^6, \;
            b_1^6, \;
            a_1 b_1 a_1^{-1} b_1^{-1}
            \big \rangle
        \end{align*}
        \vspace{-5mm}
        }
    \end{minipage}\\
    \hline

    $L_{26}$
    &
    \begin{minipage}{6cm}
        {
        \vspace{-2.5mm}
        \begin{align*}
            \big \langle
            a_1, a_2, b_1, b_2 \; \big | \;&
            a_1^2, \;
            a_2^3, \;
            (a_1 a_2)^2, && b_1^2, \;
            b_2^3, \; (b_1 b_2)^2, &&
            (a_1 b_1)^2, \\
            & a_1 b_2 a_1 b_2^{-1}, &&
            a_2 b_1 a_2^{-1} b_1, &&
            a_2 b_2 a_2^{-1} b_2^{-1}
            \big \rangle
        \end{align*}
        \vspace{-5mm}
        }
    \end{minipage}\\
    \hline

    $L_{27}$
    &
    \begin{minipage}{6cm}
        {
        \vspace{-2.5mm}
        \begin{align*}
            \big \langle
            a_1, a_2, b_1, b_2 \; \big | \;&
            a_1^2, \;
            a_2^3, \;
            (a_1 a_2)^2, &&
            b_1^2, \;
            b_2^3, \;
            (b_1 b_2)^2, &&
            (a_1 b_1)^2, \\
            & a_1 b_2 a_1 b_2^{-1}, &&
            a_2 b_2 a_2^{-1} b_2^{-1}, &&
            (a_2 b_1^{-1})^2
            \big \rangle
        \end{align*}
        \vspace{-5mm}
        }
    \end{minipage}\\
    \hline

    $L_{28}$
    &
    \begin{minipage}{6cm}
        {
        \vspace{-2.5mm}
        \begin{align*}
            \big \langle
            a_1, a_2, b_1, b_2 \; \big | \;&
            a_1^2, \;
            a_2^3, \;
            (a_1 a_2)^2, &&
            b_1^2, \;
            b_2^3, \;
            (b_1 b_2)^2, &&
            (a_1 b_1)^2, \\
            & a_1 a_2 b_2 a_1 b_2^{-1}, &&
            a_2 b_2 a_2^{-1} b_2^{-1}, &&
            (a_2 b_1^{-1})^2
            \big \rangle
        \end{align*}
        \vspace{-5mm}
        }
    \end{minipage}\\
    \hline

    $L_{29}$
    &
    \begin{minipage}{6cm}
        {
        \vspace{-2.5mm}
        \begin{align*}
            \big \langle
            a_1, a_2, b_1, b_2 \; \big | \;&
            a_1^2, \;
            a_2^3, \;
            (a_1 a_2)^2, && 
            b_1^2, \;
            b_2^3, \;
            (b_1 b_2)^2, &&
            (a_1 b_1)^2, \\
            & (a_1 b_2)^2, &&
            a_2 b_2 a_2^{-1} b_2^{-1}, &&
            (a_2 b_1^{-1})^2
            \big \rangle
        \end{align*}
        \vspace{-5mm}
        }
    \end{minipage}\\
    \hline

    $L_{30}$
    &
    \begin{minipage}{6cm}
        {
        \vspace{-2.5mm}
        \begin{align*}
            \big \langle
            a_1, b_1, b_2 \; \big | \;&
            a_1^6, \;
            b_1^2, \;
            b_2^3, \;
            (b_1 b_2)^2, \;
            a_1 b_2 a_1^{-1} b_2, \;
            (a_1 b_1)^2
            \big \rangle
        \end{align*}
        \vspace{-5mm}
        }
    \end{minipage}\\
    \hline

    $L_{31}$
    &
    \begin{minipage}{6cm}
        {
        \vspace{-2.5mm}
        \begin{align*}
            \big \langle
            a_1, b_1, b_2 \; \big | \;&
            a_1^6, \;
            b_1^2, \;
            b_2^3, \;
            (b_1 b_2)^2, \;
            a_1^3 b_2^{-1} a_1^{-1} b_2^{-1}, \;
            (a_1 b_1)^2, \;
            (a_1 b_2^{-1})^2 \;
            \big \rangle
        \end{align*}
        \vspace{-5mm}
        }
    \end{minipage}\\
    \hline

    $L_{32}$
    &
    \begin{minipage}{6cm}
        {
        \vspace{-2.5mm}
        \begin{align*}
            \big \langle
            a_1, b_1, b_2 \; \big | \;&
            a_1^6, \;
            b_1^2, \;
            b_2^3, \;
            (b_1 b_2)^2, \;
            a_1 b_1 a_1^{-1} b_1, \;
            a_1 b_2^{-1} a_1^{-1} b_2
            \big \rangle
        \end{align*}
        \vspace{-5mm}
        }
    \end{minipage}\\
    \hline

    $L_{33}$
    &
    \begin{minipage}{6cm}
        {
        \vspace{-2.5mm}
        \begin{align*}
            \big \langle
            a_1, b_1, b_2 \; \big | \;&
            a_1^6, \;
            b_1^2, \;
            b_2^3, \;
            (b_1 b_2)^2, \;
            a_1 b_1 a_1^{-1} b_2^{-1} b_1, \;
            a_1 b_2^{-1} a_1^{-1} b_2
            \big \rangle
        \end{align*}
        \vspace{-5mm}
        }
    \end{minipage}\\
    \hline

    $L_{34}$
    &
    \begin{minipage}{6cm}
        {
        \vspace{-2.5mm}
        \begin{align*}
            \big \langle
            a_1, b_1, b_2 \; \big | \;&
            a_1^6, \;
            b_1^2, \;
            b_2^3, \;
            (b_1 b_2)^2, \;
            a_1 b_1 a_1^{-1} b_1, \;
            a_1 b_2^{-1} a_1^{-1} b_2^{-1}
            \big \rangle
        \end{align*}
        \vspace{-5mm}
        }
    \end{minipage}\\
    \hline

    $L_{35}$
    &
    \begin{minipage}{6cm}
        {
        \vspace{-2.5mm}
        \begin{align*}
            \big \langle
            a_1, b_1, b_2 \; \big | \;&
            a_1^6, \;
            b_1^2, \;
            b_2^3, \;
            (b_1 b_2)^2, \;
            (a_1 b_1^{-1})^2, \;
            a_1 b_2 a_1^{-1} b_2^{-1}
            \big \rangle
        \end{align*}
        \vspace{-5mm}
        }
    \end{minipage}\\
    \hline

    \end{longtable}

    \section{Actions on the edge groups}\label{app:more_data}    

    As in Section \ref{sec:classification} we identify the local automorphism group of $(L_r, F_A^r, F_B^r)$ with a subgroup of $\Aut(E_A) \times \Aut(E_B)$, where $E_A,E_B$ are the isomorphic model edge groups.
    In the following tables we list these groups and denote them by $\Sigma^r$.
    We also introduce some notation for automorphisms of the edge groups. Let $E_A$ be a model edge group. If $E_A$ is cyclic, we denote the automorphism, that maps $a\in E_A$ to $a^{-1}$ by $\rho_A$.
    If $E_A=C_2\times C_2$ we will indicate the action of automorphism on $E_A$ by cycle notation. If $E_A = \Sym(3)$ we indicate automorphisms by inner automorphisms, and we denote the conjugation with $a\in E_A$ by $\Ad(a)$.

    \begin{longtable}{ |c|c|c|c|c|}
    \caption{$\Sigma^r$ for the groups $L_1\, \dots, L_{11}$.}
    \label{tbl:sigmasQ}\\
    \hline
    $\boldsymbol{r \in}$
    & $\boldsymbol{E_A}$
    & $\boldsymbol{E_B}$
    & $\boldsymbol{\Sigma^r}$
    &$|\boldsymbol{\Sigma^r}|$ \\
    \hline
    $\{1\}$
    &
    $C_4$
    &
    $C_6$
    &
    \begin{minipage}{3.5cm}
        {
        \vspace{-2.5mm}
        \begin{align*}
            1
        \end{align*}
        \vspace{-5mm}
        }
    \end{minipage}
    & 1 \\
    \hline
    $\{2,4,10,11\}$
    &
    $C_2\times C_2$
    &
    $\Sym(3)$
    &
    \begin{minipage}{3.5cm}
        {
        \vspace{-2.5mm}
        \begin{align*}
                1 \times \langle \Ad(b_1) \rangle
        \end{align*}
        \vspace{-5mm}
        }
    \end{minipage}
    & 2\\
    \hline
    $\{3,5,7,8\}$
    &
    $C_2\times C_2$
    &
    $C_6$
    &
    \begin{minipage}{3.5cm}
        {
        \vspace{-2.5mm}
        \begin{align*}
            1 \times \Aut(C_6)
        \end{align*}
        \vspace{-5mm}
        }
    \end{minipage}
    & 2\\
    \hline
    $\{6,9\}$
    &
    $C_4$
    &
    $\Sym(3)$
    &
    \begin{minipage}{3.5cm}
        {
        \vspace{-2.5mm}
        \begin{align*}
            1
        \end{align*}
        \vspace{-5mm}
        }
    \end{minipage}
    & 1\\
    \hline
    \end{longtable}

    \begin{longtable}{ |c|c|c|c|c|}
    \caption{$\Sigma^r$ for the groups $L_{12}, \dots, L_{21}$.}
    \label{tbl:sigmasK44}\\
    \hline
    $\boldsymbol{r\in}$ & $\boldsymbol{E_A}$ & $\boldsymbol{E_B}$ & $\boldsymbol{\Sigma^r}$ & $\boldsymbol{|\Sigma^r|}$\\
    \hline
    $\{12,13,14\}$
    &
    $C_4$
    &
    $C_4$
    &
    \begin{minipage}{3.5cm}
        {
        \vspace{-2.5mm}
        \begin{align*}
            \Aut(C_4) \times \Aut(C_4)
        \end{align*}
        \vspace{-5mm}
        }
    \end{minipage}
    & 4\\
    \hline
    $\{15,16\}$
    &
    $C_2 \times C_2$
    &
    $C_2 \times C_2$
    &
    \begin{minipage}{3.5cm}
        {
        \vspace{-2.5mm}
        \begin{align*}
            \langle (a_2, a_1 a_2) \rangle
            \times
            \langle (b_2, b_1 b_2) \rangle
        \end{align*}
        \vspace{-5mm}
        }
    \end{minipage}
    & 4 \\
    \hline
    $\{17\}$
    &
    $C_2 \times C_2$
    &
    $C_2 \times C_2$
    &
    \begin{minipage}{3.5cm}
        {
        \vspace{-2.5mm}
        \begin{align*}
            \Aut(C_2\times C_2) \times \Aut(C_2\times C_2)
        \end{align*}
        \vspace{-5mm}
        }
    \end{minipage}
    &36\\
    \hline
    $\{18,19,21\}$
    &
    $C_4$
    &
    $C_2\times C_2$
    &
    \begin{minipage}{3.5cm}
        {
        \vspace{-2.5mm}
        \begin{align*}
            \Aut(C_4) \times
            \left\langle
            (b_2, b_1b_2)
            \right\rangle
        \end{align*}
        \vspace{-5mm}
        }
    \end{minipage}
    & 4 \\
    \hline
    $\{20\}$
    &
    $C_4$
    &
    $C_2\times C_2$
    &
    \begin{minipage}{3.5cm}
        {
        \vspace{-2.5mm}
        \begin{align*}
            \Aut(C_4) \times \Aut(C_2\times C_2)
        \end{align*}
        \vspace{-5mm}
        }
    \end{minipage}
    & 12\\
    \hline
    \end{longtable}

    \begin{longtable}{ |c|c|c|c|c|}
    \caption{$\Sigma^r$ for the groups $L_{22}, \dots, L_{35}$.}
    \label{tbl:sigmasK66}\\
    \hline
    $\boldsymbol{r\in}$ & $\boldsymbol{E_A}$ & $\boldsymbol{E_B}$ & $\boldsymbol{\Sigma^r}$ & $\boldsymbol{|\Sigma^r|}$\\
    \hline
    $\{22,23,25\}$
    &
    $C_6$
    &
    $C_6$
    &
    \begin{minipage}{3.5cm}
        {
        \vspace{-2.5mm}
        \begin{align*}
            \Aut(C_6) \times \Aut(C_6)
        \end{align*}
        \vspace{-5mm}
        }
    \end{minipage}
    &4\\
    \hline
    $\{24\}$
    &
    $C_6$
    &
    $C_6$
    &
    \begin{minipage}{3.5cm}
        {
        \vspace{-2.5mm}
        \begin{align*}
            \langle \left(\rho_A, \rho_B \right) \rangle
        \end{align*}
        \vspace{-5mm}
        }
    \end{minipage}
    &2\\
    \hline
    $\{26\}$
    &
    $\Sym(3)$
    &
    $\Sym(3)$
    &
    \begin{minipage}{3.5cm}
        {
        \vspace{-2.5mm}
        \begin{align*}
            \Aut(\Sym(3)) \times \Aut(\Sym(3))
        \end{align*}
        \vspace{-5mm}
        }
    \end{minipage}
    &36\\
    \hline
    $\{27\}$
    &
    $\Sym(3)$
    &
    $\Sym(3)$
    &
    \begin{minipage}{3.5cm}
        {
        \vspace{-2.5mm}
        \begin{align*}
             \langle \Ad(a_1) \rangle \times \Aut(\Sym(3))
        \end{align*}
        \vspace{-5mm}
        }
    \end{minipage}
    &12\\
    \hline
    $\{28\}$
    &
    $\Sym(3)$
    &
    $\Sym(3)$
    &
    \begin{minipage}{3.5cm}
        {
        \vspace{-2.5mm}
        \begin{align*}
            \langle \left(\Ad(a_1), \Ad(b_1) \right), \left(\Ad(a_2), \Ad(b_2) \right) \rangle
        \end{align*}
        \vspace{-5mm}
        }
    \end{minipage}
    &6\\
    \hline
    $\{29\}$
    &
    $\Sym(3)$
    &
    $\Sym(3)$
    &
    \begin{minipage}{3.5cm}
        {
        \vspace{-2.5mm}
        \begin{align*}
            \langle \Ad(a_1) \rangle \times \langle \Ad(b_1) \rangle
        \end{align*}
        \vspace{-5mm}
        }
    \end{minipage}
    &4\\
    \hline
    $\{30,34\}$
    &
    $C_6$
    &
    $\Sym(3)$
    &
    \begin{minipage}{3.5cm}
        {
        \vspace{-2.5mm}
        \begin{align*}
            \Aut(C_6) \times \langle \Ad(b_1) \rangle
        \end{align*}
        \vspace{-5mm}
        }
    \end{minipage}
    &4\\
    \hline
    $\{31\}$
    &
    $C_6$
    &
    $\Sym(3)$
    &
    \begin{minipage}{3.5cm}
        {
        \vspace{-2.5mm}
        \begin{align*}
            \langle \left(\rho_A, \Ad(b_1) \right) \rangle
        \end{align*}
        \vspace{-5mm}
        }
    \end{minipage}
    &2\\
    \hline
    $\{32, 35\}$
    &
    $C_6$
    &
    $\Sym(3)$
    &
    \begin{minipage}{3.5cm}
        {
        \vspace{-2.5mm}
        \begin{align*}
            \Aut(C_6) \times \Aut(\Sym(3))
        \end{align*}
        \vspace{-5mm}
        }
    \end{minipage}
    &12\\
    \hline
    $\{33\}$
    &
    $C_6$
    &
    $\Sym(3)$
    &
    \begin{minipage}{3.5cm}
        {
        \vspace{-2.5mm}
        \begin{align*}
            \langle \left(\rho_A, \Ad(b_1) \right), \left(\id_A, \Ad(b_2) \right) \rangle
        \end{align*}
        \vspace{-5mm}
        }
    \end{minipage}
    &6\\
    \hline
    \end{longtable}

    The following table provides convenient data for groups acting regularly on complete bipartite graphs.

    \begin{longtable}{|c|c|c|}
    \caption{Description of the groups $L_{12}, \dots, L_{35}$.}
    \label{tbl:descriptionL12L35}\\
    \hline
    Grps & Description & Small Group ID \\
    \hline
    $L_{12}$ & $C_4 \times C_4$ & (16,2) \\
    \hline
    $L_{13} \cong L_{18} \cong L_{19}$ & $(C_4 \times C_2) \rtimes C_2$ & (16,3) \\
    \hline
    $L_{14}$ & $C_4 \rtimes C_4$ & (16,4)\\
    \hline
    $L_{15} \cong L_{16} \cong L_{21}$ & $C_2 \times D_8$ & (16,11) \\
    \hline
    $L_{17}$ & $C_2 \times C_2 \times C_2 \times C_2$ & (16,14) \\
    \hline
    $L_{20}$ & $C_4 \times C_2 \times C_2$ & (16,10)\\
    \hline
    $L_{22} \cong L_{26} \cong L_{27} \cong L_{28} \cong L_{30} \cong L_{31}$ & $\Sym(3)\times \Sym(3)$ & (36,10) \\
    \hline
    $L_{23} \cong L_{24} \cong L_{32} \cong L_{33} \cong L_{34}$ & $C_6 \times \Sym(3)$ & (36,12)\\
    \hline
    $L_{25}$ & $C_6\times C_6$ & (36,14) \\
    \hline
    $L_{29} \cong L_{35}$ & $C_2\times ((C_3 \times C_3) \rtimes C_2)$ & (36,13)\\
    \hline
    \end{longtable}

    \printbibliography

@book{BridsonHaefliger,
  author    = {Martin R. Bridson and Andr{\'e} Haefliger},
  title     = {Metric Spaces of Non-Positive Curvature},
  series    = {Grundlehren der mathematischen Wissenschaften},
  volume    = {319},
  publisher = {Springer-Verlag},
  year      = {1999},
  isbn      = {3-540-64324-9},
  doi       = {10.1007/978-3-662-12494-9}
}

@book{Brown,
  author    = {Kenneth S. Brown},
  title     = {Buildings},
  series = {Graduate Texts in Mathematics},
  volume = {248},
  publisher = {Springer},
  year      = {1988},
  isbn      = {978-0-387-96876-6},
  doi       = {10.1007/978-1-4612-1019-1}
}

@book{PayneThas,
  author    = {S. E. Payne and J. A. Thas},
  title     = {Finite Generalized Quadrangles},
  series    = {EMS Series of Lectures in Mathematics},
  publisher = {European Mathematical Society},
  year      = {2009},
  isbn      = {9783037190661},
}

@book{VanMaldeghem,
  author    = {H. Van Maldeghem},
  title     = {Generalized Polygons},
  publisher = {Birkh{\"a}user},
  year      = {1998},
  edition   = {1},
  isbn      = {978-3-0348-0271-0},
  doi       = {10.1007/978-3-0348-0271-0}
}

@book{Weiss,
  author    = {Richard M. Weiss},
  title     = {The Structure of Affine Buildings},
  series    = {Annals of Mathematics Studies},
  volume    = {168},
  publisher = {Princeton University Press},
  year      = {2009},
  isbn      = {978-0-691-13881-7}
}

@book{TitsWeiss_polygons,
  author    = {Tits, Jacques and Weiss, Richard M.},
  title     = {Moufang Polygons},
  series    = {Springer Monographs in Mathematics},
  publisher = {Springer},
  year      = {2002},
  edition   = {1},
  isbn      = {978-3-540-43714-7},
  doi       = {10.1007/978-3-662-04689-0},
}

@book{AbramenkoBrown,
  author    = {Abramenko, Peter and Brown, Kenneth S.},
  title     = {Buildings: Theory and Applications},
  series    = {Graduate Texts in Mathematics},
  volume    = {248},
  publisher = {Springer},
  year      = {2008},
  isbn      = {978-0-387-78834-0},
  doi       = {10.1007/978-0-387-78835-7},
}

@book{Rousseau,
  author    = {Guy Rousseau},
  title     = {Euclidean Buildings: Geometry and Group Actions},
  publisher = {EMS Press},
  series    = {EMS Tracts in Mathematics},
  volume    = {35},
  year      = {2023},
  doi       = {10.4171/ETM/35}
}

@incollection{Tits,
  author    = {J. Tits},
  title     = {A Local Approach to Buildings},
  booktitle = {The Geometric Vein},
  publisher = {Springer},
  year      = {1981},
  pages     = {519--547},
  isbn      = {978-1-4612-5648-9}
}

@incollection{Kantor,
  author    = {W. M. Kantor},
  title     = {Automorphism Groups of Some Generalized Quadrangles},
  booktitle = {Advances in Finite Geometries and Designs: Proceedings of the Third Isle of Thorns Conference 1990},
  publisher = {Oxford University Press},
  year      = {1991},
  pages     = {251--256},
  isbn      = {978Kazhdan’s property (T) and to be hereditarily just-infinite (after factoring out0198535928}
}

@incollection{Stallings,
  author    = {John R. Stallings},
  title     = {Non-positively curved triangles of groups},
  booktitle = {Group Theory from a Geometrical Viewpoint},
  publisher = {World Scientific},
  year      = {1991},
  pages     = {491--503}
}

@incollection{KirschmerNebe,
  author    = {Markus Kirschmer and Gabriele Nebe},
  title     = {One Class Genera of Lattice Chains Over Number Fields},
  booktitle = {Algorithmic and Experimental Methods in Algebra, Geometry, and Number Theory},
  publisher = {Springer},
  year      = {2017},
  pages     = {503--532}
}

@incollection{Tits_classification,
  author    = {J. Tits},
  title     = {Immeubles de type affine},
  booktitle = {Buildings and the Geometry of Diagrams},
  series    = {Lecture Notes in Mathematics},
  volume    = {1181},
  publisher = {Springer},
  year      = {1986},
  pages     = {159--190}
}

@article{StallingsGersten,
  author  = {John R. Stallings and S. M. Gersten},
  title   = {Casson's Idea about 3-Manifolds whose Universal Cover is $R^3$},
  journal = {International Journal of Algebra and Computation},
  year    = {1991},
  volume  = {1},
  pages   = {395--406}
}

@article{Kazhdan,
  author  = {Kazhdan, David},
  title   = {Connection of the Dual Space of a Group with the Structure of Its Closed Subgroups},
  journal = {Functional Analysis and Its Applications},
  year    = {1967},
  volume  = {1},
  pages   = {63--65},
  doi     = {10.1007/BF01075866},
}

@article{Seitz,
  author  = {G. M. Seitz},
  title   = {Flag-Transitive Subgroups of Chevalley Groups},
  journal = {Annals of Mathematics},
  volume  = {97},
  number  = {1},
  year    = {1973},
  pages   = {27--56},
  doi     = {10.2307/1970876}
}

@article{Oppenheim,
  author  = {Izhar Oppenheim},
  title   = {Property (T) for Groups Acting on Affine Buildings},
  journal = {Bulletin of the London Mathematical Society},
  volume  = {57},
  number  = {10},
  year    = {2025},
  pages   = {3151--3162},
  doi     = {10.1112/blms.70148}
}

@article{BaderCapraceLecureux,
  author  = {Bader, Uri and Caprace, Pierre-Emmanuel and L{\'e}cureux, Jean},
  title   = {On the Linearity of Lattices in Affine Buildings and Ergodicity of the Singular Cartan Flow},
  journal = {Journal of the American Mathematical Society},
  year    = {2019},
  volume  = {32},
  number  = {2},
  pages   = {491--562},
  doi     = {10.1090/jams/914},
}

@article{GrundhoferJoswigStroppel,
  author  = {Theo Grundh{\"o}fer and Michael Joswig and Markus Stroppel},
  title   = {Slanted symplectic quadrangles},
  journal = {Geometriae Dedicata},
  volume  = {49},
  pages   = {143--154},
  year    = {1994},
  doi     = {10.1007/BF01610617}
}

@article{KramerWeiss14,
  author  = {Linus Kramer and Richard M. Weiss},
  title   = {Coarse Equivalences of Euclidean Buildings},
  journal = {Advances in Mathematics},
  volume  = {253},
  year    = {2014},
  pages   = {1--49},
  doi     = {10.1016/j.aim.2013.10.031},
  note    = {With an appendix by Jeroen Schillewaert and Koen Struyve}
}

@article{BruhatTits1,
  author  = {Fran{\c{c}}ois Bruhat and Jacques Tits},
  title   = {Groupes r{\'e}ductifs sur un corps local: {I}. Donn{\'e}es radicielles valu{\'e}es},
  journal = {Publications Math{\'e}matiques de l'IH{\'E}S},
  volume  = {41},
  year    = {1972},
  pages   = {5--251}
}

@article{BruhatTits2,
  author  = {Fran{\c{c}}ois Bruhat and Jacques Tits},
  title   = {Groupes r{\'e}ductifs sur un corps local: {II}. Sch{\'e}mas en groupes. Existence d'une donn{\'e}e radicielle valu{\'e}e},
  journal = {Publications Math{\'e}matiques de l'IH{\'E}S},
  volume  = {60},
  year    = {1984},
  pages   = {5--184}
}

@article{KantorLieblerTits,
  author  = {W. M. Kantor and R. A. Liebler and J. Tits},
  title   = {On Discrete Chamber-Transitive Automorphism Groups of Affine Buildings},
  journal = {Bulletin of the American Mathematical Society},
  volume  = {16},
  number  = {1},
  year    = {1987},
  pages   = {129--133}
}

@article{Timmesfeld_large_thickness,
  author  = {F. G. Timmesfeld},
  title   = {Locally Finite Classical Tits Chamber Systems of Large Order},
  journal = {Inventiones Mathematicae},
  volume  = {87},
  year    = {1987},
  pages   = {603--641},
  doi     = {10.1007/BF01389245}
}

@article{Timmesfeld_Rank3,
  author  = {F. G. Timmesfeld},
  title   = {Classical Locally Finite Tits Chamber Systems of Rank 3},
  journal = {Journal of Algebra},
  volume  = {124},
  year    = {1989},
  pages   = {9--59}
}

@article{Zuk,
  author  = {{\.Z}uk, Andrzej},
  title   = {La Propri{\'e}t{\'e} (T) de {K}azhdan pour les Groupes Agissant sur les Poly{\`e}dres},
  journal = {Comptes Rendus de l'Acad{\'e}mie des Sciences, S{\'e}rie I Math{\'e}matiques},
  year    = {1996},
  volume  = {323},
  pages   = {453--458},
}

@article{MargulisNormalSubgroups,
  author  = {Margulis, Grigory A.},
  title   = {Finiteness of Quotient Groups of Discrete Subgroups},
  journal = {Functional Analysis and Its Applications},
  year    = {1979},
  volume  = {13},
  number  = {3},
  pages   = {178--187},
  doi     = {10.1007/BF01077485},
}

@article{FeitHigman,
  author  = {Walter Feit and Graham Higman},
  title   = {The Nonexistence of Certain Generalized Polygons},
  journal = {Journal of Algebra},
  volume  = {1},
  year    = {1964},
  pages   = {114--131}
}

@article{BaderFurmanLecureux,
  author  = {Uri Bader and Alex Furman and Jean L{\'e}cureux},
  title   = {Normal Subgroup Theorem for groups acting on {$\widetilde{A}_2$}-buildings},
  journal = {Journal of the European Mathematical Society},
  year    = {2026},
  doi     = {10.4171/JEMS/1747},
}

@misc{DixmierZara,
  author = {S. Dixmier and F. Zara},
  title  = {{\'E}tude d'un quadrangle g{\'e}n{\'e}ralis{\'e} autour de deux de ses points non li{\'e}s},
  note   = {Unpublished manuscript},
  year   = {1976}
}

@misc{LecureuxWitzel,
  author        = {Jean L{\'e}cureux and Stefan Witzel},
  title         = {The Normal Subgroup Theorem for lattices on two-dimensional Euclidean buildings},
  year          = {2026},
  eprint        = {2605.06163},
  archivePrefix = {arXiv},
  primaryClass  = {math.GR},
}

@misc{TitzMiteWitzel,
  author        = {Titz Mite, Thomas and Witzel, Stefan},
  title         = {Non-residually finite $\tilde{C}_2$-lattices},
  year          = {2025},
  eprint        = {2509.05054},
  archivePrefix = {arXiv},
  primaryClass  = {math.GR},
  doi           = {10.48550/arXiv.2509.05054}
}

@misc{ChamberRegularCode,
  author       = {Titz Mite, Thomas},
  title        = {{ChamerRegularLattices}},
  year         = {2026},
  howpublished = {\url{https://github.com/TitzMite/ChamerRegularLattices}},
  note         = {GitHub repository, accessed on 29 April 2026}
}

@manual{GAP4,
    organization = "The GAP~Group",
    title        = "{GAP -- Groups, Algorithms, and Programming,
                    Version 4.15.0}",
    year         = 2025,
    url          = "\url{https://www.gap-system.org}",
}
    
\end{document}